\documentclass[11pt]{amsart}
\usepackage[margin=1in]{geometry}

\usepackage{amssymb}
\usepackage{amsthm}
\usepackage{amsmath}
\usepackage{mathrsfs}
\usepackage{amsbsy}
\usepackage[all]{xy}
\usepackage{bm}
\usepackage{hyperref}
\usepackage{tikz}
\usetikzlibrary{decorations.pathreplacing}
\usepackage{array}
\usepackage{float}
\usepackage{enumerate}
\usepackage{xcolor}
\usepackage{hhline}
\allowdisplaybreaks
\usepackage[noadjust]{cite}

\usepackage{caption}
\usepackage{subcaption}
\usepackage{tabu}
\usepackage{diagbox}
\usepackage{tikz}
\usepackage{bbm}
\usepackage{booktabs}

\DeclareFontFamily{U}{rcjhbltx}{}
\DeclareFontShape{U}{rcjhbltx}{m}{n}{<->rcjhbltx}{}
\DeclareSymbolFont{hebrewletters}{U}{rcjhbltx}{m}{n}

\let\aleph\relax\let\beth\relax

\DeclareMathSymbol{\aleph}{\mathord}{hebrewletters}{39}
\DeclareMathSymbol{\beth}{\mathord}{hebrewletters}{98}

\usepackage[noabbrev,capitalise]{cleveref}

\newenvironment{enumerate*}%
  {\begin{enumerate}[(I)]%
    \setlength{\itemsep}{10pt}%
    \setlength{\parskip}{0pt}}%
  {\end{enumerate}}

\newtheorem{theorem}{Theorem}[section]
\newtheorem{proposition}[theorem]{Proposition}
\newtheorem{corollary}[theorem]{Corollary}

\newtheorem{lemma}[theorem]{Lemma}

\theoremstyle{definition}

\DeclareMathOperator{\vol}{vol}

\DeclareMathOperator{\Prog}{Prog}
\DeclareMathOperator{\Approx}{Approx}
\DeclareMathOperator{\lcm}{lcm}
\DeclareMathOperator{\acc}{acc}

\author[Vanshika Jain]{Vanshika Jain}
\address[]{Department of Mathematics, Princeton University, Princeton, NJ 08540, USA}
\email{vanshika@princeton.edu}

\author[Noah Kravitz]{Noah Kravitz}
\address[]{Department of Mathematics, Princeton University, Princeton, NJ 08540, USA}
\email{nkravitz@princeton.edu}

\title{Relative Lonely Runner spectra}

\begin{document}
\maketitle

\begin{abstract}
For a subtorus $T \subseteq (\mathbb{R}/\mathbb{Z})^n$, let $D(T)$ denote the $L^\infty$-distance from $T$ to the point $(1/2, \ldots, 1/2)$.  For a subtorus $U \subseteq (\mathbb{R}/\mathbb{Z})^n$, define $\mathcal{S}_1(U)$, the \emph{Lonely Runner spectrum relative to $U$}, to be the set of all values of $D(T)$ as $T$ ranges over the $1$-dimensional subtori of $U$ not contained in the union of the coordinate hyperplanes of $(\mathbb{R}/\mathbb{Z})^n$.  The relative spectrum $\mathcal{S}_1((\mathbb{R}/\mathbb{Z})^n)$ is the ordinary Lonely Runner spectrum that has been studied previously.

Giri and the second author recently showed that the relative spectra $\mathcal{S}_1(U)$ for $2$-dimensional subtori $U \subseteq (\mathbb{R}/\mathbb{Z})^n$ essentially govern the accumulation points of the Lonely Runner spectrum $\mathcal{S}_1((\mathbb{R}/\mathbb{Z})^n)$.  In the present work, we prove that such relative spectra $\mathcal{S}_1(U)$ have a very rigid arithmetic structure, and that one can explicitly find a complete characterization of each such relative spectrum with a finite calculation; carrying out this calculation for a few specific examples sheds light on previous constructions in the literature on the Lonely Runner Problem.
\end{abstract}

\section{Introduction}

\subsection{The Lonely Runner Conjecture}
The Lonely Runner Conjecture of Wills \cite{wills} and Cusick \cite{cusick} originally arose as a possible converse to Dirichlet's theorem on Diophantine approximation.  There are now several equivalent formulations of the Lonely Runner Problem.  In the most popular version, due to Bienia, Goddyn, Gvozdjak, Seb\H{o}, and Tarsi \cite{BGGST}, we have $n$ runners standing at the start line of a unit-length circular track.  The runners simultaneously begin running around the track, each at a constant nonzero integer speed.  The \emph{maximum loneliness} of a set of speeds is defined to be the largest $L$ such that there is a time when all of the runners are a distance at least $L$ from the start line.  The Lonely Runner Conjecture asserts that the maximum loneliness is always at least $1/(n+1)$, regardless of the choice of the $n$ speeds.  (Equality holds for the speeds $1,2,\ldots, n$, for instance.)  The combined work of many authors over the last several decades has established the Lonely Runner Conjecture for $n \leq 6$ (see \cite{BW, CP, BGGST, bohman, Renault, BS}), but it remains open for $n \geq 7$.

In this paper, we will work with the following \emph{view-obstruction} formulation of the Lonely Runner Problem, as introduced by Cusick \cite{cusick}.  For a subtorus $T \subseteq (\mathbb{R}/\mathbb{Z})^n$, let $D(T)$ denote the $L^\infty$-distance from $T$ to the point $(1/2, \ldots, 1/2)$.  For $1 \leq k \leq n$, let $\mathcal{S}_k(n)$ denote the set of all of the values of $D(T)$ as $T$ ranges over the $k$-dimensional proper subtori of $(\mathbb{R}/\mathbb{Z})^n$; here we say that a subtorus is \emph{proper} if it is not contained in the union of the coordinate hyperplanes of $(\mathbb{R}/\mathbb{Z})^n$.  The set $\mathcal{S}_1(n)$ is called the \emph{$n$-th Lonely Runner spectrum}. Note that each $\mathcal{S}_k(n)$ is a subset of $[0,1/2)$.

It is easy to check that if $v_1, \ldots, v_n$ are nonzero integers, then the maximum loneliness of this set of speeds equals $1/2-D(\langle (v_1,\ldots, v_n) \rangle_{\mathbb{R}})$.  The condition on the $v_i$'s being all nonzero is equivalent to the condition on the torus $\langle (v_1,\ldots, v_n) \rangle_{\mathbb{R}}$ being proper.  Hence, in the language of view-obstruction, the Lonely Runner Conjecture asserts that $$\max \mathcal{S}_1(n)=1/2-1/(n+1).$$  
The equality case mentioned at the end of the first paragraph gives $D(\langle (1,2,\ldots, n) \rangle_\mathbb{R})=1/2-1/(n+1) \in \mathcal{S}_1(n)$, and the difficult part of the Lonely Runner Conjecture is showing that for every $1$-dimensional proper subtorus $T \subseteq (\mathbb{R}/\mathbb{Z})^n$, we have $D(T) \leq 1/2-1/(n+1)$.

Until recently, work on the Lonely Runner Conjecture focused narrowly on the maximum value of each Lonely Runner spectrum $\mathcal{S}_1(n)$.  In 2020, the second author \cite{thesis} proposed studying the entire set $\mathcal{S}_1(n)$: For instance, one can try to characterize not only its maximum value but all of its near-maximal values, and one can study other properties of $\mathcal{S}_1(n)$, such as its accumulation points.  With regard to the first question, the paper \cite{thesis} conjectured that
\begin{equation}\label{eq:old-conj}
\mathcal{S}_1(n) \cap \left(\frac{1}{2}-\frac{1}{n},\frac{1}{2}\right]=\left\{\frac{1}{2}-\frac{1}{n}+\frac{1}{n^2 s+n}: s \in \mathbb{N}\right\}
\end{equation}
(where for us $\mathbb{N}=\{1,2,3,\ldots\}$).  The $D$-values on the right-hand side are achieved by the subtori $\langle (1,2,\ldots, n-1, ns) \rangle_\mathbb{R}$.  The appeal of this conjecture is that it provides a ``good'' explanation for the appearance of the quantity $1/2-1/(n+1)$ in the Lonely Runner Conjecture: This quantity is the endpoint (corresponding to $s=1$) of a highly arithmetically structured sequence of $D$-values.  One of the main results of \cite{thesis} was a proof of this conjecture for $n=2,3$.  Fan and Sun \cite{computer} later discovered an infinite family of counterexamples for $n=4$, namely, 
\begin{equation}\label{eq:alec-alex}
D(\langle (8, 4s+3, 4s+11, 4s+19) \rangle_\mathbb{R})=\frac{1}{2}-\frac{1}{4}+\frac{1}{16s+60}.
\end{equation}
These elements of $\mathcal{S}_1(4)$ are different from the values in \eqref{eq:old-conj} because $60 \not\equiv 4 \pmod{16}$.  Nonetheless, this sequence of large elements of $\mathcal{S}_1(4)$ is highly arithmetically structured, just like the sequence in \eqref{eq:old-conj}.  The main result of the present paper will show that this is no coincidence.

\subsection{Relative Lonely Runner spectra}
We now discuss the qualitative structure of Lonely Runner spectra and the role of higher-dimensional subtori.  For a set $S$ of real numbers, let $\acc(S)$ denote the set of accumulation points of $S$.  One can distinguish between accumulation points from above and from below, denoted $\acc^+(S), \acc^-(S)$ (respectively)  In \cite{thesis}, the second author showed that
\begin{equation}\label{eq:hierarchy}
\mathcal{S}_1(n-1) \subseteq \acc^+(\mathcal{S}_1(n))
\end{equation}
and then asked whether equality holds and whether $\acc^-(\mathcal{S}_1(n))=\emptyset$.  Giri and the second author \cite{vikram} recently made some progress towards understanding these questions.  One of their main results is that indeed $\acc^-(\mathcal{S}_1(n))=\emptyset$.  They also showed that
\begin{equation}\label{eq:GK-containment}
\acc(\mathcal{S}_1(n)) \subseteq \mathcal{S}_2(n),
\end{equation}
and that equality holds if one interprets $\mathcal{S}_1(n)$ suitably as a multiset.  They also conjecture that
$$\mathcal{S}_2(n)=\mathcal{S}_1(n-1),$$
which would imply that equality holds in \eqref{eq:hierarchy}.  Easy calculations verify this conjecture for $n=2,3$; see the discussion in \cite{vikram}.

The thrust of the argument for \eqref{eq:GK-containment} is that if $T_1, T_2, \ldots$ is a sequence of $1$-dimensional proper subtori with $D(T_i) \to d$, then, once we pass to a subsequence of the $T_i$'s, there is proper subtorus $U$ of dimension at least $2$ such that $D(U)=d$, the $T_i$'s are all contained in $U$, and $T_i$'s become $o(1)$-dense in $U$ as $i \to \infty$.  Since $D(T_i) \geq D(U)$, we see that $D(T_i)$ approaches $d$ from above, and it follows that $\acc^-(\mathcal{S}_1(n))=\emptyset$.  Only $D$-values of $2$-dimensional subtori appear on the right-hand side of \eqref{eq:GK-containment} because of the trivial containments
$$\mathcal{S}_n(n) \subseteq \mathcal{S}_{n-1}(n) \subseteq \cdots \subseteq \mathcal{S}_2(n).$$
The paper \cite{vikram} also contains some further results involving ``subgroup Lonely Runner spectra'' in which $k$-dimensional subtori are replaced by $k$-dimensional subgroups.

To make full use of the results of \cite{vikram}, it is useful to introduce the notion of relative Lonely Runner spectra.  For a subtorus $U \subseteq (\mathbb{R}/\mathbb{Z})^n$ and a natural number $k \leq \dim(U)$, define $\mathcal{S}_k(U)$, the \emph{order-$k$ Lonely Runner spectrum relative to $U$}, to be the set of all values of $D(T)$ as $T$ ranges over the $k$-dimensional proper subtori of $U$.  The special case $\mathcal{S}_1((\mathbb{R}/\mathbb{Z})^n)=\mathcal{S}_1(n)$ corresponds to the Lonely Runner spectra defined above.

The natural output of the aforementioned argument in \cite{vikram} is that for each natural number $n$ and real number $d \in \acc(\mathcal{S}_1(n))$, there are $\varepsilon>0$ and a finite list of proper subtori $U_1, \ldots, U_t \subseteq (\mathbb{R}/\mathbb{Z})^n$, with each $U_j$ satisfying $\dim(U_j) \geq 2$ and $D(U_j)=d$, such that
\begin{equation}\label{eq:acc-union-of-relative-spectra}
\mathcal{S}_1(n) \cap (d,d+\varepsilon)=\left(\bigcup_{j=1}^t \mathcal{S}_1(U_j) \right) \cap (d,d+\varepsilon).
\end{equation}
In other words, the accumulation of $\mathcal{S}_1(n)$ at the point $d$ is completely described by the (finitely many) relative spectra $\mathcal{S}_1(U_j)$, and so studying accumulation points essentially amounts to studying relative spectra.  (The equation \eqref{eq:acc-union-of-relative-spectra} also holds trivially for $d \notin \acc(\mathcal{S}_1(n))$ since taking $\varepsilon$ sufficiently small makes the left-hand side empty.)

\subsection{Main result}
Our main result is a structural characterization of order-$1$ relative spectra for $2$-dimensional subtori, which can be understood as a ``quantitative'' complement to the ``qualitative'' results described in the previous section.  We show that these relative spectra are highly arithmetically structured, in a sense that captures the infinite families described in \eqref{eq:old-conj} and in \eqref{eq:alec-alex}.  In other words, we show that convergence to accumulation points in Lonely Runner spectra is subject to surprisingly rigid constraints.  Let $$\Prog(\alpha, \beta):=\{\alpha s+\beta: s \in \mathbb{Z}_{\geq 0}\}$$
denote the infinite one-sided arithmetic progression with common difference $\alpha$ and shift $\beta$.

\begin{theorem}\label{thm:relative-spectrum-improved}
Let $n \geq 2$ be a natural number, and let $U \subseteq (\mathbb{R}/\mathbb{Z})^n$ be a $2$-dimensional proper subtorus.  Then there are $N \in \mathbb{Z}_{\geq 0}$ and positive rationals $\alpha_1, \ldots, \alpha_N, \beta_1, \ldots, \beta_N$ such that $\mathcal{S}_1(U)$ has finite symmetric difference with the set
$$\bigcup_{i=1}^N (D(U)+1/\Prog(\alpha_i, \beta_i)).$$
\end{theorem}

We make three remarks about the statement of this theorem.  First, the finite symmetric difference arises because a few $1$-dimensional subtori of $U$ can be too ``small'' to see the structure leading to the sets $D(U)+1/\Prog(\alpha_i, \beta_i)$; below we give an explicit example where this finite symmetric difference is necessary.  Second, our argument actually shows the ostensibly slightly stronger statement that there are only finitely many $1$-dimensional proper subtori $T$ of $U$ with $D(T)$ lying outside of the set $\{D(U)\} \cup [\cup_{i=1}^N (D(U)+1/\Prog(\alpha_i, \beta_i))]$.  Third, the set $D(U)+1/\Prog(\alpha_i, \beta_i)$ can be understood as a simple transformation of the set $\mathcal{S}_1(2)=1/\Prog(4,6)$.

Notice that all of the subtori $U_j$ in \eqref{eq:acc-union-of-relative-spectra} are necessarily $2$-dimensional for an accumulation point $d\notin \mathcal{S}_3(n)$; this is the case, for instance, for $d=\max \acc(\mathcal{S}_1(n))=\max \mathcal{S}_1(n-1)$ (see \cite{vikram}).  Here, Theorem~\ref{thm:relative-spectrum-improved} immediately gives the following pleasant characterization.

\begin{corollary}\label{cor:explicit-accumulation}
Let $n \geq 2$ be a natural number, and let $d \in \acc(\mathcal{S}_1(n)) \setminus \mathcal{S}_3(n)$ be a real number.  Then there are $\varepsilon>0$, $N \in \mathbb{Z}_{\geq 0}$, and positive rationals $\alpha_1, \ldots, \alpha_N, \beta_1, \ldots, \beta_N$ such that $\mathcal{S}_1(n) \cap (d,d+\varepsilon)$ has finite symmetric difference with the set 
$$\left(\bigcup_{i=1}^N (d+1/\Prog(\alpha_i, \beta_i))\right)\cap (d,d+\varepsilon).$$
\end{corollary}

\subsection{Comparison with previous work}
Section 6 of \cite{thesis} was concerned with the ``one very fast runner'' setting of the Lonely Runner Problem, in which we fix nonzero integers $v_1, \ldots, v_{n-1}$ and consider the maximum loneliness values associated with the speeds $v_1, \ldots, v_{n-1}, v_n$ for $v_n$ large.  This essentially amounts to studying the relative spectrum $$\mathcal{S}_1(\langle (v_1, \ldots, v_{n-1}) \rangle_\mathbb{R} \times (\mathbb{R}/\mathbb{Z})).$$  Our Theorem~\ref{thm:relative-spectrum-improved} can be understood as an extension of this study to general $2$-dimensional subtori.

For other computation-related results on the Lonely Runner Problem, see, e.g., \cite{tao,malikiosis}.

\subsection{Computing relative spectra}
The proof of Theorem~\ref{thm:relative-spectrum-improved} is quite explicit and provides an algorithm (albeit not so simple) for calculating the structure parameters $\alpha_i,\beta_i$ for any given $2$-dimensional subtorus $U$.  As a proof of concept, we carry out this calculation for several particular choices of $U$.

As a first example, we determine $\mathcal{S}_1(4)$ up to the first accumulation point, namely, $1/2-1/4=1/4$.  We show that this portion of the fourth Lonely Runner spectrum consists (up to finitely many exceptions) of only the infinite progressions in \eqref{eq:old-conj} and in \eqref{eq:alec-alex}; along the way, our argument provides a new (and more systematic) proof of the identity \eqref{eq:alec-alex} from \cite{computer}.

\begin{theorem}\label{thm:S_1(4)-first-acc-point}
The set $\mathcal{S}_1(4) \cap (1/4,1/2]$ has finite symmetric difference with the set
$$1/4+1/\Prog(8,12)=1/4+1/4\Prog(2,3).$$
\end{theorem}
Characterizing this finite symmetric difference is itself a finite calculation that we have not attempted to carry out (but numerical experiments suggest that there are no exceptional elements).

For our second example, we determine $\mathcal{S}_1(3)$ up to the second accumulation point, namely, $1/2-2/5=1/10$.  Recall from \cite{thesis} that $\mathcal{S}_1(3) \cap (1/6,1/2]=1/6+1/\Prog(9,12)=1/6+1/3\Prog(3,4)$.

\begin{theorem}\label{thm:S_1(3)-second-acc-point}
The set $\mathcal{S}_1(3) \cap (1/10,1/6]$ has finite symmetric difference with the set \[(1/10 + 4/5\Prog(5, 7)) \cup (1/10 + 3/5\Prog(5, 9)).\]
\end{theorem}

Our third example, of a slightly different flavor, is the discovery of a new infinite family in $\mathcal{S}_1(6)$ limiting to the first accumulation point, namely, $1/2-1/6=1/3$.
\begin{theorem}\label{thm:S_1(6)-prog}
The set $\mathcal{S}_1(6) \cap (1/3, 1/2]$ contains the set \[1/3 + 1/6\Prog(6, 11).\]
\end{theorem}

One could in principle calculate $\mathcal{S}_1(6) \cap (1/3,1/2]$ up to finite symmetric difference, following the model of the proof of Theorem~\ref{thm:S_1(4)-first-acc-point}, but this would be fairly lengthy and we have not attempted it.

In general, before using our methods to determine $\mathcal{S}_1(n)$ up to the first accumulation point, one must first determine all tight instances for $\mathcal{S}_1(n-1)$, i.e., proper $1$-dimensional subtori $T \subseteq (\mathbb{R}/\mathbb{Z})^{n-1}$ with $D(T)=\max \mathcal{S}_1(n-1)$; see Section~\ref{sec:identifying-candidates-u} below.  Determining all tight instances for $\mathcal{S}_1(n)$ is a (large) finite calculation for any particular $n$, thanks to \cite{vikram}, but at the present it is available in the literature only for $n \in \{1,2,3,5\}$.  This is one reason why studying accumulation points of $\mathcal{S}_1(6)$ is currently more convenient than studying accumulation points of $\mathcal{S}_1(5)$.  (Numerical simulations suggest that \eqref{eq:old-conj} holds for $n=5$.)

In a related direction, we give a necessary and sufficient geometric condition (Proposition~\ref{prop:finiteness-criterion}) for the relative spectrum of a $2$-dimensional subtorus being finite, i.e., having no progressions.

\subsection{Future directions}
There are several appealing directions for future work.  The most obvious is computing more examples of relative spectra for low-volume $2$-dimensional subtori.  Such calculations could lead to new families of counterexamples to \eqref{eq:old-conj}; perhaps one could could even discover such families for infinitely many values of $n$ or make a ``reasonable'' conjecture constraining $\mathcal{S}_1(n)$ up to the first accumulation point.  It also seems that a thorough understanding of $\mathcal{S}_1(3)$ is within reach; this set is well-ordered with order type $\omega^2+1$, and, using relative spectra, one could hope to characterize it up to an exceptional set of order type at most $\omega$.

Our arguments for Theorem~\ref{thm:relative-spectrum-improved} would likely lead to an analogous structure theorem for $\mathcal{S}_{m-1}(U)$ whenever $U$ is an $m$-dimensional proper subtorus, but the details would be messier and we have not pursued this.

Leaving the codimension-$1$ regime, a more ambitious open problem is obtaining a structure theorem for order-$1$ relative spectra of $3$-dimensional subtori.  One new difficulty is that the $1$-dimensional subgroups of $(\mathbb{R}/\mathbb{Z})^3$ are much more complicated than the the $1$-dimensional subgroups of $(\mathbb{R}/\mathbb{Z})^2$ (see Section~\ref{sec:finite-cyclic}). It is conceivable that for a $3$-dimensional proper subtorus $U$, the relative spectrum $\mathcal{S}_{1}(U)$ might ``look like'' a finite union of sets obtained in a simple way from $\mathcal{S}_1(3)$; presumably an improved understanding of $\mathcal{S}_1(3)$ would be essential here.

\subsection{Organization of the paper}
We prove Theorem~\ref{thm:relative-spectrum-improved} in Section~\ref{sec:main-proof}.  In Section~\ref{sec:identifying-candidates-u} we describe some techniques for determining all of the $2$-dimensional subtori with a given $D$-value and carry out this determination for the cases that will figure in our explicit examples.  We then prove Theorems~\ref{thm:S_1(4)-first-acc-point}, \ref{thm:S_1(3)-second-acc-point}, and~\ref{thm:S_1(6)-prog} in Sections~\ref{sec:explicit-comp-s1(4)}, \ref{sec:explicit-comp-s_1(3)}, and~\ref{sec:S_1(6)-prog}.  In Section~\ref{sec:finite} we establish a criterion for the finiteness of a relative spectrum and give an example of a finite, non-empty relative spectrum.

\section{Proof of Theorem~\ref{thm:relative-spectrum-improved}}\label{sec:main-proof}

\subsection{Proof strategy}\label{sec:sketch}
Before diving into the technical details of the proof of Theorem~\ref{thm:relative-spectrum-improved}, we sketch the main points of the argument.  Fix a choice of $u,v \in \mathbb{Z}^2$ such that $U=\langle u,v \rangle_{\mathbb{R}}$.  $1$-dimensional subtori of $U$ are of the form $T=\langle Au+Bv \rangle_\mathbb{R}$ for coprime integers $A,B$ satisfying $A \geq 0$.  The main goal is to express $D(T)$ in terms of $A,B$.  An easy perturbation argument \cite[Proposition 2.1]{thesis} shows that
$$D(T)=\min_{i,j,\epsilon} D(T_{i,j,\epsilon}),$$
where for $1 \leq i<j \leq n$ and $\epsilon \in \{+,-\}$, we have set $$T_{i,j,\epsilon}:=T \cap \{x_i=\epsilon x_j\}.$$
We will analyze each $D(T_{i,j,\epsilon})$ individually.

An important observation is that (besides some trivial cases) $T_{i,j,\epsilon}$ is a finite cyclic subgroup of
$$U_{i,j,\epsilon}:=U \cap \{x_i=\epsilon x_j\},$$
which is embedded in $(\mathbb{R}/\mathbb{Z})^n$ as a copy of $(\mathbb{R}/\mathbb{Z}) \times (\mathbb{Z}/K\mathbb{Z})$ for some natural number $K$.  In Section~\ref{sec:finite-cyclic}, we recall the structure of finite cyclic subgroups of $(\mathbb{R}/\mathbb{Z}) \times (\mathbb{Z}/K\mathbb{Z})$ and then show that the ``structure parameters'' for the cyclic group $T_{i,j,\epsilon}$ depend in a linear fashion on $A,B$ once we account for the residues of $A,B$ with respect to some large modulus $M$ (depending on $U$).

Consider the function $D: (\mathbb{R}/\mathbb{Z})^n \to [0,1/2]$ given by $D(x_1, \ldots, x_n):=\max_k |x_k-1/2|$.  The restriction of $D$ to each connected component $U_{i,j,\epsilon,\ell}$ of $U_{i,j,\epsilon}$ is continuous and piecewise-linear.  The work of Section~\ref{sec:finite-cyclic} essentially reduces the task of computing $D(T_{i,j,\epsilon,\ell})$ to the task of computing
$$\min_{x \in \langle 1/q \rangle} f(x)$$
where $q$ is a natural number and $f: \mathbb{R}/\mathbb{Z} \to \mathbb{R}$ is a continuous and piecewise-linear function assuming its minimum value at some points $\tau_1, \ldots, \tau_H$.  So it suffices to understand how well the $\tau_h$'s can be approximated by multiples of $1/q$.  This work is carried out in Section~\ref{sec:approximating-minima}, and it is here that the expressions $1/\Prog(\alpha,\beta)$ appear.  The output of this section (ignoring some important edge cases) is a sector decomposition of $\mathbb{R}_{\geq 0} \times \mathbb{R}$ such that if we partition the pairs $(A,B)$ according to the residues of $A,B$ modulo some $M$ and the sector containing $(A,B)$, then on each part of the partition we obtain a formula of the form
\begin{equation*}\label{eq:general-form-of-D(T_i,j,eps)}
D(T)=D(U)+1/(EA+FB)
\end{equation*}
for some rationals $E,F$.

At this point, a change of variables gives that on each part of our partition, $D(T)$ takes values in the set $D(U)+1/\Prog(\alpha,\beta)$
for some positive rationals $\alpha,\beta$ depending on $E,F$.  In order to check that the value $D(U)+1/(\alpha s+\beta)$ is in fact attained
for all sufficiently large $s \in \mathbb{N}$, we must verify that there are coprime $A,B$ satisfying $EA+FB=\alpha s+\beta$.  This last step, which boils down to a brief analytic number theory argument that we provide in Section~\ref{sec:coprime}, requires us to put a few modular constraints on $s$, which is harmless.


\subsection{Intersections of subtori and subgroups}\label{sec:finite-cyclic}
The first set of lemmas concerns the possible intersections of subtori and subspaces in a $2$-dimensional torus.  In the following discussion, a \emph{torus} is any topological group that is isomorphic to $(\mathbb{R}/\mathbb{Z})^n$ for some natural number $n$, and by a \emph{$k$-dimensional subgroup} of a torus we mean a closed $k$-dimensional Lie subgroup.  A \emph{subtorus} is a subgroup that is also a torus.

\begin{lemma}\label{lem:1-dim-subgroup}
Let $U$ be a $2$-dimensional torus.  Then every $1$-dimensional subgroup $\Gamma$ of $U$ is of the form $U' \times H$ where $U'$ is a $1$-dimensional subtorus of $U$ and $H$ is a finite cyclic group.  In particular, every such $\Gamma$ is isomorphic (as a Lie group) to $(\mathbb{R}/\mathbb{Z}) \times (\mathbb{Z}/K\mathbb{Z})$ for some natural number $K$.
\end{lemma}

\begin{proof}
Let $\Gamma$ be a $1$-dimensional subgroup of $U$.  Let $U'$ denote the identity component of $\Gamma$. Then $U'$ is a $1$-dimensional subtorus of $U$, and $\Gamma/U'$ is a discrete subgroup of $U/U' \cong \mathbb{R}/\mathbb{Z}$.  Hence, considered in $\mathbb{R}/\mathbb{Z}$, the group $\Gamma/U'$ is equal to $\langle 1/K \rangle$ for some natural number $K$.  Let $g \in \Gamma$ be such that $gU'$ is identified with $1/K$.  Then $Kg \in U'$, and (since $U' \cong \mathbb{R}/\mathbb{Z}$ is a divisible group) there is some $g' \in U'$ such that $Kg=Kg'$.  Now the element $g'':=g-g'$ has order $K$, and $g''U'$ is still identified with $1/K$.  Let $H:=\langle g'' \rangle$.  Then $|H|=K$ and $\Gamma=U' \times H$, as desired.
\end{proof}

Notice that every such isomorphism $\Gamma \cong (\mathbb{R}/\mathbb{Z}) \times (\mathbb{Z}/K\mathbb{Z})$ preserves the set of points with all coordinates rational.  In particular, there is a rational change of coordinates on $(\mathbb{R}/\mathbb{Z})^2$ such that in the new coordinates we have $\Gamma=(\mathbb{R}/\mathbb{Z}) \times \langle 1/K \rangle_\mathbb{Z} \subseteq (\mathbb{R}/\mathbb{Z})^2$.  We will now describe the intersection of such a $\Gamma$ with an arbitrary $1$-dimensional subtorus.

\begin{lemma}\label{lem:intersection-structure}
Let $K$ be a natural number, and let $A',B'$ be coprime integers with $B' \neq 0$.  Then
$$((\mathbb{R}/\mathbb{Z}) \times \langle 1/K \rangle_\mathbb{Z}) \cap \langle (A',B') \rangle_\mathbb{R}=\bigcup_{\ell=0}^{K-1} ((a\ell/(Kq)+\langle 1/q \rangle_\mathbb{Z})\times \{\ell/K\}),$$
where $q:=|B'|$, and $0 \leq a<K$ satisfies $a \equiv A' \pmod{K}$ if $B'>0$ and satisfies $a \equiv -A' \pmod{K}$ if $B'<0$.
\end{lemma}

\begin{proof}
Write
$$G:=((\mathbb{R}/\mathbb{Z}) \times \langle 1/K \rangle_\mathbb{Z}) \cap \langle (A',B') \rangle_{\mathbb{R}}.$$
Since $B' \neq 0$, we know that the $1$-dimensional subtorus $\langle (A',B') \rangle_\mathbb{R}$ is not equal to $(\mathbb{R}/\mathbb{Z}) \times \{0\}$.  Hence $G$ is a discrete group, and we can compute
\begin{align*}
G &=\{t(A',B'): t \in \mathbb{R}, tB' \in \langle 1/K \rangle_\mathbb{Z}\}\\
 &=\langle (A'/(KB'),1/K) \rangle_{\mathbb{Z}}\\
 &=\langle (A'\delta/(Kq),1/K) \rangle_{\mathbb{Z}},
\end{align*}
where $q=|B'|$ and $\delta:=B'/|B'| \in \{-1,1\}$.  Since $A',B'$ are coprime, we have
$$G \cap ((\mathbb{R}/\mathbb{Z}) \times \{0\})=\langle (A'\delta/q,0)\rangle_\mathbb{Z}=\langle 1/q \rangle_\mathbb{Z} \times \{0\}.$$
Finally, for each $0 \leq \ell<K$, the point $(A'\delta \ell/(Kq),\ell/K) \in G$ tells us that
$$G \cap ((\mathbb{R}/\mathbb{Z}) \times \{\ell/K\})=(A'\delta \ell/(Kq),\ell/K)+(\langle 1/q \rangle_\mathbb{Z} \times \{0\})=((a\ell/(Kq)+\langle 1/q \rangle_\mathbb{Z})\times \{\ell/K\}),$$
where $a \equiv A'\delta \pmod{K}$, as desired.


\end{proof}

We will apply the first lemma with $\Gamma=U_{i,j,\epsilon}$ and (after a change of coordinates) the second lemma with $\langle (A',B') \rangle_\mathbb{R}=T$.  We first require a bit of setup.

As in the sketch in Section~\ref{sec:sketch}, suppose that $U \subseteq (\mathbb{R}/\mathbb{Z})^n$ ($n \geq 2$) is a proper $2$-dimensional subtorus that is not contained in any subspace of the form $\{x_i=x_j\}$ or $\{x_i=-x_j\}$ for $1 \leq i<j \leq n$.  Then $U$ is the image modulo $\mathbb{Z}^n$ of the subspace $\langle u,v \rangle_\mathbb{R}$ for some vectors $u,v \in \mathbb{Z}^n$, and we can choose $u,v$ such that
\begin{equation}\label{eq:lattice-condition}
\langle u,v \rangle_\mathbb{R} \cap \mathbb{Z}^n=\langle u,v \rangle_\mathbb{Z};
\end{equation}
fix such a choice of $u,v$.
Every $1$-dimensional subtorus $T$ of $U$ is of the form
$$T=\langle A u+B v \rangle_\mathbb{R},$$
i.e., $T$ is the image modulo $\mathbb{Z}^n$ of $\langle Au+Bv \rangle_\mathbb{R}$, for some coprime integers $A,B$ with $A \geq 0$ and $(A,B) \notin \{(0,0),(0,-1)\}$. This correspondence between $T$ and $(A,B)$ is one-to-one.  For ease of reference, let $\mathcal{T}$ denote the set of all such pairs $(A,B)$.

For the remainder of this subsection, fix a choice of $1 \leq i<j \leq n$ and a choice of a sign $\epsilon \in \{+,-\}$.  The subgroup
$$U_{i,j,\epsilon}:=U \cap \{x_i=\epsilon x_j\}$$
is a $1$-dimensional subgroup of $U \cong (\mathbb{R}/\mathbb{Z})^2$ (recall our assumption that $U$ does not lie in the subspace $\{x_i=\epsilon x_j\}$), so
Lemma~\ref{lem:1-dim-subgroup} tells us that $U_{i,j,\epsilon}$ is isomorphic to the cartesian product of $\mathbb{R}/\mathbb{Z}$ and $\mathbb{Z}/K\mathbb{Z}$ for some natural number $K=K_{i,j,\epsilon}$.   In order to apply Lemma~\ref{lem:intersection-structure}, we need to write down an explicit isomorphism $U \cong (\mathbb{R}/\mathbb{Z})^2$ in coordinates ``adapted'' to $U_{i,j,\epsilon}$.  To this end, fix a primitive vector $u'=u'_{i,j,\epsilon} \in \mathbb{Z}^n$ such that the identity component of $U_{i,j,\epsilon}$ equals $\langle u' \rangle_\mathbb{R}$.  Now fix a vector $v'=v'_{i,j,\epsilon} \in \mathbb{Z}^n$ such that
$$\langle u',v' \rangle_\mathbb{Z}=\langle u,v \rangle_\mathbb{Z};$$
such a $v'$ exists by the primitivity assumption on $u'$.  Hence the linear map $\psi=\psi_{i,j,\epsilon}:\langle u,v \rangle_\mathbb{R} \to \mathbb{R}^2$ given by $\psi(u'):=(1,0)$ and $\psi(v'):=(0,1)$ descends to an isomorphism from $U$ to $(\mathbb{R}/\mathbb{Z})^2$.  Since $\psi$ sends the identity component of $U_{i,j,\epsilon}$ to $\mathbb{R}/\mathbb{Z} \times \{0\}$, we see that
$$\psi(U_{i,j,\epsilon})=(\mathbb{R}/\mathbb{Z}) \times \langle 1/K \rangle_\mathbb{Z}$$
is a concrete witness to the isomorphism described earlier in this paragraph.

Write
$$\psi(T)=\psi(\langle Au+Bv \rangle_\mathbb{R})=\langle A\psi(u)+B\psi(v) \rangle_\mathbb{R}.$$
Since $\langle u',v' \rangle_\mathbb{Z}=\langle u,v \rangle_\mathbb{Z}$, the vectors $\psi(u),\psi(v)$ have integer coordinates, say,
$$\psi(u)=(z_1,z_2) \quad \text{and} \quad \psi(v)=(z_3,z_4),$$
with $z_1,z_2,z_3,z_4 \in \mathbb{Z}$ (depending on $i,j,\epsilon$).  So we write
\begin{equation}\label{eq:T-in-new-coordinates}
\psi(T)=\langle A\psi(u)+B\psi(v)\rangle_\mathbb{R}=\langle (Az_1+Bz_3,Az_2+Bz_4)\rangle_{\mathbb{R}}.
\end{equation}
Since $\psi(u),\psi(v)$ are linearly independent, we have $z_1z_4 \neq z_2z_3$.  Now we apply Lemma~\ref{lem:intersection-structure} in order to obtain the main result of this section, as follows.

\begin{proposition}\label{prop:cyclic-subgroup-parameters}
Let $n,U,u,v,i,j,\epsilon, K=K_{i,j,\epsilon},\psi=\psi_{i,j,\epsilon}, z_1=z_{1,i,j,\epsilon},z_2=z_{2,i,j,\epsilon},z_3=z_{3,i,j,\epsilon},z_4=z_{4,i,j,\epsilon}$ be as above.  Then there is a natural number $M$ (depending on all of the parameters introduced so far) such that for each choice of $0 \leq \aleph,\beth<M$ and each choice of a sign $\delta \in \{+,-\}$, the following holds: There are rational numbers $E=E_{\aleph, \beth,\delta}, F=F_{\aleph, \beth,\delta}$ (not both zero) and a nonnegative integer $a= a_{\aleph, \beth,\delta}<K'_{\aleph, \beth,\delta}$ such that for $T=\langle Au+Bv \rangle_\mathbb{R}$, we have
$$\psi(T_{i,j,\epsilon})=\bigcup_{\ell=0}^{K-1} ((a\ell/(Kq)+\langle 1/q \rangle_\mathbb{Z})\times \{\ell/K\})$$
with
$$q=EA+FB$$
whenever $(A,B) \in \mathcal{T}$ satisfies $A \equiv \aleph \pmod{M}$ and $B \equiv \beth \pmod{M}$, and
$$\delta(Az_2+Bz_4)>0.$$
\end{proposition}

\begin{proof}
Let $$\omega:=\gcd(Az_1+Bz_3,Az_2+Bz_4).$$
Then the integers $A':=(Az_1+Bz_3)/\omega$ and $B':=(Az_2+Bz_4)/\omega$ are coprime, and we can apply Lemma~\ref{lem:intersection-structure} to find that
$$\psi(T_{i,j,\epsilon})=\bigcup_{\ell=0}^{K-1} ((a\ell/(Kq)+\langle 1/q \rangle_\mathbb{Z})\times \{\ell/K\}),$$
with $q:=\delta B'$ and $0 \leq a<K$ satisfying $a \equiv \delta A' \pmod{K}$ (by the definition of $\delta$).  In particular, we will obtain the conclusion of the lemma with the choices
$$E:=\delta z_2/\omega \quad \text{and} \quad F:=\delta z_4/\omega.$$
To complete the proof of the proposition, it remains to show that the quantities $A',B'$ are linear functions of $A,B$, under appropriate modular conditions.  In particular, it suffices to show that $\omega$ depends only on the residues of $A,B$ modulo some fixed natural number $M$ to be specified shortly.

Notice that $\omega$ divides
$$z_4(Az_1+Bz_3)-z_3(Az_2+Bz_4)=(z_1z_4-z_2z_3)A$$
and
$$z_1(Az_2+Bz_4)-z_2(Az_1+Bz_3)=(z_1z_4-z_2z_3)B.$$
Since $A,B$ are coprime by assumption, we conclude that $\omega$ divides
$$M:=|z_1z_4-z_2z_3|,$$
which is nonzero by the remark following \eqref{eq:T-in-new-coordinates}.  It follows that $\omega$ depends only on the residues of $Az_1-Bz_3,Az_2+Bz_4$ modulo $M$, which in turn depend only on the residues of $A,B$ modulo $M$.
\end{proof}

\subsection{Approximating minima}\label{sec:approximating-minima}
The next set of lemmas concerns approximating minima of continuous piecewise-linear functions on $\mathbb{R}/\mathbb{Z}$; see Figure~\ref{fig:approx}.

For $\tau,b \in \mathbb{Q}$ and $q \in \mathbb{N}$, define
$$\Approx^-(\tau,b;q):=\min\{\tau-(b/q+r/q): \text{$r \in \mathbb{Z}$, $b/q+r/q \leq \tau$}\}$$
and
$$\Approx^+(\tau,b;q):=\min\{(b/q+r/q)-\tau: \text{$r \in \mathbb{Z}$, $b/q+r/q \geq \tau$}\}$$
to be the errors in best approximations to $\tau$ (from below and above) by numbers of the form $b/q+r/q$, for $r \in \mathbb{Z}$.  Since $\Approx^-(\tau,b;q)=\Approx^-(\tau+1/q,b;q)$, this quantity is also well-defined for $\tau \in \mathbb{Q}/\mathbb{Z}$ (and likewise for $\Approx^+(\tau,b;q)$).  The next lemma shows that once we specify the residue class of $q$ with respect to some modulus depending on $\tau,b$, the functions $\Approx^-(\tau,b;q),\Approx^+(\tau,b;q)$ have a simple form. This calculation is the source of the expressions $1/\Prog(\alpha,\beta)$ in Theorem~\ref{thm:relative-spectrum-improved}.

\begin{lemma}\label{lem:approximating-by-1/q}
Let $\tau,b \in \mathbb{Q}$, and write $\tau=w/x$ and $b=y/z$ with $w,x,y,z$ nonnegative integers.  Let $q$ be a natural number.  Then
$$\Approx^-(\tau,b;q)=\frac{R^-}{xzq} \quad \text{and} \quad \Approx^+(\tau,b;q)=\frac{R^+}{xzq},$$
where $0 \leq R^-,R^+ <xz$ satisfy
$$R^- \equiv wzq-xy \pmod{xz} \quad \text{and} \quad R^+ \equiv xy-wzq \pmod{xz}.$$
\end{lemma}
Notice that $R^-,R^+$ depend on only the residue of $q$ modulo $xz$, and that the parameters $w,x,y,z$ depend on only $\tau,b$.  Note also that $R^-+R^+=xz$ except when $R^-=R^+=0$.

\begin{proof}
We prove the lemma only for $\Approx^-(\tau,b,q)$ since the argument for $\Approx^+(\tau,b,q)$ is identical.  In the definition of $\Approx^-(\tau,b;q)$, the minimum is achieved for $r=\lfloor q(\tau-b/q)\rfloor=\lfloor q\tau-b \rfloor$.  We compute
$$q\tau-b= wq/x-y/z =(wzq-xy)/(xz).$$
The fractional part of $(wzq-xy)/(xz)$ is $R^-/(xz)$, so
$$r=(q\tau-b)-R^-/(xz)$$
and
$$\Approx^-(\tau,b;q)=\tau-(b/q+r/q)=R^-/(xzq),$$
as desired.
\end{proof}

We will now apply this lemma where $\tau$ ranges over the points at which a piecewise-linear function achieves its minimum value.

\begin{lemma}\label{lem:D-on-a-slice}
Let $f: \mathbb{R}/\mathbb{Z} \to \mathbb{R}$ be a continuous, piecewise-linear function with finitely many pieces, each with rational slope and rational endpoints, and let $b$ be a fixed rational number.  Then there is a natural number $M$ (depending on $f,b$) such for each integer $0 \leq Q <M$, the following holds: There is a nonnegative rational $\gamma_Q$ (depending also on $f,b$) such that
\begin{equation}\label{eq:min-f-on-an-AP}
\min_{t \in b/q+\langle 1/q \rangle} f(t)=\min_{t \in \mathbb{R}/\mathbb{Z}}f(t)+\frac{\gamma_Q}{q}
\end{equation}
for all sufficiently large natural numbers $q$ satisfying $q \equiv Q \pmod{M}$.
\end{lemma}

\begin{proof}
Write $m:=\min_{t \in \mathbb{R}/\mathbb{Z}}f(t)$.  We first dispose of the trivial case where $f(t)=m$ on some interval $I \subseteq \mathbb{R}/\mathbb{Z}$ of strictly positive length.  Since $b/q+\langle 1/q \rangle$ is $1/q$-dense in $\mathbb{R}/\mathbb{Z}$, we see that
$$\min_{t \in b/q+\langle 1/q \rangle} f(x)=m$$
for all $q \geq |I|$ (with no mention of modular restrictions), and the conclusion of the lemma holds with $M=1$ and $\gamma_0=0$.

Now suppose that $f$ achieves the value $m$ at only finitely many points $\tau_1, \ldots, \tau_H$.  Then there are some $\delta, \varepsilon>0$ such that $f(t)>m+\delta$ whenever $t$ does not lie within $\varepsilon$ of some $\tau_h$.  Since $f$ is continuous and $b/q+\langle 1/q \rangle$ is $1/q$-dense, we know that $\min_{t \in b/q+\langle 1/q \rangle} f(t)<m+\delta$ for all $q$ sufficiently large (in terms of $f$).  In particular, the minimum in \eqref{eq:min-f-on-an-AP} is achieved at a point $t$ that lies within $\varepsilon$ of some $\tau_h$; for the remainder of the proof we will assume that $q$ is large enough for this to be the case.

For each $\tau_h$, let $-\lambda_h^-, \lambda_h^+$ denote the (rational by assumption) slopes of the pieces of $f$ directly to the left and right (respectively) of $\tau_h$; notice that $\lambda_h^-,\lambda_h^+>0$.  After possibly shrinking $\varepsilon$ further, we find that $$f(\tau_h-t)=m+\lambda_h^-t \quad \text{and} \quad f(\tau_h+t)=m+\lambda_h^+t$$
for every $h$ and for all $0 \leq t<\varepsilon$.  Hence $\min_{t \in b/q+\langle 1/q \rangle} f(t)$ is equal to the minimum of the values
$$m+\lambda_h^- \Approx^-(\tau_h,b;q) \quad \text{and} \quad m+\lambda_h^+ \Approx^+(\tau_h,b;q)$$
over $1 \leq h \leq H$.  Write $b=y/z$ and $\tau_h=w_h/x_h$ as fractions, and let 
$M:=\lcm(z,x_1,\ldots, x_H)$.  Then for each $h$ we have
$$m+\lambda_h^- \Approx^-(\tau_h,b;q)=m+\frac{\lambda_h^-R_h^-}{x_hzq} \quad \text{and} \quad m+\lambda_h^+ \Approx^+(\tau_h,b;q)=m+\frac{\lambda_h^+R_h^+}{x_hzq},$$
where $R_h^-, R_h^+$ are the parameters output by Lemma~\ref{lem:approximating-by-1/q} for $\tau_h,b$.  It is important that $R_h^-,R_h^+$ depend on only the residue of $q$ modulo $x_h z$, which divides $M$.  For each $0 \leq Q<M$, consider the case corresponding to $q \equiv Q \pmod{M}$, and let $\gamma_Q$ denote the minimum of the values $\lambda_h^-R_h^-/(x_h z),\lambda_h^+R_h^+/(x_h z)$ over all $h$.  (Since $\lambda_h^-,\lambda_h^+>0$, we have $\gamma_Q \geq 0$.)  Then for all sufficiently large $q$ congruent to $Q$ modulo $M$, we have
$$\min_{t \in b/q+\langle 1/q \rangle} f(t)=m+\frac{\gamma_Q}{q},$$
as desired.
\end{proof}

We now apply this lemma where $f$ is the restriction of $D$ to a connected component of $U_{i,j,\epsilon}$ and the quantities $b,q$ depend ``nicely'' on the subtorus $T$, as described in Proposition~\ref{prop:cyclic-subgroup-parameters}.

\begin{figure}[!htb]
    \centering
\begin{tikzpicture}
    \draw[thick] (0,9) -- (0,0) -- (12,0) -- (12,9);
    \foreach \x in {1,4,7,10} {
        \fill (\x, 0) circle (2pt);
    }
    \draw[thick] (0,5) -- (2,3) -- (5,9) -- (7,7) -- (8,4) -- (11,6) -- (12,5);
    \draw[dashed] (2,0) -- (2,3);
    \draw[dashed] (1,0) -- (1,4);
    \draw[dashed] (4,0) -- (4,7);
    \draw[dashed] (0,3) -- (2,3);
    \draw[dashed] (0,4) -- (1,4);
    \draw[dashed] (0,7) -- (4,7);
    \node[below] at (2,-0.1) {$\tau$};
    \node[left] at (0,3) {$m$};
    \node[left] at (0,4) {$m+\lambda^-\Approx^-$};
    \node[left] at (0,7) {$m+\lambda^+\Approx^+$};
    \draw[decorate,decoration={brace,amplitude=10pt}] (2,0) -- (4,0);
    \draw[decorate,decoration={brace,amplitude=10pt}] (4.07,7) -- (4.07,3);
    \node[below] at (3,1) {$\Approx^+$};
    \node[right] at (4.4,5) {$\lambda^+ \Approx^+$};
    \node[left] at (3,5) {$\lambda^+$};
\end{tikzpicture}
    \caption{This figure illustrates some of the quantities in Lemmas~\ref{lem:approximating-by-1/q} and~\ref{lem:D-on-a-slice}.  Here we are approximating the minimum at $\tau=1/6$ using the coset $1/12+\langle 1/4 \rangle_\mathbb{Z}$, and the slopes of the pieces to the left and right of $\tau$ are $-\lambda^-=-1$ and $\lambda^+=2$.}
    \label{fig:approx}
\end{figure}


\begin{proposition}\label{prop:D(T_i,j,eps)-formula}
Let $n \geq 2$ be a natural number, and let $U \subseteq (\mathbb{R}/\mathbb{Z})^n$ be a $2$-dimensional proper subtorus that is not contained in any subspace of the form $\{x_i=x_j\}$ or $\{x_i=-x_j\}$ for $1 \leq i<j \leq n$.  Let $u,v \in \mathbb{Z}^n$ be such that $U=\langle u,v \rangle_\mathbb{R}$ and \eqref{eq:lattice-condition} is satisfied.  Then there are a natural number $M'$ and a finite set $\mathcal{L}$ of rational half-lines in $\mathbb{R}_{\geq 0} \times \mathbb{R}$ (depending on the parameters introduced so far) such that for each choice of $0 \leq \aleph,\beth<M'$ the following holds:

There are a partition of $(\mathbb{R}_{\geq 0} \times \mathbb{R}) \setminus \mathcal{L}$ into sectors $\sigma_1=\sigma_{\aleph,\beth,1}, \ldots, \sigma_X=\sigma_{\aleph, \beth,X(\aleph,\beth)}$; and $\kappa=\kappa_{\aleph,\beth,k} \in \{0,1\}$ and rationals $E'=E'_{\aleph,\beth,k},F'=F'_{\aleph,\beth,k}$ for each $1 \leq k \leq X(\aleph,\beth)$, such that for $T=\langle Au+Bv \rangle_\mathbb{R}$, we have
$$D(T)=D(U)+\frac{\kappa}{E'A+F'B}$$
whenever $(A,B) \in \mathcal{T} \setminus \mathcal{L}$ satisfies $(A,B) \equiv (\aleph,\beth) \pmod{M'}$, $(A,B) \in \sigma_k$, and $A^2+B^2$ is sufficiently large.  Moreover, for each half-line $L \in \mathcal{L}$, there are $\kappa=\kappa_{\aleph,\beth,L} \in \{0,1\}$ and rationals $E'=E'_{\aleph,\beth,L},F'=F'_{\aleph,\beth,L}$ such that for $T=\langle Au+Bv \rangle_\mathbb{R}$, we have
$$D(T)=D(U)+\frac{\kappa}{E'A+F'B}$$
whenever $(A,B) \in \mathcal{T} \cap L$ satisfies $(A,B) \equiv (\aleph,\beth) \pmod{M'}$ and $A^2+B^2$ is sufficiently large.
\end{proposition}


\begin{proof}
Let $M''$ be the lcm of all of the values of $M=M_{i,j,\epsilon}$ output by Proposition~\ref{prop:cyclic-subgroup-parameters} as we range over indices $1 \leq i<j \leq n$ and signs $\epsilon\in \{+,-\}$.  For the remainder of the proof, assume that the residues of $A,B$ modulo $M''$ are fixed.  Notice that this choice fixes the residues of $A,B$ modulo each individual $M$ output by Proposition~\ref{prop:cyclic-subgroup-parameters}.  Our partition of $\mathbb{R}_{\geq 0} \times \mathbb{R}$ into sectors will be the least common refinement of partitions $\mathcal{P}_{i,j,\epsilon}$ and $\mathcal{P}_{i,j,\epsilon,\ell; i',j',\epsilon',\ell'}$ defined later in the proof.  Each of these partitions will have a finite number of sectors, so their refinement will also have a finite number of sectors.  Since $\mathcal{T}$ contains at most one point on each line through the origin, the intersection of $\mathcal{T}$ with the boundary lines defining our sectors will be finite.  By taking $\sqrt{A^2+B^2}$ sufficiently large, we can ensure that $(A,B)$ never lies on the boundary of a sector, and we will tacitly use this assumption in what follows.

Consider now a fixed choice of indices $1 \leq i<j \leq n$ and a sign $\epsilon \in \{+,-\}$, and take $K,\psi,z_1,z_2,z_3,z_4$ (all depending on $i,j,\epsilon$) as in the statement of Proposition~\ref{prop:cyclic-subgroup-parameters}.  Let $\mathcal{P}_{i,j,\epsilon}$ be the partition of $\mathbb{R}_{\geq 0} \times \mathbb{R}$ into two sectors according to the sign of the expression $xz_2+yz_4$ (for $(x,y) \in \mathbb{R}_{\geq 0} \times \mathbb{R}$).  Fix one of the sectors $\sigma \in \mathcal{P}_{i,j,\epsilon}$, and assume that $(A,B) \in \sigma$.  Then Proposition~\ref{prop:cyclic-subgroup-parameters} supplies integers $0 \leq a<K$ and rationals $E,F$ (depending on $i,j,\epsilon,\sigma$) such that
$$\psi(T_{i,j,\epsilon})=\bigcup_{\ell=0}^{K-1} ((a\ell/(Kq)+\langle 1/q \rangle_\mathbb{Z})\times \{\ell/K\}),$$
where $q=EA+FB$.  Now write $U_{i,j,\epsilon}$ as the disjoint union
$$U_{i,j,\epsilon}=\cup_{\ell=0}^{K-1} U_{i,j,\epsilon,\ell},$$
such that
$$\psi(U_{i,j,\epsilon,\ell})=(\mathbb{R}/\mathbb{Z}) \times \{\ell/K\}$$
for each $0 \leq \ell <K$.  Also define
$$T_{i,j,\epsilon,\ell}:=T\cap U_{i,j,\epsilon,\ell}=T_{i,j,\epsilon}\cap U_{i,j,\epsilon,\ell},$$
and notice that $$\psi(T_{i,j,\epsilon,\ell})=(a\ell/(Kq)+\langle 1/q \rangle_\mathbb{Z})\times \{\ell/K\}.$$
Since the function $D$ is continuous and piecewise linear with finitely many pieces, so is its restriction to $U_{i,j,\epsilon,\ell}$.   It follows that the function $D_{i,j,\epsilon,\ell}: \mathbb{R}/\mathbb{Z} \to \mathbb{R}$ defined by $D_{i,j,\epsilon,\ell}(x):=(D \circ \psi^{-1})(x,\ell/K)$ satisfies the hypotheses of Lemma~\ref{lem:D-on-a-slice}; let $M,\gamma_Q$ (for $0 \leq Q<M$) be as in the conclusion of that lemma with $b=a\ell/K$.  The residue of $q$ modulo $M$ depends only on the residues of $A,B$ modulo the product of $M$ with the denominators of $E,F$; call this natural number $M_{i,j,\epsilon,\sigma,\ell}$.  Suppose (as we will ensure shortly) that the residue classes of $A,B$ modulo $M_{i,j,\epsilon,\sigma,\ell}$ are fixed.  This also fixes the residue class of $q$ modulo $M$, say, $q \equiv Q \pmod{M}$; write $\gamma$ for the $\gamma_Q$ from Lemma~\ref{lem:D-on-a-slice}.

Now Lemma~\ref{lem:D-on-a-slice} tells us that
$$D(T_{i,j,\epsilon,\ell})=D_{i,j,\epsilon,\ell}(a\ell/(Kq)+\langle 1/q \rangle_\mathbb{Z})=D(U_{i,j,\epsilon,\ell})+\gamma/q$$
whenever $q$ is sufficiently large, say, at least some $q_0=q_{0,i,j,\epsilon,\sigma,\ell}$.  Since $q$ is a rational linear combination of $A,B$, we see that the set of pairs $(A,B)$ corresponding to $q<q_0$ lies on a finite set $\mathcal{L}_{i,j,\epsilon,\sigma,\ell}$ of rational half-lines.\footnote{This set of pairs $(A,B)$ lies naturally on a finite set of rational lines.  Any such non-vertical line intersects $\mathbb{R}_{\geq 0} \times \mathbb{R}$ in a half-line, and any such vertical line can be arbitrarily broken into two half-lines.}  Suppose now that $(A,B)\notin \mathcal{L}_{i,j,\epsilon,\sigma,\ell}$.  By dividing through by $\gamma$ if $\gamma \neq 0$, we can write the previous centered equation as
\begin{equation}\label{eq:D-for-T_i,j,eps,ell}
D(T_{i,j,\epsilon,\ell})=D(U_{i,j,\epsilon,\ell})+\frac{\kappa}{E'A+F'B},
\end{equation}
where $\kappa \in \{0,1\}$ and $E',F'$ are still rational (and now depend on $\ell$ in addition to $i,j,\epsilon$).  

We pause to update a few parameters.  Let $\mathcal{L}$ be the union of the sets $\mathcal{L}_{i,j,\epsilon,\sigma,\ell}$ from the previous paragraph; for this paragraph and the two that follow, assume that $(A,B)$ does not lie on $\mathcal{L}$, so that \eqref{eq:D-for-T_i,j,eps,ell} holds for all $i,j,\epsilon,\ell$.  Let $M'$ be the lcm of $M''$ and the natural numbers $M_{i,j,\epsilon,\sigma,\ell}$ from the previous paragraph; for the remainder of the proof, assume that $(A,B) \equiv (\aleph,\beth) \pmod{M'}$ for fixed $0 \leq \aleph, \beth<M'$.  Let $\mathcal{P}'$ (depending on $\aleph, \beth$) be the least common refinement of the partitions $\mathcal{P}_{i,j,\epsilon}$ from the previous paragraph; for the remainder of the proof, assume that $(A,B)$ lies in a fixed sector $\sigma$ of $\mathcal{P}_{i,j,\epsilon}$.  We will later refine the partition $\mathcal{P}'$ further.  

It is a basic fact (see, e.g., \cite[Proposition 2.1]{thesis}) that
$$D(T)=\min_{i,j,\epsilon} D(T_{i,j,\epsilon}).$$
Hence
$$
D(T)=\min_{i,j,\epsilon,\ell}D(T_{i,j,\epsilon,\ell}).$$
Let $Y$ be the set of quadruples $(i,j,\epsilon,\ell)$ such that $D(U_{i,j,\epsilon,\ell})=D(U)$, and note that
$$
D(T)=\min_{(i,j,\epsilon,\ell) \in Y}D(T_{i,j,\epsilon,\ell})$$
for $A^2+B^2$ sufficiently large since the right-hand side tends to $D(U)$ and each $(i,j,\epsilon,\ell) \notin Y$ has $D(T_{i,j,\epsilon,\ell})$ uniformly bounded above $D(U)$.  We first dispose of the special case in which there is some $(i,j,\epsilon,\ell) \in Y$ with $\kappa_{i,j,\epsilon,\ell}=0$.  If this occurs, then
$$D(T)=D(U)$$
for all $(A,B)$ under consideration, and we obtain the conclusion of the proposition with $\kappa:=0$ and (say) $E'=1, F'=0$.  It remains to consider the case where $\kappa_{i,j,\epsilon,\ell}=1$ for all $(i,j,\epsilon,\ell) \in Y$.  

For each pair of  distinct tuples $(i,j,\epsilon,\ell),(i',j',\epsilon',\ell') \in Y$, let us compare the expressions in \eqref{eq:D-for-T_i,j,eps,ell} appearing for these two tuples.
There are rationals $E_1,E_2,F_1,F_2$ such that
$$D(T_{i,j,\epsilon,\ell})=D(U)+\frac{1}{E_1A+F_1B} \quad \text{and} \quad D(T_{i',j',\epsilon',\ell'})=D(U)+\frac{1}{E_2A+F_2B}.$$
Which of these two expressions is smaller is determined by the sign of the expression
$$(E_2A+F_2B)-(E_1A+F_1B)=(E_2-E_1)A+(F_2-F_1)B.$$
Hence there is a partition $\mathcal{P}_{i,j,\epsilon,\ell; i',j',\epsilon',\ell'}$ of $\mathbb{R}_{\geq 0} \times \mathbb{R}$ into at most two sectors such that on the each sector, the above expression is either always non-negative or always non-positive.  Let $\mathcal{P}$ be the least common refinement of $\mathcal{P}'$ (from above) and the partitions $\mathcal{P}_{i,j,\epsilon,\ell; i',j',\epsilon',\ell'}$.  Hence each sector $\sigma$ of $\mathcal{P}$ determines a tuple $(i,j,\epsilon,\ell)$, with its associated rationals $E',F'$, such that
\begin{equation*}\label{eq:D-final-formula}
D(T)=D(T_{i,j,\epsilon,\ell})=D(U)+\frac{1}{E'A+F'B}
\end{equation*}
whenever $(A,B) \in (\mathcal{T} \setminus \mathcal{L}) \cap \sigma$.  This establishes the first conclusion of the proposition.

It remains to analyze what happens when $(A,B)$ lies on $\mathcal{L}$.  Fix a half-line $L \in \mathcal{L}$, and consider pairs $(A,B) \in L$.  If $L \cap \mathcal{T}$ is finite, then there is nothing to show, so assume that this intersection is infinite.  Let $Y'=Y'_{\aleph,\beth,L}$ denote the set of all quadruples $(i,j,\epsilon,\ell) \in Y$ whose corresponding values of $q$ tend to infinity as $(A,B)$ tends to infinity along $L$.  We claim that $Y'$ is nonempty.  It was shown in \cite{vikram} that the volume of the torus $T=\langle Au+Bv \rangle_\mathbb{R}$ tends to infinity with $A^2+B^2$ and that therefore $T$ becomes $o(1)$-dense in $U$ as $A^2+B^2$ tends to infinity; in particular, $D(T)$ tends to $D(U)$.  Since each quadruple $(i,j,\epsilon,\ell) \notin Y'$ has $D(T_{i,j,\epsilon,\ell})$ uniformly bounded away from $D(U)$ from above, we see that $Y'$ is nonempty, as claimed.  Now we repeat the argument from the previous two paragraphs, with $Y$ replaced by $Y'$, and we note that, sufficiently far from the origin, the half-line $L$ lies entirely in a single sector of $\mathcal{P}$.  This establishes the second conclusion of the proposition.
\end{proof}

\subsection{Coprime pairs in arithmetic progressions}\label{sec:coprime}

The third stage of our preparatory lemmas concerns when long arithmetic progressions in $\mathbb{Z}^2$ are guaranteed to contain some points with coprime coordinates.  The first lemma will help us find coprime pairs $(A,B)$ on intersections of rational lines with the sectors from Proposition~\ref{prop:D(T_i,j,eps)-formula}.

\begin{lemma}\label{lem:coprime}
Let $C \geq 0$ be a real number, and let $a_1,a_2$ be nonzero integers.  For $N$ a parameter, let $b_1,b_2$ be integers with $|b_1|,|b_2| \leq CN$, and let $I$ be an interval (of integers) of length $N$.  Suppose that $\gcd(a_1,a_2,b_1,b_2)=1$ and $a_1b_2 \neq a_2b_1$.  Then $$|\{x \in I: \gcd(a_1x+b_1,a_2x+b_2)=1\}|$$ tends to infinity with $N$ (uniformly in $b_1,b_2, I$).
\end{lemma}

\begin{proof}
Note that $\gcd(a_1x+b_1,a_2x+b_2)$ divides
$$Z:=|a_1(a_2x+b_2)-a_2(a_1x+b_1)|=|a_1b_2-a_2b_1|,$$
which by assumption is nonzero and has size $Z=O(N)$.  Hence, we wish to count $x \in I$ such that $\gcd(a_1x+b_1,a_2x+b_2)$ is coprime to $Z$.  Let $p$ be a prime divisor of $Z$.  The forms $a_1x+b_1,a_2x+b_2$ cannot both be identically zero modulo $p$ by the assumption that $\gcd(a_1,a_2,b_1,b_2)=1$, so there is at most a single residue class $x_p$ modulo $p$ where both forms vanish.  If there is no such residue class, then pick $x_p$ arbitrarily.

By the Chinese Remainder Theorem, there is some integer $R$ such that $R \equiv x_p \pmod{p}$ for each prime $p$ dividing $Z$, and we have $\gcd(a_1x+b_1,a_2x+b_2)=1$ whenever $x-R$ is coprime to $Z$.  So it suffices to lower-bound the number of elements of $J:=I-R$ that are coprime to $Z$.  Let $\mu$ denote the M\"obius function, and recall that $\sum_{d|m}\mu(d)$ equals $1$ if $m=1$ and equals $0$ if $m>1$.  Using this identity, we compute


\begin{align*}
|\{x \in J: \gcd(x,Z)=1\}| &=\sum_{x \in J} \mathbbm{1}_{\gcd(x,Z)=1}\\
 &=\sum_{x \in J} \sum_{d|x,Z} \mu(d)\\
 &=\sum_{d|Z} \mu(d) \cdot |\{x \in J: d|x\}|\\
 &=\sum_{d|Z} \mu(d) (|J|/d+O(1))\\
 &=(\varphi(Z)/Z)N+O(\tau(Z)),
\end{align*}
where $\varphi$ is Euler's totient function and $\tau$ is the divisor function.  Recalling that $Z=O(N)$, we apply the standard bounds $\varphi(Z)/Z \gg 1/\log\log(Z)$ and $\tau(Z) \ll e^{O(\log(Z)/\log\log(Z))}$ to find that
$$|\{x \in J: \gcd(x,Z)=1\}|\gg N/\log\log(N),$$
which certainly tends to infinity with $N$.
\end{proof}



The next lemma will help us characterize the coprime pairs $(A,B)$ on the exceptional half-lines from Proposition~\ref{prop:D(T_i,j,eps)-formula}.

\begin{lemma}\label{lem:coprime-on-a-line}
Let $a_1, b_1,a_2,b_2$ be integers with $a_1>0$.  Then there is a natural number $M$ such that for $x$ a sufficiently large natural number, whether or not $\gcd(a_1x+b_1,a_2x+b_2)=1$ depends only on the residue of $x$ modulo $M$.
\end{lemma}

\begin{proof}
As in the proof of the previous lemma, note that $\gcd(a_1x+b_1,a_2x+b_2)$ divides
$$a_1(a_2x+b_2)-a_2(a_1x+b_1)=a_1b_2-a_2b_1.$$
Suppose first that $a_1b_2-a_2b_1=0$.  For each $i \in \{1,2\}$, let
$$a_i':=a_i/\gcd(a_i,b_i) \quad \text{and} \quad b_i':=b_1/\gcd(a_i,b_i),$$
so that $\gcd(a_i',b_i')=1$ and we have
$$a_ix+b_i=\gcd(a_i,b_i)(a_i'x+b_i').$$
Now $a_1 b_2=a_2 b_1$ implies that $a_1'b_2'=a'_2b'_1$, and from $\gcd(a_1',b_1')=\gcd(a_2',b_2')=1$ we conclude that $(a'_1,b'_1)=\pm (a'_2,b'_2)$.  It follows that $a_1x+b_1,a_2x+b_2$ are both integer multiples of $a_1'x+b_1'$.  For sufficiently large $x$, we have $a_1'x+b_1'>1$ and hence $\gcd(a_1x+b_1,a_2x+b_2) \geq a_1'x+b_1' \neq 1$, so the conclusion of the lemma holds with $M=1$ (say).

Now, suppose that $a_1b_2-a_2b_1 \neq 0$.  Then $\gcd(a_1x+b_1,a_2x+b_2)$ depends only on the residue of $x$ modulo $M:=|a_1b_2-a_2b_1|$, as desired.
\end{proof}

We are finally ready to prove Theorem~\ref{thm:relative-spectrum-improved}.

\begin{proof}[Proof of Theorem~\ref{thm:relative-spectrum-improved}]
If $U$ is contained in some subspace of the form $\{x_i=\epsilon x_j\}$, then the projection $\pi(U)$ onto all but the $i$-th coordinate clearly satisfies
$$\mathcal{S}_1(U)=\mathcal{S}_1(\pi(U)),$$
and we can apply induction on $n$.  Now assume that $U$ is not contained in any subspace of the form $\{x_i=\epsilon x_j\}$, and apply Proposition~\ref{prop:D(T_i,j,eps)-formula}.  In the language of Proposition~\ref{prop:D(T_i,j,eps)-formula}, fix a choice of residues $0 \leq \aleph, \beth<M'$, and assume that $(A,B) \in \mathcal{T}$ has $A^2+B^2$ sufficiently large and satisfies $A \equiv \aleph \pmod{M'}$ and $B \equiv \beth \pmod{M'}$.

We start with the pairs $(A,B)$ not lying on $\mathcal{L}$: Fix  a sector $\sigma$ from Proposition~\ref{prop:D(T_i,j,eps)-formula}, and assume that $(A,B) \in (\mathcal{T} \setminus \mathcal{L}) \cap \sigma$.  For such $(A,B)$, there are $\kappa \in \{0,1\}$ and rationals $E',F'$ such that
$$D(T)=D(U)+\frac{\kappa}{E'A+F'B}.$$
If $\kappa=0$, then $D(T)=D(U)$ contributes only a single value to $\mathcal{S}_1(U)$.  Now suppose that $\kappa=1$.  Write $E'=y/w$ and $F'=z/w$ where $y,z$ are coprime integers and $w$ is a rational number, so that
$$\frac{1}{E'A+F'B}=\frac{w}{yA+zB}.$$
Notice that $yA+zB \equiv y\aleph+z\beth \pmod{M}$ for all $(A,B)$ under consideration.

For $s'$ an integer, define the level set
$$V(s'):=\{(A',B') \in ((\aleph+M'\mathbb{Z}) \times (\beth+M'\mathbb{Z})) \cap \sigma: yA'+zB'=s'M'+y\aleph+z\beth\}.$$
We have not required $A',B'$ to be coprime in the definition of $V(s')$; the expression $yA+zB$ assumes the value $s'M'+y\aleph+z\beth$ if and only if $V(s')$ contains a point of $\mathcal{T} \setminus \mathcal{L}$.  For $s'$ sufficiently large, the set $V(s')$ intersects each half-line of $\mathcal{L}$ in at most one point, so $V(s')$ contains a point of $\mathcal{T} \setminus \mathcal{L}$ whenever $V(s')$ contains a sufficiently large number of points with coprime coordinates.  Since $y,z$ are coprime, there are natural numbers $b'_1,b'_2$ such that $yb'_1+zb'_2=1$.  Hence $$y(s'M'b'_1+\aleph)+z(s'M'b'_2+\beth)=s'M'+y\aleph+z\beth,$$
and $V(s')$ consists of the points
$$(zM'x+s'M'b'_1+\aleph,-yM'x+s'M'b'_2+\beth)$$
for $x$ in an appropriate interval $I(s')$ of integers.  See Figure~\ref{fig:level-sets} for an illustration.

\begin{figure}
    \centering
    \begin{tikzpicture}[scale=0.55]
    \def\circlesize{3pt}
    \fill[gray!6] (0,0) -- (9.5, 9.5/3) -- (9.5, 11.5) --(5.75,11.5) -- cycle;

    \foreach \x/\y in { 1/1, 1/2, 2/1, 2/3, 3/1, 3/2, 3/4, 3/5, 4/3, 4/5, 4/7, 5/2, 5/3, 5/4, 5/6, 5/7, 5/8, 5/9, 6/5, 6/7, 6/11, 7/3, 7/4, 7/5, 7/6, 7/8, 7/9, 7/10, 7/11, 8/3, 8/5, 8/7, 8/9, 8/11, 9/4, 9/5, 9/7, 9/8, 9/10, 9/11} {
        \fill[black] (\x, \y) circle (\circlesize);
    }

        \foreach \x/\y in {2/2, 2/4, 3/3, 3/6, 4/2, 4/4, 4/6, 4/8, 5/5, 5/10, 6/2, 6/3, 6/4, 6/6, 6/8, 6/9, 6/10, 7/7, 8/4, 8/6, 8/8, 8/10, 9/3, 9/6, 9/9} {
        \draw[black] (\x, \y) circle (\circlesize);
    }

    \draw[thick,->] (-1, 0) -- (10, 0) node[right] {$A$};
    \draw[thick,->] (0, -1) -- (0, 11.5) node[above] {$B$};
    \draw[thick] (0, -0) -- (5.75, 11.5);
    \draw[thick] (0, -0) -- (9.5, 9.5/3);
    \draw[gray, thick] (1/3,2/3)--(3/4,1/4);
    \draw[gray, thick] (2/3,4/3)--(6/4,2/4);
    \draw[gray, thick] (1,2)--(9/4,3/4);
    \draw[gray, thick] (4/3,8/3)--(12/4,1);
    \draw[gray, thick] (5/3,10/3)--(15/4,5/4);
    \draw[gray, thick] (2,4)--(18/4,6/4);
    \draw[gray, thick] (7/3,14/3)--(21/4,7/4);
    \draw[gray, thick] (8/3,16/3)--(6,2);
    \draw[gray, thick] (9/3,18/3)--(27/4,9/4);
    \draw[gray, thick] (10/3,20/3)--(30/4,10/4);
    \draw[gray, thick] (11/3,22/3)--(33/4,11/4);
    \draw[gray, thick] (12/3,24/3)--(36/4,12/4);
    \draw[gray, thick] (13/3,26/3)--(9.5,3.5);
    \draw[gray, thick] (14/3,28/3)--(9.5,4.5);
    \draw[gray, thick] (15/3,30/3)--(9.5,5.5);
    \draw[gray, thick] (16/3,32/3)--(9.5,6.5);
    \draw[gray, thick] (17/3,34/3)--(9.5,7.5);
    \draw[gray, thick] (6.5,11.5)--(9.5,8.5);
    \draw[gray, thick] (7.5,11.5)--(9.5,9.5);
    \draw[gray, thick] (8.5,11.5)--(9.5,10.5);
    
\end{tikzpicture}
\caption{This figure illustrates how level sets intersect a sector.  Here, the sector $\sigma=\{x/3 \leq y \leq 2x\}$ is shaded.  Filled-in circles indicate elements of $\mathcal{T} \cap \sigma$ (i.e., points with coprime coordinates), and outlined circles indicate other lattice points in $\sigma$.  The gray diagonal lines with slope $-1$ indicate the various level sets $V(s')$ of the expression $A+B$ (where there are no modular constraints placed on $A,B$).}
    \label{fig:level-sets}
\end{figure}

Notice that $I(s')$ is empty if $s'M'+y\aleph+x\beth<0$.  The argument from the last paragraph of the proof of Proposition~\ref{prop:D(T_i,j,eps)-formula} shows that the quantity $E'A+F'B$ tends to infinity as $A^2+B^2$ tends to infinity.  In particular, each $I(s')$ is finite, and the length of $I(s')$ grows linearly with $s'$ as $s'$ tends to infinity.  We wish to apply Lemma~\ref{lem:coprime} with the parameters
$$I:=I(s'), \quad a_1:=zM', \quad a_2:=-yM', \quad b_1:=s'M'b'_1+\aleph, \quad b_2:=s'M'b'_2+\beth;$$
the parameter $N:=|I|$ grows linearly with $s'$ (as observed above), and the constant $C$ in the statement of Lemma~\ref{lem:coprime} can be chosen appropriately (depending on $M',y,z,b'_1,b'_2$).  We check that
$$a_1b_2-a_2b_1=(M')^2 s'+M'(y\aleph+z\beth)$$
is nonzero for $s'$ sufficiently large.  It remains to check whether or not $\gcd(a_1,a_2,b_1,b_2)=1$.  It is always the case that $\gcd(a_1,a_2,b_1,b_2)$ divides $\gcd(a_1,a_2)=M'$.  In particular, whether or not $\gcd(a_1,a_2,b_1,b_2)=1$ depends only on the residue class of $s'$ modulo $M'$.  Fix some $0 \leq r<M'$, and suppose that $s'=M's+r$ for $s$ an integer.  If $r$ is such that $\gcd(a_1,a_2,b_1,b_2)>1$, then $V(s')$ does not contain any points with coprime coordinates.  Suppose instead that $r$ is such that $\gcd(a_1,a_2,b_1,b_2)=1$, in which case Lemma~\ref{lem:coprime} ensures that the number of points of $V(s')$ with coprime coordinates tends to infinity with $s'$.  By the discussion in the previous paragraph, this ensures that $D(T)$ assumes all of the values
$$D(U)+\frac{w}{(M')^2 s+M'r+y\aleph +z\beth}$$
for $s$ sufficiently large.  Setting $\alpha:=(M')^2/w$ and $\beta:=(M'r+y\aleph +z\beth)/w$, we find that $D(T)$ assumes all but finitely many elements of the set
$$D(U)+\frac{1}{\Prog(\alpha,\beta)}.$$
We obtain only finitely many such progressions because there were finitely many choices of $\aleph,\beth, \sigma,r$.

It remains to treat the exceptional half-lines in $\mathcal{L}$.  Fix a half-line $L \in \mathcal{L}$ from Proposition~\ref{prop:D(T_i,j,eps)-formula}, and consider pairs $(A,B) \in \mathcal{T} \cap L$, still with $A^2+B^2$ sufficiently large and with $A \equiv \aleph \pmod{M'}$ and $B \equiv \beth \pmod{M'}$.  As above, Proposition~\ref{prop:D(T_i,j,eps)-formula} provides $\kappa \in \{0,1\}$ and rationals $E',F'$ such that
$$D(T)=D(U)+\frac{\kappa}{E'A+F'B}$$
for such $(A,B)$.  We again restrict our attention to the case $\kappa=1$ and clear denominators, after which we can immediately apply Lemma~\ref{lem:coprime-on-a-line} instead of Lemma~\ref{lem:coprime}; this leads, as before, to finitely many progressions of the form $D(U)+1/\Prog(\alpha,\beta)$.  This completes the proof of Theorem~\ref{thm:relative-spectrum-improved}.
\end{proof}

\subsection{Remarks on the proof}\label{sec:remarks}
We now briefly mention a few further pieces of information that can be extracted from the proof of Theorem~\ref{thm:relative-spectrum-improved}; some of this additional structure will streamline the calculations in the remaining sections of the paper.

\subsubsection{Sectors and exceptional half-lines}\label{sec:sectors-and-lines}
In the proof above, the relative spectrum $\mathcal{S}_1(U)$ is obtained as the union of contributions from sectors and contributions from exceptional half-lines.  One can show (as follows) that in fact the progressions in $\mathcal{S}_1(U)$ come either entirely from sectors or entirely from exceptional half-lines, according to whether $U$ achieves its $D$-value at finitely or infinitely many points.  This observation significantly shortens calculations in concrete examples.

First, suppose that $U$ achieves its $D$-value at infinitely many points.  (This is always the case, for instance, when $U$ is of the form $U=U' \times (\mathbb{R}/\mathbb{Z})$.)  Then there is some quadruple $(i,j,\epsilon,\ell)$ such that $D(U_{i,j,\epsilon,\ell})=D(U)$ and $D_{i,j,\epsilon,\ell}$ is equal to $D(U)$ on an interval of positive length.  It follows that $D(T)=D(U)$ except when $(A,B)$ lies on one of the exceptional half-lines for the quadruple $(i,j,\epsilon,\ell)$, so the entire relative spectrum $\mathcal{S}_1(U)$ comes from these half-lines.  In particular, we do not have to consider contributions from the various sectors.  For an example, see Section~\ref{sec:s1(4)-first-example}.

Second, suppose that $U$ achieves its $D$-value at only finitely many points.  We claim that the contribution to $\mathcal{S}_1(U)$ from exceptional half-lines is already contained in the contribution from sectors and hence can be ignored.  It suffices to show that the first conclusion of Proposition~\ref{prop:D(T_i,j,eps)-formula} continues to hold on exceptional half-lines.  Fix a choice of residues of $A,B$ modulo $M'$ as in the proof of that proposition.  Let $L$ be an exceptional half-line, and let $Y''$ denote the set of quadruples $(i,j,\epsilon,\ell)$ for which $L$ is exceptional; these are precisely the quadruples for which the corresponding value $q_{i,j,\epsilon,\ell}=|T_{i,j,\epsilon,\ell}|$ is constant on $L$.  It follows that for each $(i,j,\epsilon,\ell) \in Y''$, the quantity $D(T_{i,j,\epsilon,\ell})$ is equal to some constant, say, $D(T_{i,j,\epsilon,\ell})=D(U)+\eta_{i,j,\epsilon,\ell}$, for all $(A,B) \in L$.  If $\eta_{i,j,\epsilon,\ell}=0$ for some $(i,j,\epsilon,\ell) \in Y''$, then $D(T)=D(U)$ and $L$ contributes only the single value $D(U)$ to $\mathcal{S}_1(U)$.

Now consider the case where $\eta_{i,j,\epsilon,\ell}>0$ for every $(i,j,\epsilon,\ell) \in Y''$.  For $(i,j,\epsilon,\ell) \in Y''$, the parameter $q=q_{i,j,\epsilon,\ell}$ (which is constant) may be too small for the formula \eqref{eq:D-for-T_i,j,eps,ell} to hold; in particular, the formula \eqref{eq:D-for-T_i,j,eps,ell} may erroneously predict that $D(T_{i,j,\epsilon,\ell})=D(U)+\eta'_{i,j,\epsilon,\ell}$ for some $\eta_{i,j,\epsilon,\ell}'>0$ that is different from $\eta_{i,j,\epsilon,\ell}$.  We know (see \cite{vikram}) that $D(T) \to D
(U)$ as $A^2+B^2 \to \infty$.  Thus, for $A^2+B^2$ sufficiently large, the first minimum over $Y$ in the proof of Proposition~\ref{prop:D(T_i,j,eps)-formula} must be achieved by some quadruple $(i,j,\epsilon,\ell) \in Y$ with $D(T_{i,j,\epsilon,\ell}) \to D(U)$.  In particular, this minimum is achieved on $Y \setminus Y''$, and it makes no difference if the $\eta$'s are replaced by the $\eta'$'s on $Y''$.  The rest of the proof then goes through verbatim.

\subsubsection{Other symmetries and reductions} \label{sec:other-symmetries}
In the proof of Proposition~\ref{prop:D(T_i,j,eps)-formula}, we apply Lemma~\ref{lem:approximating-by-1/q} (via Lemma~\ref{lem:D-on-a-slice}) for each choice of $(i,j,\epsilon,\ell), \tau,\delta$, with $q:=\delta(EA+FB)$ and $b:=\delta a\ell/K$ for some parameters $E,F,a,K$ depending on $(i,j,\epsilon,\ell)$ and the residues of $A,B$ modulo $M''$.  Lemma~\ref{lem:approximating-by-1/q} then provides formulas for the residues of $R^-,R^+$ modulo $M$ as certain linear combinations of $A,B$ (with coefficients depending on all of the parameters introduced so far).  In these formulas, changing the sign $\delta$ changes the signs of $q,b$ (as linear combinations of $A,B$), which has the effect of swapping $R^-,R^+$.  In particular, the sign $\delta$ does not affect whether or not $\min\{R^-,R^+\}=0$ or, by extension, whether or not $D(T_{i,j,\epsilon,\ell})=0$.  Thus, when we make our sector decomposition in the proof of Proposition~\ref{prop:D(T_i,j,eps)-formula}, either $D(T)=D(U)$ on all sectors or $D(T)>D(U)$ on all sectors: If $D(T)=D(U)$ on some sector, then the $(i,j,\epsilon,\ell),\tau$ realizing this $D$-value (corresponding to $\min\{R^-,R^+\}=0$) will also realize the same $D$-value on the rest of $\mathbb{R}_{\geq 0} \times \mathbb{R}$. This can be seen in Figure~\ref{tab:sectors-best-approx-u2}.

We conclude with a few minor remarks about conventions and computational redundancies.  First, in the conclusion of Lemma~\ref{lem:approximating-by-1/q}, we can divide through all of $R^-,R^+,xz$ by the common factor $\gcd(x,z)$; this reduces the length of some of our later numerical calculations.
Second, in Section~\ref{sec:main-proof}, we restricted $(A,B)$ to the half-space $\{A \geq 0\}$, but any other half-space would work equally well; sometimes in examples it will be convenient to restrict to different half-spaces.

Third, we observe that since the function $D$ on $(\mathbb{R}/\mathbb{Z})^n$ is an even function, its pushforward by each $\psi=\psi_{i,j,\epsilon}$ is also even.  It follows that when we apply Lemmas~\ref{lem:approximating-by-1/q} and~\ref{lem:D-on-a-slice} in the proof of Proposition~\ref{prop:D(T_i,j,eps)-formula}, the critical values of $\tau$ come in pairs (except for $2$-torsion points).  More precisely, suppose that $U_{i,j,\epsilon}$ has $K$ components.  Then, in the notation of the proof of Proposition~\ref{prop:D(T_i,j,eps)-formula}, we have $D_{i,j,\epsilon,\ell}(x)=D_{i,j,\epsilon,-\ell}(-x)$ (with the index $-\ell$ interpreted modulo $K$).  Since the subgroup $T_{i,j,\epsilon}$ is also obviously invariant under the map $x \mapsto -x$, we see that $D_{i,j,\epsilon,\ell}(\tau_0)=D(U)$ if and only if $D_{i,j,\epsilon,-\ell}(-\tau_0)=D(U)$; in this case, the best approximations to these two minima by elements of $\psi( T_{i,j,\epsilon})$ make the same contribution to the expression of $D(T_{i,j,\epsilon})$ as a minimum over approximations to all critical points $\tau$.  It follows that in order to compute $D(T_{i,j,\epsilon})$ (and, by extension, $D(T)$), it suffices to consider one point from each pair described in the previous sentence.  For example, we could consider only values of $\tau$ in $[0,1/2]$; or we could consider only values of $\tau$ coming from $0 \leq \ell \leq k/2$ (and, for $\ell \in \{0,k/2\}$, look at only $\tau \in [0,1/2]$).  We will use this reduction in all of our computational examples below.

\section{Identifying candidate subtori $U$}\label{sec:identifying-candidates-u}
We now turn to our computational results Theorems~\ref{thm:S_1(3)-second-acc-point} and~\ref{thm:S_1(4)-first-acc-point}, which characterize $\mathcal{S}_1(3) \cap (1/10, 1/6]$ and $\mathcal{S}_1(4) \cap (1/4, 1/2]$. The first step is identifying the relevant $2$-dimensional subtori whose relative spectra contribute to these sets.  In particular, we must find all proper subtori $U \subseteq (\mathbb{R}/\mathbb{Z})^3$ satisfying $\dim(U)=2$ and $D(U)=1/10$, and all proper subtori $U \subseteq (\mathbb{R}/\mathbb{Z})^4$ satisfying $\dim(U)=2$ and $D(U)=1/4$.  We will begin with some general considerations that aid the identification of $2$-dimensional subtori with $D$-value above a given threshold.


\subsection{General setup}
We can reduce the length of our calculations by taking advantage of natural symmetries in the Lonely Runner Problem. The function $D$ on $(\mathbb{R}/\mathbb{Z})^n$ is preserved by the automorphisms of $(\mathbb{R}/\mathbb{Z})^n$ given by negating and permuting coordinates. These automorphisms also permute the proper subtori of $(\mathbb{R}/\mathbb{Z})^n$, so, for the purposes of this section, it suffices to consider a single subtorus $U$ in each orbit of this automorphism group; we will say that subtori in the same orbit are equivalent ``up to symmetry''.

Suppose that $U \subseteq (\mathbb{R}/\mathbb{Z})^n$ is a $2$-dimensional proper subtorus. Then there are nonparallel vectors $u=(u_1,\ldots,u_n),v=(v_1,\ldots,v_n) \in \mathbb{Z}^n$ such that $U=\langle u,v \rangle_\mathbb{R}$. The linear independence of $u,v$ implies that there are indices $i<j$ such that $(u_i,u_j),(v_i,v_j) \in \mathbb{Z}^2$ are linearly independent; up to symmetry, we may assume that $(i,j)=(1,2)$.  Then, by replacing $u,v$ with $$(v_2-v_1)u+(u_1-u_2)v, \quad (-v_2-v_1)u+(u_2+u_1)v,$$
we can assume that $u_1=u_2$ and $v_1=-v_2$. For the remainder of this section, we will work with subtori $U$ of the form
\begin{equation}\label{eq:U-symmetrized-form}
U=\langle (u_1, u_1, u_3, \ldots, u_n),(v_1,-v_1,v_3,\ldots,v_n) \rangle_\mathbb{R},
\end{equation}
where $u_1,u_3, \ldots, u_n, v_1,v_3, \ldots, v_n \in \mathbb{Z}$.  We will split our analyses into several cases depending on which of these entries are zero.

\subsection{Subtori $U \subset (\mathbb{R}/\mathbb{Z})^3$ with $D(U) = 1/10$.} \label{sec:u-to-check-1/10}
In this section we identify all $2$-dimensional proper subtori $U \subseteq (\mathbb{R}/\mathbb{Z})^3$ with $D(U) = 1/10$. We will show that, up to symmetry, the only such subtori are \begin{align*}
    U^3:= \langle (0, 1, 4), (1, 0, 0) \rangle_\mathbb{R}, &\quad U^4:= \langle (0, 2, 3), (1, 0, 0) \rangle_\mathbb{R},\\ 
    U^5:= \langle (0, 1, 3), (1, 0, 1) \rangle_\mathbb{R}, &\quad  U^6:= \langle (0, 1, 2), (1, 1, 0) \rangle_\mathbb{R}.
\end{align*}

The arguments in this subsection generalize in a straightforward way to let one find all $2$-dimensional proper subtori $U \subset (\mathbb{R}/\mathbb{Z})^3$ with $D(U) \geq M$ for any fixed $M>0$. The only modification is that the list in Equation \eqref{eq:list-of-speeds-3d} below must be enlarged to contain all pairs $\{x,2r+1-x\}$ with $4r+2 \leq 1/M$ and $1 \leq x \leq r$ ($r,x \in \mathbb{N}$).

For the remainder of this subsection, we assume that $U$ is a subtorus of the form
$$U=\langle (a, a, b), (c, -c, d) \rangle_\mathbb{R}$$
(with $a,b,c,d \in \mathbb{Z}$) such that $D(U)= 1/10$.  Without loss of generality we may assume that $\gcd(a,b)=\gcd(c,d)=1$.  We condition on which of $a,b,c,d$ vanish.

\subsubsection{No zero entries} Suppose first that $a, b, c,d$ are all nonzero.  Note that
$$1/10= D(U) \leq D(\langle (a,a,b) \rangle_\mathbb{R})=D(\langle (|a|,|b|) \rangle_\mathbb{R}).$$
Recall from~\cite[Theorem 4.1]{thesis} that $D(\langle (|a|,|b|) \rangle_\mathbb{R})$ equals $0$ if $a,b$ are both odd and equals $1/2(a+b)$ if $a,b$ have different parities.  Hence we must have
\begin{equation} \label{eq:list-of-speeds-3d}
    \{|a|, |b|\} \in \{\{1, 2\}, \{1, 4\}, \{2, 3\}\},
\end{equation}
and likewise for $\{|c|,|d|\}$.

This leads to finitely many possibilities for $U$, and we can compute $D(U)$ for each.  We conclude that, up to symmetry, $U$ must be equal to one of the subtori
\begin{align*}
    &\langle (1, 1, 4), (-1, 1, 4)\rangle_\mathbb{R} = \langle (0, 1, 4), (1, 0, 0) \rangle_\mathbb{R}=U^3, \\ 
    &\langle (2, 2, 3), (-2, 2, 3)\rangle_\mathbb{R} = \langle (0, 2, 3), (1, 0, 0) \rangle_\mathbb{R}=U^4, \\
    &\langle (1, 1, 4), (1, -1, -2)\rangle_\mathbb{R} = \langle (0, 1, 3), (1, 0, 1) \rangle_\mathbb{R}=U^5,
\end{align*}
as desired.

\subsubsection{A single zero entry} Next, suppose that exactly one of $a,b,c,d$ vanishes.  If $a=0$, then we have
$$D(U)=D(\langle (0,0,1),(c,-c,d) \rangle_\mathbb{R})=D(\langle (c,-c) \rangle_\mathbb{R})=0,$$
which contradicts the assumption that $D(U)=1/10$.  So we may assume that $a \neq 0$, and likewise that $c \neq 0$.  The remaining case (up to symmetry) is where $b=0$.  Then
$$U=\langle (1,1,0),(c,-c,d) \rangle_\mathbb{R},$$
where we know that $\{|c|, |d|\} \in \{\{1, 2\}, \{1, 4\}, \{2, 3\}\}$ (by the argument above).  This again leads to only finitely many possibilities for $U$, and computing $D(U)$ for each leads to the conclusion that, up to symmetry, $U$ must be equal to
\begin{align*}
    \langle (1, 1, 0), (1, -1, -4)\rangle = \langle (0, 1, 2), (1, 1, 0) \rangle_\mathbb{R}=U^6.
\end{align*}
This concludes the second subcase.

\subsubsection{Two or more zero entries} 
Finally, suppose that at least two of $a,b,c,d$ vanish.  Since $(a,a,b),(c,-c,d)$ are both nonzero, we see that exactly two of $a,b,c,d$ vanish.  We cannot have $a=c=0$ or $b=d=0$ since we assumed that $U$ is proper.  Up to symmetry, it remains only to consider $a=d=0$, in which case
$$D(U)=D(\langle (0,0,1),(1,-1,0) \rangle_\mathbb{R})=0$$
gives a contradiction.  This completes the proof.


\subsection{Subtori $U \subset (\mathbb{R}/\mathbb{Z})^4$ with $D(U) = 1/4$}

In this subsection, we will describe all proper subtori $U \subseteq (\mathbb{R}/\mathbb{Z})^4$ with $\dim(U) \geq 2$ and $D(U)= 1/4$. More precisely, we will show that, up to symmetry, the only such subtori are $$U^1:=\langle (0,1,2,3),(1,0,0,0) \rangle_{\mathbb{R}}, \quad  U^2 := \langle (1,0,1,1),(1,1,0,2) \rangle_{\mathbb{R}}.$$
We remark that $U^1$ already appeared in \cite{thesis} but $U^2$ would have been more difficult to guess.

We do not have to consider subtori of dimension larger than $2$ because \cite{vikram} tells us that $\mathcal{S}_3(4)=\mathcal{S}_1(2)$ has maximum value $1/6$.   For the remainder of this subsection, we assume that $U$ is a subtorus of the form
$$U=\langle (a, a, b,c), (d, -d, e,f) \rangle_\mathbb{R}$$
(with $a,b,c,d,e,f \in \mathbb{Z}$) such that $D(U)=1/4$.

We will repeatedly use the fact that if $a,b,c$ are all nonzero, then
$$1/4= D(U) \leq D(\langle (a,a,b,c)\rangle_{\mathbb{R}})=D(\langle (|a|,|b|,|c|)\rangle_{\mathbb{R}}).$$
We may assume without loss of generality that $\gcd(|a|,|b|,|c|)=1$, and then the characterization of equality for the Lonely Runner Conjecture with $3$ runners \cite[Section 5]{thesis} implies that $$\{|a|,|b|,|c|\}=\{1,2,3\};$$
this is useful because it completely determines one of the generators of $U$.

We also note that if $d=0$ (i.e., $U$ is contained in the subspace $\{x_1=x_2\}$), then we have
$$1/4= D(U)=D(\langle (a,a,b,c),(0,0,e,f)\rangle_{\mathbb{R}})=D(\langle (a,b,c),(0,e,f)\rangle_{\mathbb{R}}) \in \mathcal{S}_2(3).$$
But this is impossible since the largest value of $\mathcal{S}_2(3)=\mathcal{S}_1(2)$ is $1/6$, so we can always assume that $d \neq 0$.  The same argument lets us assume that $a \neq 0$ (and more generally that $U$ is not contained in any subspace $\{x_i=\epsilon x_j\}$).  We proceed by a case analysis as in the previous subsection.

\subsubsection{No zero entries} 
Suppose that $a,b,c,d,e,f$ are all nonzero.  Then, after possibly scaling the generators of $U$, we have $$\{|a|,|b|,|c|\}=\{|d|,|e|,|f|\}=\{1,2,3\}.$$
These finitely many possibilities lead, up to symmetry, to only
$$\langle (1, 1, 2, 3), (-1, 1,  2,  3)\rangle = \langle (0, 1, 2, 3), (1, 0,  0,  0) \rangle_\mathbb{R}=U^1.$$

\subsubsection{A single zero entry}
Suppose that exactly one of $a,b,c,d,e,f$ vanishes.  Up to symmetry, the only possibility is $b=0$ (recall that $a,d \neq 0$).  After possibly scaling the second generator of $U$, we may assume that $\{|d|,|e|,|f|\}=\{1,2,3\}$; fix one such choice of values for $d,e,f$.  A change of basis, with an eye on the third and fourth coordinates, gives
$$U = \langle (cd + a(e-f), -cd + a(e-f), ce, ce), (cd - a(e+f), -cd - a(e+f), ce, -ce)\rangle_\mathbb{R}.$$
This change of basis has determinant $2ce$, which is nonzero by the assumption that $c,e \neq 0$.  If all of the entries of these new generators are nonzero, then we land in the case that we already considered in the first subsection (with the third and fourth coordinates playing the role of the first and second coordinates).  It remains to consider the case where some entry vanishes. Recall that $d,e,f$ are already fixed. Since each entry is a linear combination of $a,c$, the vanishing of an entry determines the ratio of $a,c$, which in turn determines the vector $(a,a,0,c)$ up to scaling. Computing $D(U)$ for each possibility, we find that, up to symmetry, $U$ must equal \[\langle (1, 1, 2, 3), (1, -1, 0. -1) \rangle_\mathbb{R} = \langle (1, 0, 1, 1), (1, 1, 0, 2) \rangle_\mathbb{R}=U^2.\]

\subsubsection{Two or more zero entries} 
Suppose that at least two of $a,b,c,d,e,f$ vanish, and recall from above that we may assume that $a,d, \neq 0$.  We cannot have $b=e=0$ or $c=f=0$ because of the assumption that $U$ is proper.  So exactly two of $b,c,e,f$ vanish.  If $b=c=0$, then we have
$$U=\langle (a,a,0,0),(d,-d,e,f) \rangle_{\mathbb{R}}=\langle (1,1,0,0),(d,-d,e,f) \rangle_{\mathbb{R}},$$
and we can take $\{|d|,|e|,|f|\}=\{1,2,3\}$; this leads to finitely many possibilities for $U$, none of which leads to any new possibilities for $U$ besides $U^1$ and $U^2$. The case $e=f=0$ is identical.  

It remains to consider the cases where $b=f=0$ and where $c=e=0$; these are equivalent up to symmetry, so we will consider only the former. Suppose that  $b=f=0$, so that
$$U=\langle (a,a,0,c),(d,-d,e,0) \rangle_{\mathbb{R}}.$$
Replacing these two generators with linear combinations to make the third and fourth coordinates more agreeable yields
$$U = \langle (cd + ae, -cd + ae, ce, ce), (cd - ae, -cd - ae, ce, -ce) \rangle_{\mathbb{R}}$$
(again this change of coordinates has determinant $2ce\neq 0$).
If all of the entries of these new generators are nonzero, then we land in the case that we already considered in the previous subsection, with the third and fourth coordinates playing the role of the first and second coordinates.

It remains to consider the case where some entry vanishes, i.e., $cd=\pm ae$ (recall $ce \neq 0$).  These cases are again equivalent up to symmetry, so suppose that the former occurs and
$$U=\langle (a,a,0,c),(a,-a,c,0) \rangle_{\mathbb{R}}.$$
Another change of basis, this time focusing on the first and third coordinates, gives
$$U=\langle (ac,ac-2a^2,ac,c^2-ac),(ac,ac+2a^2,-ac,ac+c^2) \rangle_{\mathbb{R}}$$
(this change of coordinates has determinant $-2ac \neq 0$).
We are done if all of the entries of these new generators are nonzero, so it remains only to check the cases where some entry vanishes, viz., $c= \pm 2a$ or $c= \pm a$ (recall $a,c \neq 0$).  Since we are free to rescale the generators of $U$, we may assume that $a=1$.  Then each of the four cases gives a single value for $c$; this in turn determines the ratio $d/e=a/c$, and so we can take $d=a$ and $e=c$.  Computing $D(U)$ in each case, we do not find any new possibilities for $U$ besides $U^1,U^2$. This completes the proof.

\section{The top of the spectrum $\mathcal{S}_1(4)$} \label{sec:explicit-comp-s1(4)}

As an illustration of Theorem~\ref{thm:relative-spectrum-improved}, we will now prove Theorem~\ref{thm:S_1(4)-first-acc-point}, which states that the set $\mathcal{S}_1(4) \cap (1/4, 1/2]$ has finite symmetric difference with the set $1/4 + 1/4\Prog(2, 3)$.  In the previous section we showed that, up to symmetry, 
$$U^1 = \langle (0, 1, 2, 3), (1, 0, 0, 0) \rangle_\mathbb{R} \quad \text{and} \quad U^2 = \langle(1, 0, 1, 1), (1, 1, 0, 2) \rangle_{\mathbb{R}}$$
are the only $2$-dimensional subtori $U \subseteq (\mathbb{R}/\mathbb{Z})^4$ with $D(U)= 1/4$. Now \eqref{eq:acc-union-of-relative-spectra} tells us that to prove Theorem~\ref{thm:S_1(4)-first-acc-point}, it remains only to use Theorem~\ref{thm:relative-spectrum-improved} to analyze $\mathcal{S}_1(U^1)$ and $\mathcal{S}_1(U^2)$ individually.  Notice that these choices of generators for $U^1,U^2$ satisfy the condition \eqref{eq:lattice-condition}.

\subsection{Computing $\mathcal{S}_1(U^1)$}\label{sec:s1(4)-first-example}
The objective of this subsection is to prove the following proposition.
\begin{proposition}\label{prop:U_1-calculation}
The set $\mathcal{S}_1(U^1)$ has finite symmetric difference with the set $1/4+1/4\Prog(4,5)$.
\end{proposition}

We parameterize $1$-dimensional subtori of $U^1$ as
 \[T = \langle A(0, 1, 2, 3)+B(1, 0,  0,  0)\rangle_\mathbb{R} =  \langle (B, A, 2A, 3A) \rangle_\mathbb{R},\] 
for coprime integers $A,B$ with $B \geq 0$ (see Section~\ref{sec:other-symmetries}).  We begin by computing the intersections $U_{i, j, \epsilon}= U^1 \cap \{x_i=\epsilon x_j\}$; see Figure~\ref{fig:U-i-j-eps}.
\begin{figure}\label{fig:U-i-j-eps}
  \centering
  \caption{Intersections of $U^1$ with the subspaces $\{x_i = \epsilon x_j\}$. The last column gives (where applicable) the values of $\ell$ for which $D(U_{i,j,\epsilon,\ell}) = 1/4$ and the values of $\tau$ where this $D$-value is achieved; we write ``$\ell: S$'' to indicate that $U_{i,j,\epsilon,\ell}$ achieves the $D$-value $1/4$ on the set $S$ (and we omit $\ell$ when $U_{i, j \epsilon}$ is connected). In the case of the antepenultimate row, the $D$-value $1/4$ is achieved by $\tau \in [1/4,3/4]$ for both $\ell=1$ and $\ell=3$.  \label{fig:U-i-j-eps}}
  \label{tab:intersection-planes-u1}
  \begin{tabular}{llll}
    \toprule
    Subspace & \textbf{$U_{i,j,\epsilon}$} & \textbf{$D(U_{i,j,\epsilon})$} & $\tau$ achieving $1/4$\\
    \midrule
    $\{x_1 = x_2\}$ & $\langle(1, 1, 2, 3)\rangle_{\mathbb{R}}$ & $1/4$ & $\{1/4, 3/4\}$\\
    $\{x_1 = -x_2\}$ & $ \langle(-1, 1, 2, 3)\rangle_{\mathbb{R}}$ & $1/4$ & $\{1/4, 3/4\}$\\
    $\{x_1 = x_3\}$ & $\langle(2, 1, 2, 3)\rangle_{\mathbb{R}}$ & $1/4$ & $\{1/4, 3/4\}$\\
    $\{x_1 = -x_3\}$ & $\langle(-2, 1, 2, 3)\rangle_{\mathbb{R}}$ & $1/4$ & $\{1/4, 3/4\}$\\
    $\{x_1 = x_4\}$ & $\langle(3, 1, 2, 3)\rangle_{\mathbb{R}}$ & $1/4$ & $\{1/4, 3/4\}$\\
    $\{x_1 = -x_4\}$ & $\langle(-3, 1, 2, 3)\rangle_{\mathbb{R}}$ & $1/4$ & $\{1/4, 3/4\}$\\
    $\{x_2 = x_3\}$ & $\langle (1, 0, 0, 0) \rangle_{\mathbb{R}}$ & $1/2$ & $-$\\
    $\{x_2 = -x_3\}$ & $\cup_{\ell = 0}^2 \langle (1, 0, 0, 0) \rangle_{\mathbb{R}} + (0, \ell/3, 2\ell/3, 0)$ & $1/2$ & $-$\\
    $\{x_2 = x_4\}$ & $\cup_{\ell = 0}^1 \langle (1, 0, 0, 0) \rangle_{\mathbb{R}} + (0, \ell/2, 0, \ell/2)$ & $1/2$ & $-$\\
    $\{x_2 = -x_4\}$ & $\cup_{\ell = 0}^3 \langle (1, 0, 0, 0) \rangle_{\mathbb{R}} + (0, \ell/4, 2\ell/4, 3\ell/4)$ & $1/4$ & $1, 3: [1/4, 3/4]$\\
    $\{x_3 = x_4\}$ & $\langle (1, 0, 0, 0) \rangle_{\mathbb{R}}$ & $1/2$ & $-$\\
    $\{x_3 = -x_4\}$ & $\cup_{\ell = 0}^4 \langle (1, 0, 0, 0) \rangle_{\mathbb{R}} + (0, \ell/5, 2\ell/5, 3\ell/5)$ & $3/10$ & $-$\\
    \bottomrule
  \end{tabular}
\end{figure}

Carrying out the reductions outlined in Section~\ref{sec:sectors-and-lines},
we quickly dispose of all pairs $(A,B)$ except for those with $|A| = 1$ and $B \equiv 0 \pmod{4}$. The key idea is to leverage the fact that the restriction of $D$ to $U_{2,4,-,1}$ and $U_{2,4,-,3}$ equals $1/4$ on intervals of positive length; we will find that $T$ inevitably intersects these intervals unless $(A,B)$ is of the special form described in the previous sentence.

Consider the isomorphism $\psi=\psi_{2,4,-}:\langle (0,1,2,3),(1,0,0,0) \rangle_\mathbb{R} \to \mathbb{R}^2$ given by $\psi(0,1,2,3)=(0,1)$ and $\psi(1,0,0,0)=(1,0)$; since $\psi^{-1}(\langle (1,0),(0,1) \rangle_\mathbb{Z})=\langle (0,1,2,3),(1,0,0,0) \rangle_\mathbb{R} \cap \mathbb{Z}^4$, the map $\psi$ descends to an isomorphism $U^1 \to (\mathbb{R}/\mathbb{Z})^2$.  In these new coordinates, we compute that
$$\psi(U_{2,4,-})=(\mathbb{R}/\mathbb{Z}) \times \langle 1/4 \rangle_\mathbb{Z}$$
and
$$\psi(T) = \psi(\langle A(0,1,2,3) + B(1,0,0,0) \rangle_\mathbb{R}) = \langle(B, A)\rangle_\mathbb{R}.$$
One can check that $D_{2,4,-}$ assumes the constant value $1/4$ on the intervals $[1/4,3/4] \times \{1/4,3/4\}$.  As described in Proposition~\ref{prop:cyclic-subgroup-parameters}, the subgroup $$\psi(T_{2,4,-})=\langle(B, A)\rangle_\mathbb{R} \cap ((\mathbb{R}/\mathbb{Z}) \times \langle 1/4 \rangle_\mathbb{Z})=\bigcup_{\ell = 0}^{3} \left(\left(\frac{B \ell}{4A} + \left\langle\frac{1}{|A|}\right\rangle_{\mathbb{Z}}\right) \times \left\{\frac{\ell}{4} \right\}\right)$$
intersects each coset $(\mathbb{R}/\mathbb{Z}) \times \{\ell/4\}$ in exactly $|A|$ equally-spaced points (since $A,B$ are coprime).  In particular, if $|A| \geq 2$, then $\psi(T_{2,4,-})$ intersects (both of) the aforementioned intervals where $D$ assumes the constant value $1/4$, which implies that $D(T) \leq 1/4$; since of course $D(T) \geq D(U^1)=1/4$, we conclude that $D(T)=1/4$, and so all such pairs $(A,B)$ contribute only the value $1/4$ to the set $\mathcal{S}_1(U^1)$.  Henceforth, we will restrict our attention to the case $|A|=1$.  In the language of Section~\ref{sec:main-proof}, the pairs $(1,B)$ and $(-1,B)$ lie on two exceptional half-lines.  Now, when $B \not\equiv 0 \pmod{4}$, we have
\begin{equation}\label{eq:U-2-4---1}
\psi(T_{2,4,-,1})=\{(B/4A,1/4)\} \in [1/4,3/4] \times \{1/4\},
\end{equation}
which implies as above that $D(T)=1/4$.  Thus it remains only to treat the case $B \equiv 0 \pmod{4}$.\footnote{We remark that the reductions in this paragraph can also be explained in terms of runners, using the language of ``pre-jumps'' (see~\cite{BGGST}).  Consider runners with speeds $B,A,2A,3A$.  At times of the form $1/(4A)+s/A$ (for $s \in \mathbb{Z}$), the runners are at the positions $B/4A+Bs/A,1/4,1/2,3/4$; the point is that only the first of these positions actually changes with $s$, since the pre-jump by $1/A$ fixes the other three positions.  The first position ranges over the set $B/4A+\langle 1/|A| \rangle_\mathbb{Z}$ as $s$ varies; in particular, it lies in $[1/4,3/4]$ for some $s$ unless $|A|=1$ and $B \equiv 0 \pmod{4}$.}

Assume that $|A|=1$, that $B \equiv 0 \pmod{4}$, and that $B \geq 0$.  The next step is computing $D(T_{i,j,\epsilon,\ell})$ for each quadruple $(i,j,\epsilon,\ell)$ with $D(U_{i,j,\epsilon,\ell})=1/4$.  It follows from~\eqref{eq:U-2-4---1} that $D(T_{2,4,-,1})=1/2$, and likewise $D(T_{2,4,-,3})=1/2$.  For each remaining $(i,j,\epsilon,\ell)$ under consideration, we have $\ell=0$ and the critical values of $\tau$ are $1/4,3/4$.  As described in Section~\ref{sec:other-symmetries}, it suffices to analyze the behavior around $\tau=1/4$.

As an example, let us write down some details of the calculation for $(i,j,\epsilon,\ell)=(1,2,+,0)$; the other five cases are completely analogous.  This time, our isomorphism $\psi=\psi_{1,2,+}$ is given by $\psi(1,1,2,3)=(1,0)$ and $\psi(1,0,0,0)=(0,1)$, so that $$\psi(T)=\psi(\langle A(0,1,2,3) + B(1,0,0,0) \rangle_\mathbb{R})=\psi(\langle A(1,1,2,3) + (B-A)(1,0,0,0) \rangle_\mathbb{R})=\langle (A,B-A) \rangle_\mathbb{R}.$$
Since $\psi(U_{1,2,+,0})=(\mathbb{R}/\mathbb{Z}) \times \{0\}$, we have
$$\psi(T_{1,2,+,0})= \langle 1/|B-A| \rangle_\mathbb{Z} \times \{0\}.$$
Notice that $B-A \geq 4-1>0$ except for the single case $(A,B)=(\pm 1,0)$, which we can safely ignore.  Then
$$\Approx^+(1/4;0,B-A)=\frac{\lceil (B-A)/4 \rceil}{B-A}-\frac{1}{4}=\frac{R^+}{4(B-A)},$$
where $R^+ \in \{0,1,2,3\}$ satisfies $R^+ \equiv A-B \pmod{4}$.  Likewise, $\Approx^-(1/4;0,B-A)=R^-/4(B-A)$, where $R^- \in \{0,1,2,3\}$ satisfies $R^- \equiv B-A \pmod{4}$.  The function $D_{1,2,+,0}(t)=D(t(1,1,2,3))$ assumes its minimum value at $\tau=1/4$ (and at $3/4$).  By computing
$$D_{1,2,+,0}(1/4+\varepsilon)=D((1/4+\varepsilon,1/4+\varepsilon, 1/2+2\varepsilon, 3/4+3\varepsilon))=1/4+\max \{-\varepsilon,3\varepsilon\},$$
we see that $\lambda^+=3$ and $\lambda^-=1$ determine the slopes of the pieces of $D_{1,2,+,0}$ directly adjacent to the point $\tau=1/4$.  (Recall that the slope of the piece to the left is defined to be $-\lambda^-$.) Figure~\ref{tab:approx-1/4-0123-1000} shows the output of these calculations, in the case $A=1$, for all of the relevant quadruples $(i,j,\epsilon,\ell)$.

\begin{figure}
  \centering
 \caption{Approximations to $\tau = 1/4$ for $U^1$ in the case $A = 1$, $B>0, B \equiv 0 \pmod{4}$.  In each case, the values of $R^+,R^-$ were calculated as the unique values in $\{0,1,2,3\}$ congruent modulo $4$ to certain integer linear combinations of $A,B$.}
  \label{tab:approx-1/4-0123-1000}
  \begin{tabular}{llllllll}
    \toprule
    $i, j, \epsilon, \ell$ & $\psi_{i, j, \epsilon}(T_{i, j, \epsilon, \ell})$ & $\lambda^+$ & $\lambda^-$ & $R^+$ & $R^-$ & $\Approx^+$ & $\Approx^-$\\
    \midrule
    $1, 2, +, 0$ & $\langle \frac{1}{|B-A|} \rangle_{\mathbb{Z}} \times \{0\}$ & 3 & 1 & $A-B \leadsto 1$ & $B-A \leadsto 3$ & $\frac{1}{4(B-1)}$ & $\frac{3}{4(B-1)}$\\
    $1, 2, -, 0$ & $\langle \frac{1}{|B+A|} \rangle_{\mathbb{Z}} \times \{0\}$ & 3 & 1 & $-B-A \leadsto 3 $ & $A+B \leadsto 1 $ & $\frac{3}{4(B+1)}$ & $\frac{1}{4(B+1)}$\\
    $1, 3, +, 0$ & $\langle \frac{1}{|B-2A|} \rangle_{\mathbb{Z}} \times \{0\}$ & 3 & 1 & $2A-B \leadsto 2 $ & $B-2A \leadsto 2 $ & $\frac{2}{4(B-2)}$ & $\frac{2}{4(B-2)}$\\
    $1, 3, -, 0$ & $\langle \frac{1}{|B+2A|} \rangle_{\mathbb{Z}} \times \{0\}$ & 3 & 1 & $-B-2A \leadsto 2 $ & $2A+B \leadsto 2 $ & $\frac{2}{4(B+2)}$ & $\frac{2}{4(B+2)}$\\
    $1, 4, +, 0$ & $\langle \frac{1}{|B-3A|} \rangle_{\mathbb{Z}} \times \{0\}$ & 3 & 1 & $3A-B \leadsto 3 $ & $B-3A \leadsto 1 $ & $\frac{3}{4(B-3)}$ & $\frac{1}{4(B-3)}$\\
    $1, 4, -, 0$ & $\langle \frac{1}{|B+3A|} \rangle_{\mathbb{Z}} \times \{0\}$ & 3 & 1 & $-B-3A \leadsto 1 $ & $3A+B \leadsto 3 $ & $\frac{1}{4(B+3)}$ & $\frac{3}{4(B+3)}$\\
    \bottomrule
  \end{tabular}
\end{figure}

When $B$ is sufficiently large, we have $$D(T_{i,j,\epsilon,\ell})=1/4+\min \{\lambda^+ \Approx^+,\lambda^- \Approx^-\}$$ for each quadruple $(i,j,\epsilon,\ell)$ from Figure~\ref{tab:approx-1/4-0123-1000}, and then $D(T)$ is the minimum of these six quantities.  The result of these comparisons is that $D(T)=1/4+1/4(B+1)$ (coming from $T_{1,2,-,0}$).  Substituting $B=4s$---note that $A,B$ are automatically coprime since $A=1$---we obtain the progression $1/4+1/4\Prog(4,5)$.  The analogous calculation for the half-line with $A=-1$ yields the same progression.  Hence $\mathcal{S}_1(U^1)$ has finite symmetric difference with $1/4+1/4\Prog(4,5)$, as desired.

\subsection{Computing $\mathcal{S}_1(U^2)$} \label{subsec:computing-s1(u2)}
Our calculation of $\mathcal{S}_1(U^2)$ is more involved and can be summarized as follows; Theorem~\ref{thm:S_1(4)-first-acc-point} is then a quick consequence since $\Prog(4,5) \subseteq \Prog(2,3)$.
\begin{proposition}\label{prop:U_2-calculation}
The set $\mathcal{S}_1(U^2)$ has finite symmetric difference with the set $1/4+1/4\Prog(2,3)$.
\end{proposition}
We parameterize $1$-dimensional subtori of $U^2$ as 
\[T = \langle A(1, 0, 1, 1) + B(1, 1, 0, 2) \rangle_\mathbb{R} = \langle (A + B, B, A, A + 2B) \rangle_\mathbb{R},\] for coprime $A$ and $B$ with $A \geq 0$.  As before, we intersect $U^2$ with the subspaces $\{x_i = \epsilon x_j\}$ to get the subgroups $U_{i, j, \epsilon}$, and then we compute the values $D(U_{i, j, \epsilon})$; see Figure~\ref{tab:intersection-planes-u2}. In contrast to the example in the previous subsection, the $D$-value $1/4$ is achieved by only finitely many points of $U^2$ (as opposed to intervals of positive length).  The considerations of Section~\ref{sec:sectors-and-lines} tell us that it will suffice to study only sectors, rather than exceptional half-lines (as in the previous subsection).  By Section~\ref{sec:other-symmetries}, it suffices to work with $\tau=1/4$.
 
 
\begin{figure}
  \centering
  \caption{Intersections of $U^2$ with the subspaces $\{x_i = \epsilon x_j\}$.}
  \label{tab:intersection-planes-u2}
  \begin{tabular}{llll}
    \toprule
    Subspace & $U_{i, j, \epsilon}$ & $D(U_{i, j, \epsilon})$ & $\tau$ achieving $1/4$\\
    \midrule
    $\{x_1 = x_2\}$ & $\langle(1, 1, 0, 2)\rangle_{\mathbb{R}}$ & 1/2 & $-$\\
    $\{x_1 = -x_2\}$ & $ \langle(1, -1, 2, 0)\rangle_{\mathbb{R}}$ & 1/2 & $-$\\
    $\{x_1 = x_3\}$ & $\langle(1, 0, 1, 1)\rangle_{\mathbb{R}}$ & 1/2 & $-$\\
    $\{x_1 = -x_3\}$ & $\langle(1, 2, -1, 3)\rangle_{\mathbb{R}}$ & 1/4 & $\{1/4, 3/4\}$\\
    $\{x_1 = x_4\}$ & $\langle(1, 0, 1, 1)\rangle_{\mathbb{R}}$ & 1/2 & $-$\\
    $\{x_1 = -x_4\}$ & $\langle(-1, 2, -3, 1)\rangle_{\mathbb{R}}$ & 1/4 &  $\{1/4, 3/4\}$\\
    $\{x_2 = x_3\}$ & $\langle (2, 1, 1, 3) \rangle_{\mathbb{R}}$ & 1/4 & $\{1/4, 3/4\}$\\
    $\{x_2 = -x_3\}$ & $\langle (0, 1, -1, 1) \rangle_{\mathbb{R}}$ & 1/2 & $-$\\
    $\{x_2 = x_4\}$ & $\langle (0, 1, -1, 1) \rangle_{\mathbb{R}}$ & 1/2 & $-$\\
    $\{x_2 = -x_4\}$ & $\langle (2, -1, 3, 1) \rangle_{\mathbb{R}}$ & 1/4 & $ \{1/4, 3/4\}$\\
    $\{x_3 = x_4\}$ & $\cup_{\ell = 0}^1 \langle (1, 0, 1, 1) \rangle_{\mathbb{R}} + (\ell/2, \ell/2, 0, 0)$ & 1/4 & $ 1: \{1/4, 3/4\}$\\
    $\{x_3 = -x_4\}$ & $\cup_{\ell = 0}^1 \langle (0, 1, -1, 1) \rangle_{\mathbb{R}} + (\ell/2, 0, \ell/2, \ell/2)$ & 1/4 & $ 1: \{1/4, 3/4\}$\\
    \bottomrule
  \end{tabular}
\end{figure}



Let us write out some of the details for approximating $\tau=1/4$ in $U_{3, 4, +, 1}$; the other cases are similar to this case or to what we demonstrated in the previous subsection.
The isomorphism $\psi = \psi_{3, 4, +}$ is given by $\psi(1, 0, 1, 1) = (1, 0)$ and $\psi(1, 1, 0, 2) = (0, 1)$, and $\psi(T) = \langle (A, B) \rangle_\mathbb{R}$. Following  Proposition~\ref{prop:cyclic-subgroup-parameters}, we have \[\psi(T_{3, 4, +, 1}) = \left( \frac{a}{2|B|} + \left\langle \frac{1}{|B|} \right\rangle_{\mathbb{Z}} \right) \times \left\{ \frac{1}{2} \right\},\] where $a \equiv A \pmod{2}.$  (Here $a$ is independent of the sign of $B$ since $A \equiv -A \pmod{2}$.)  Lemma~\ref{lem:approximating-by-1/q} tells us that the best approximations of $\tau = 1/4$ by elements of $\frac{a}{2|B|} + \left\langle \frac{1}{|B|} \right\rangle_{\mathbb{Z}}$ 
have errors 
\[\Approx^-\left(\frac{1}{4}, \frac{a}{2}; |B|\right) = \frac{R^-}{4 |B|}, \quad R^- \equiv |B| - 2A \pmod{4},\] \[\Approx^+\left(\frac{1}{4}, \frac{a}{2}; |B|\right) = \frac{R^+}{4 |B|}, \quad R^+ \equiv 2A - |B| \pmod{4},\] 
where $R^+,R^- \in \{0,1,2,3\}$. 
Finally, to compute the slopes $\lambda^+$ and $\lambda^-$, we write $D_{3,4,+,1}(t) = D(t(1, 0, 1, 1) + (1/2, 1/2, 0, 0))$ and compute
\[D_{3,4,+,1}(1/4 + \varepsilon) = D((3/4 + \varepsilon, 1/2, 1/4 + \varepsilon, 1/4 + \varepsilon)) = 1/4 + \max\{-\varepsilon, \varepsilon\}\]
for small $\varepsilon$.
Thus $\lambda^+=\lambda^- = 1$.  See Figure~\ref{tab:approx-1/4-1011-1102} for these calculations for all quadruples $(i,j,\epsilon,\ell)$.

\begin{figure}
  \centering
  \caption{Approximations to $\tau = 1/4$ for $U^2$. In the fifth row, $a \equiv A \pmod{2}$, and in the sixth row, $a \equiv B \pmod{2}$. The values of $R^+,R^-$ lie in $\{0,1,2,3\}$.  The headings for the last four columns indicate the sign of $\delta$ in parentheses.}
  \label{tab:approx-1/4-1011-1102}
  \begin{tabular}{llllllll}
    \toprule
    $i, j, \epsilon, \ell$ & $\psi(T_{i, j, \epsilon, \ell})$ & $\lambda^+$ & $\lambda^-$ & $R^+$ ($+$) & $R^-$ ($+$)& $R^+$ ($-$) & $R^-$ ($-$)\\
    \midrule \vspace{1mm}
    $1, 3, -, 0$ & $\langle \frac{1}{|2A+B|} \rangle_\mathbb{Z} \times \{0\}$ & 3 & 1 & $-2A-B$ & $2A+B$ & $2A+B$ & $-2A-B$\\ \vspace{1mm}
    $1, 4, -, 0$ & $\langle \frac{1}{|2A+3B|} \rangle_\mathbb{Z} \times \{0\}$ & 3 & 1 & $2A+B$ & $2A+3B$ & $2A+3B$ & $2A+B$\\ \vspace{1mm}
    $2, 3, +, 0$ & $\langle \frac{1}{|B-A|} \rangle_\mathbb{Z} \times \{0\}$ & 3 & 1 & $A-B$ & $B-A$ & $B-A$ & $A-B$\\ \vspace{1mm}
    $2, 4, -, 0$ & $\langle \frac{1}{|A+3B|} \rangle_\mathbb{Z} \times \{0\}$ & 3 & 1 & $B+3A$ & $A+3B$ & $A+3B$ & $B+3A$\\ \vspace{1mm}
    $3, 4, +, 1$ & $(\frac{a}{2|B|} + \langle \frac{1}{|B|} \rangle_\mathbb{Z}) \times \{\frac{1}{2}\}$ & 1 & 1 & $2A-B$ & $B-2A$ & $B-2A$ & $2A-B$\\ \vspace{1mm}
    $3, 4, -, 1$ & $(\frac{a}{2|A + B|} + \langle \frac{1}{|A+B|} \rangle_\mathbb{Z}) \times \{\frac{1}{2}\}$ & 1 & 1 & $B-A$ & $A-B$ & $A-B$ & $B-A$\\
    \bottomrule
  \end{tabular}
\end{figure}

For each $(i,j,\epsilon,\ell)$, we partition $\mathbb{R}_{\geq 0} \times \mathbb{R}$ into two sectors, corresponding to the sign of the expression in the absolute value in the second column of Figure~\ref{tab:approx-1/4-1011-1102}; we use $\delta_{i,j,\epsilon,\ell} \in \{+,-\}$ to denote this sign.  The common refinement of these partitions gives a partition of $\mathbb{R}_{\geq 0} \times \mathbb{R}$ into several sectors, each of which must be analyzed separately.

For example, consider the sector $\{(x,y):0 \leq x \leq y\}$, corresponding to the case where all of the $\delta_{i,j,\epsilon,\ell}$'s are $+$.  For each choice of residues $(A,B) \pmod{4}$ and each quadruple $(i,j,\epsilon,\ell)$, we have
$$D(T_{i,j,\epsilon,\ell})=1/4+\min \{\lambda^+ \Approx^+,\lambda^- \Approx^-\},$$
and this quantity equals $1/4+\gamma/q$, where $q=|T_{i,j,\epsilon,\ell}|$ is a linear combination of $A,B$. Figure~\ref{fig:six_tables} shows the outcome of this calculation for all six quadruples $(i,j,\epsilon,\ell)$ and the various choices of residues $(A,B) \pmod{4}$ (all in the sector $\{(x,y): 0 \leq x \leq y\}$).

For each choice of residues $(A,B) \pmod{4}$, we now compare the quantities $D(T_{i,j,\epsilon,\ell})$, which will make us further partition our sector $\{(x,y)):0 \leq x \leq y\}$.  Here we will illustrate this procedure in the case $(A,B) \equiv (1,3) \pmod{4}$.
We wish to find the minimum among the quantities \[\frac{1}{2A + B}, \quad \frac{3}{2A + 3B}, \quad \frac{2}{B - A}, \quad \frac{2}{A + 3B}, \quad \frac{1}{B}, \quad \frac{2}{A + B}.\] 
Since $0 \leq A \leq B$, the first of these quantities is always smaller than or equal to the second, third, fifth, and sixth.  It remains to compare $1/(2A+B)$ and $2/(A+3B)$: The former is smaller when $B<3A$, and the latter is smaller when $B>3A$, so we accordingly use the slope $3$-ray through the origin to subdivide the sector $\{(x,y): 0 \leq x \leq y\}$ into two sectors.  See Figure~\ref{tab:sectors-best-approx-u2} for the outcomes of these calculations for other choices of the $\delta_{i,j,\epsilon,\ell}$'s and pairs of residues $(A,B) \pmod{4}$, and see Figure~\ref{fig:sector-decomp-picture} for a visual representation of our running example $(A,B) \equiv (1,3) \pmod{4}$.

\begin{figure}[!htb]
    \centering
    \caption{The values of $4\gamma$ in the sector $\{0 \leq A \leq B\}$ for $U^2$.  Each grid represents a quadruple $(i,j,\epsilon,\ell)$.  For each, we show how $\gamma$ depends on $(A,B) \pmod{4}$; note that we have excluded the cases where $A,B$ are both even, due to our standing assumption that $A,B$ are coprime.  For the reader's convenience, we have reiterated the value of $q=|T_{i,j,\epsilon,\ell}|$ in each case.
    }
    \label{fig:six_tables}
    \begin{subfigure}[b]{0.4\textwidth}
        \centering
        \begin{tabular}{|c|c|c|c|c|}
        \hline
        \backslashbox{$B$}{$A$} & \textbf{0} & \textbf{1} & \textbf{2} & \textbf{3} \\ \hline
        \textbf{0} & - & 2 & - & 2\\
        \hline
        \textbf{1} & 1 & 3 & 1 & 3\\
        \hline
        \textbf{2} & - & 0 & - & 0\\
        \hline
        \textbf{3} & 3 & 1 & 3 & 1\\
        \hline
        \end{tabular}
        \caption{$(1, 3, -, 0)$ and $q = 2A + B$.}
        \label{tab:table1}
    \end{subfigure}
    \quad
    \begin{subfigure}[b]{0.4\textwidth}
        \centering
        \begin{tabular}{|c|c|c|c|c|}
        \hline
        \backslashbox{$B$}{$A$} & \textbf{0} & \textbf{1} & \textbf{2} & \textbf{3} \\
        \hline
        \textbf{0} & - & 2 & - & 2 \\
        \hline
        \textbf{1} & 3 & 1 & 3 & 1 \\
        \hline
        \textbf{2} & - & 0 & - & 0 \\
        \hline
        \textbf{3} & 1 & 3 & 1 & 3 \\
        \hline
        \end{tabular}
        \caption{$(1, 4, -, 0)$ and $q = 2A + 3B$.}
    \end{subfigure}
    \\
    \begin{subfigure}[b]{0.4\textwidth}
        \centering
        \begin{tabular}{|c|c|c|c|c|}
        \hline
        \backslashbox{$B$}{$A$} & \textbf{0} & \textbf{1} & \textbf{2} & \textbf{3} \\
        \hline
        \textbf{0} & - & 3 & - & 1  \\
        \hline
        \textbf{1} & 1 & 0 & 3 & 2 \\
        \hline
        \textbf{2} & - & 1 & - & 3 \\
        \hline
        \textbf{3} & 3 & 2 & 1 & 0 \\
        \hline
        \end{tabular}
        \caption{$(2, 3, +, 0)$ and $q = B-A$.}
    \end{subfigure}
    \quad
    \begin{subfigure}[b]{0.4\textwidth}
        \centering
        \begin{tabular}{|c|c|c|c|c|}
        \hline
        \backslashbox{$B$}{$A$} & \textbf{0} & \textbf{1} & \textbf{2} & \textbf{3} \\
        \hline
        \textbf{0} & - & 1 & - & 3  \\
        \hline
        \textbf{1} & 3 & 0 & 1 & 2 \\
        \hline
        \textbf{2} & - & 3 & - & 1 \\
        \hline
        \textbf{3} & 1 & 2 & 3 & 0 \\
        \hline
        \end{tabular}
        \caption{$(2, 4, -, 0)$ and $q = A + 3B$.}
    \end{subfigure}
    \\
    \begin{subfigure}[b]{0.4\textwidth}
        \centering
        \begin{tabular}{|c|c|c|c|c|}
        \hline
        \backslashbox{$B$}{$A$} & \textbf{0} & \textbf{1} & \textbf{2} & \textbf{3} \\
        \hline
        \textbf{0} & - & 2 & - & 2  \\
        \hline
        \textbf{1} & 1 & 1 & 1 & 1 \\
        \hline
        \textbf{2} & - & 0 & - & 0 \\
        \hline
        \textbf{3} & 1 & 1 & 1 & 1 \\
        \hline
        \end{tabular}
        \caption{$(3, 4, +, 1)$ and $q = B$.}
    \end{subfigure}
    \quad
    \begin{subfigure}[b]{0.4\textwidth}
        \centering
        \begin{tabular}{|c|c|c|c|c|}
        \hline
        \backslashbox{$B$}{$A$} & \textbf{0} & \textbf{1} & \textbf{2} & \textbf{3} \\
        \hline
        \textbf{0} & - & 1 & - & 1  \\
        \hline
        \textbf{1} & 1 & 0 & 1 & 2 \\
        \hline
        \textbf{2} & - & 1 & - & 1 \\
        \hline
        \textbf{3} & 1 & 2 & 1 & 0 \\
        \hline
        \end{tabular}
        \caption{$(3, 4, -, 1)$ and $q = A + B$.}
    \end{subfigure}
\end{figure}


\begin{figure}
  \centering
  \caption{$D$-values for sectors for $U^2$.  Each row corresponds to a pair of residues $(A,B) \pmod{4}$, as recorded in the first column.  The second column gives the slopes of the dividing rays for the corresponding sector decomposition of $\mathbb{R}_{\geq 0} \times \mathbb{R}$. In the third column we record (proceeding through the sectors clockwise) the corresponding values of the quantity $4(D(T)-1/4)$, which we term the ``offset'' for brevity.  The last column shows the values achieved by the offsets for coprime $A,B$. 
 }
  \label{tab:sectors-best-approx-u2}
  \begin{tabular}{llll}
    \toprule
    $(A, B)$ & Slopes of dividing rays & Offsets & Scaled progressions\\
    \midrule
    $(0, 1)$ & $-3/4$ & $\frac{1}{2A + B}, \frac{1}{-A - 3B}$ & $\frac{1}{4t + 1}, \frac{1}{4t + 1}$\\
    $(0, 3)$ & $-1/4$ & $\frac{1}{2A + 3B}, \frac{1}{A - B}$ & $\frac{1}{4t + 1}, \frac{1}{4t + 1}$\\
    $(1, 0)$ & $0, -4$ & $\frac{1}{A + 3B}, \frac{1}{A - B}, \frac{2}{-2A - 3B}$ & $\frac{1}{4t + 1}, \frac{1}{4t + 1}, \frac{1}{2t + 1}$\\
    $(1, 1)$ & (none) & $0$ & $-$\\
    $(1, 2)$ & (none) & $0$ & $-$\\
    $(1, 3)$ & $3, -1$ & $\frac{2}{A + 3B}, \frac{1}{2A + B}, \frac{1}{-2A - 3B}$ & $\frac{1}{2t + 1}, \frac{1}{4t + 1}, \frac{1}{4t + 1}$\\
    $(2, 1)$ & $1/2, -1/2, -3/2$& $\frac{1}{A + 3B}, \frac{1}{2A + B}, \frac{1}{A - B}, \frac{1}{-2A - 3B}$ & $\frac{1}{4t + 1}, \frac{1}{4t + 1}, \frac{1}{4t + 1}, \frac{1}{4t + 1}$\\
    $(2, 3)$ & $-1/2$ & $\frac{1}{2A + 3B}, \frac{1}{-A - 3B}$ & $\frac{1}{4t + 1}, \frac{1}{4t + 1}$\\
    $(3, 0)$ & $0, -4/7$ & $\frac{2}{2A + 3B}, \frac{2}{2A + B}, \frac{1}{-A - 3B}$ & $\frac{1}{2t + 1}, \frac{1}{2t + 1}, \frac{1}{4t + 1}$\\
    $(3, 1)$ & $-3/7, -1$& $\frac{1}{2A + 3B}, \frac{2}{A - B}, \frac{2}{-A - 3B}$ & $\frac{1}{4t + 1}, \frac{1}{2t + 1}, \frac{1}{2t + 1}$\\
    $(3, 2)$ & (none) & $0$ & $-$\\
    $(3, 3)$ & (none) & $0$ & $-$\\
    \bottomrule
  \end{tabular}
\end{figure}

\begin{figure}
    \centering
    \begin{tikzpicture}
    \fill[gray!10] (0,0) -- (0,3) -- (1,3) -- cycle;
    \fill[blue!10] (0,0) -- (2.5,-2.5) -- (0,-2.5) -- cycle;

    \draw[thick,->] (-0.5,0) -- (4,0) node[right] {$A$}; 
    \draw[thick,<->] (0,-2.5) -- (0,3) node[above] {$B$}; 

    \draw[thick] (0,0) -- (1,3) node[below right] {$B = 3A$}; 
    \draw[thick] (0,0) -- (2.5,-2.5) node[above right] {$B = -A$};
    
    \node at (0.45,2.7) {$\frac{1}{2t + 1}$};
    \node at (2,0.7) {$\frac{1}{4t + 1}$};
    \node at (0.8,-1.5) {$\frac{1}{4t + 1}$}; 
    \end{tikzpicture}
    \caption{This figure illustrates the sector decomposition for $U^2$ when $(A, B) \equiv (1, 3) \pmod{4}$. Each sector is labeled with the progression of its corresponding offset (see Figure \ref{tab:sectors-best-approx-u2}).}
    \label{fig:sector-decomp-picture}
\end{figure}

For each entry in the third column of Figure~\ref{tab:sectors-best-approx-u2}, we must determine which values are assumed by the denominator as $(A,B)$ ranges over pairs of coprime integers satisfying the modular conditions from the first column and the linear inequalities from the second column.  To illustrate this step, consider the sector $\{(x,y): 0 \leq 3x \leq y\}$ with $(A,B) \equiv (1,3) \pmod{4}$, which corresponds to the $D$-value offset $2/(A+3B)$.  Let us parameterize $A,B$ as $A=4r+1$, $B=4s+3$, so that $A+3B=4r+12s+10$.  This last expression is always equivalent to $2$ modulo $4$, so we wish to determine the set of integers $t$ such that there are valid choices of $r,s$ (in the sense that $A,B$ are coprime and satisfy $B \geq 3A \geq 0$) with $$A+3B=4r+12s+10=4t+2,$$ i.e., $r+3s+2=t$.  Such pairs $(r,s)$ can in turn be parameterized as
$$(r,s)=(t-2+3x,-x)$$
for $x$ an integer in a suitable interval (of length growing with $t$), and this leads to
$$(A,B)=(4t-7+12x,-4x+3).$$
Lemma~\ref{lem:coprime} tells us that since $\gcd(4t-7,12,-4,3)=1$, there are many choices of $x$ making $(A,B)$ coprime (with $B \geq 3A \geq 0$) when $t$ is large.  In particular, every sufficiently large value of $t$ is achieved (even if one insists on throwing out exceptional half-lines, which is unnecessary in light of Section~\ref{sec:sectors-and-lines}), and the expression $2/(A+3B)$ assumes the values \[2/(4t+2)=1/(2t+1).\]  The last column of Figure~\ref{tab:sectors-best-approx-u2} shows the outcome of the analogous calculations for the other entries in the third column of the table.  Since $\{4t+1: t \in \mathbb{N}\} \subseteq \{2t+1: t \in \mathbb{N}\}$, we conclude that $\mathcal{S}_1(U^2)$ has finite symmetric difference with the single progression $1/4 + 1/4\Prog(2, 3)$.  This proves Proposition~\ref{prop:U_2-calculation}.

Let us return to the progression \eqref{eq:alec-alex} found by Fan and Sun~\cite{computer}.  This family of $T$'s corresponds to the case $A+4s+3$, $B=8$ in the setting currently under consideration.  
In the $(A,B) \equiv (3,0) \pmod{4}$ row of Figure~\ref{tab:sectors-best-approx-u2}, these choices of $(A,B)$ fall in the first sector (and do not coincide with an exceptional half-line).  Since $A,B$ are coprime, we see that indeed $$D(T)=\frac{1}{4}+\frac{2}{4(2A+3B)}=\frac{1}{4}+\frac{2}{4(2(4s+3)+3(8))}=\frac{1}{4}+\frac{1}{16s+60};$$
this provides a systematic way of carrying out the calculation of \eqref{eq:alec-alex}. 

\section{The spectrum $\mathcal{S}_1(3)$ to the second accumulation point}\label{sec:explicit-comp-s_1(3)}
We will now prove Theorem~\ref{thm:S_1(3)-second-acc-point}, which states that $\mathcal{S}_1(3) \cap (1/10, 1/6]$ has finite symmetric difference with the union of $1/10 + 4/5\Prog(5, 7)$ and $1/10 + 3/5\Prog(5, 9)$.  Recall that \cite{thesis} determined $\mathcal{S}_1(3)$ up to the first accumulation point, viz., $\mathcal{S}_1(3) \cap (1/6, 1/2]=1/6 + 1/6\Prog(6, 7)$.  Theorem~\ref{thm:S_1(3)-second-acc-point} determines $\mathcal{S}_1(3)$ up to the second accumulation point ($1/10$), up to finitely many exceptions.  We reiterate (see the comment at the beginning of Section~\ref{sec:u-to-check-1/10}) that our method here would let us compute $\mathcal{S}_1(3) \cap (\varepsilon,1/2]$, up to finitely many exceptions, for any $\varepsilon>0$: We can compute $\mathcal{S}_1(3)$ up to as many accumulation points as we want, and these accumulation points approach $0$.

In Section~\ref{sec:u-to-check-1/10}, we showed that, up to symmetry, the only $2$-dimensional proper subtori $U \subseteq (\mathbb{R}/\mathbb{Z})^3$ with $D(U) = 1/10$ are \begin{align*}
    U^3= \langle (0, 1, 4), (1, 0, 0) \rangle_\mathbb{R}, &\quad U^4= \langle (0, 2, 3), (1, 0, 0) \rangle_\mathbb{R},\\ 
    U^5= \langle (0, 1, 3), (1, 0, 1) \rangle_\mathbb{R}, &\quad  U^6= \langle (0, 1, 2), (1, 1, 0) \rangle_\mathbb{R}.
\end{align*} 
Notice that each of these choices of generators satisfies the condition \eqref{eq:lattice-condition}.

All that remains for the proof of Theorem~\ref{thm:S_1(3)-second-acc-point} is computing $S_1(U^i)$ for $i \in \{3, 4, 5, 6\}$ individually.  
The argument is similar to the argument in Section~\ref{sec:explicit-comp-s1(4)}. In each of the first two subsections, we will be able to quickly dispose of all but finitely many exceptional lines; in each of the last two subsections, we will deal with full sector decompositions of $\mathbb{R}_{\geq 0} \times \mathbb{R}$. One topically new (but unimportant) feature is that we have different values of $\tau$ for different tuples $(i, j, \epsilon,\ell)$. We will abbreviate parts of the exposition that are similar to Section~\ref{sec:explicit-comp-s1(4)}. 

\subsection{Computing $\mathcal{S}_1(U^3)$}

\begin{figure}
  \centering
  \caption{Intersections of $U^3$ with the subspaces $\{x_i = \epsilon x_j\}$.}
  \label{tab:intersection-planes-u1-s13}
  \begin{tabular}{llll}
    \toprule
    Subspace & $U_{i, j, \epsilon}$ & $D(U_{i, j, \epsilon})$ & $\tau$ achieving $D(U_{i, j, \epsilon})$\\
    \midrule
    $\{x_1 = x_2\}$ & $\langle(1, 1, 4)\rangle_{\mathbb{R}}$ & 1/10 & $\{2/5, 3/5\}$ \\
    $\{x_1 = -x_2\}$ & $ \langle(-1, 1, 4)\rangle_{\mathbb{R}}$ & 1/10 & $\{2/5, 3/5\}$\\
    $\{x_1 = x_3\}$ & $\langle (4, 1, 4) \rangle_{\mathbb{R}}$ & 1/10 & $\{2/5, 3/5\}$\\
    $\{x_1 = -x_3\}$ & $\langle (-4, 1, 4) \rangle_{\mathbb{R}}$ & 1/10 & $\{2/5, 3/5\}$\\
    $\{x_2 = x_3\}$ & $\cup_{\ell = 0}^2 \langle (1, 0, 0) \rangle_{\mathbb{R}} + (0, \ell/3, \ell/3)$ & 1/6 & $-$\\
    $\{x_2 = -x_3\}$ & $\cup_{\ell = 0}^4 \langle (1, 0, 0) \rangle_{\mathbb{R}} + (0, \ell/5, 4\ell/5)$ & 1/10 & $2, 3: [2/5, 3/5]$\\
    \bottomrule
  \end{tabular}
\end{figure}

We parametrize $1$-dimensional subtori of $U^3$ as \[T = \langle A(0, 1, 4) + B(1, 0, 0) \rangle_\mathbb{R}\] for coprime integers $A$, $B$ with $B \geq 0$. We compute the intersections $U_{i, j, \epsilon} = U^3 \cap \{x_i = \epsilon x_j\}$ in Figure \ref{tab:intersection-planes-u1-s13}. Note that on $U_{2,3,-}$, the $D$-value $1/10$ is achieved on intervals of positive length; we will use this to restrict our attention to a small number of exceptional half-lines.

Consider the isomorphism $\psi = \psi_{2, 3, -}$ given by  $\psi((1, 0, 0)) = (1, 0)$ and $\psi((0, 1, 4)) = (0, 1)$.  Then Proposition~\ref{prop:cyclic-subgroup-parameters} gives \[\psi(T_{2, 3, -}) = \bigcup_{\ell = 0}^{4} \left(\left(\frac{a \ell}{(5|A|)} + \left\langle\frac{1}{|A|}\right\rangle_{\mathbb{Z}}\right) \times \left\{\frac{\ell}{5} \right\}\right),\] where $\delta$ is the sign of $A$ and $a \equiv \delta B \pmod{5}$.  One can check $D_{2,3,-}$ achieves the constant values $1/10$ on the set $[2/5, 3/5] \times \{2/5, 3/5\}$. If $|A| \geq 5$, then $\psi(T_{2, 3, -})$ certainly intersects $[2/5, 3/5] \times \{2/5, 3/5\}$, which immediately implies that $D(T) = 1/10$. It remains only to consider $|A|\leq 4$. 

For each value of \(m:= |A|\), there are several equivalence classes of \( B \) modulo \( 5m \) that give $D(T)=1/10$ and hence can be disregarded. For $m=1$, for instance, this occurs exactly when $2a/5 \in [2/5,3/5]$ (or, equivalently, $3a/5 \in [2/5,3/5]$), i.e., $a=\delta B \equiv 1,4 \pmod{5}$; so it remains only to consider $B \equiv 0,2,3 \pmod{5}$.  Likewise, for $m=2$, we can ignore the cases where $$2a/10=a/5 \in [2/5/3/5] \cup (1/2+[2/5,3/5]);$$ after also accounting for the condition that $A,B$ are coprime (i.e., $B$ is odd), we are left with checking only $B \equiv 1,9 \pmod{10}$.  Similarly, for $m=3$, we need to consider only $B \equiv 5,10 \pmod{15}$, and we can completely ignore the case $m=4$.


Each of the remaining cases from the previous paragraph corresponds to an exceptional line; as \( B \) increases, the parameter \( q = |A| \) remains fixed. By symmetry (see Section~\ref{sec:other-symmetries}), it suffices to study the behavior around the point \( \tau = 2/5 \) for each tuple \( (i, j, \epsilon) \neq (2, 3, -) \) with \( D(U_{i, j, \epsilon}) = 1/10 \); see Figure \ref{tab:approx-2/5-100-014}. All of our exceptional lines eventually lie within the sector containing the positive \( B \)-axis, which corresponds to the case where every \( \delta_{i, j, \epsilon} = + \).

\begin{figure}
  \centering
  \caption{Approximations to $\tau = 2/5$ for $U^3$. The values of $R^+,R^-$ lie in $\{0,1,2,3,4\}$. We treat only the sector where every $\delta = +$, as this is the region containing all of the exceptional half-lines of interest.}
  \label{tab:approx-2/5-100-014}
  \begin{tabular}{llllll}
    \toprule
    $i, j, \epsilon, \ell$ & $\psi(T_{i, j, \epsilon, \ell})$ & $\lambda^+$ & $\lambda^-$ & $R^+$ ($+$) & $R^-$ ($+$)\\
    \midrule \vspace{1mm}
    $1, 2, +, 0$ & $\langle \frac{1}{|B-A|} \rangle_{\mathbb{Z}} \times \{0\}$ & 4 & 1 & $3(B-A)$ & $2(B-A)$\\ \vspace{1mm}
    $1, 2, -, 0$ & $\langle \frac{1}{|B+A|} \rangle_{\mathbb{Z}} \times \{0\}$ & 4 & 1 & $3(B+A)$ & $2(B+A)$\\ \vspace{1mm}
    $1, 3, +, 0$ & $\langle \frac{1}{|B-4A|} \rangle_{\mathbb{Z}} \times \{0\}$ & 4 & 1 & $3(B-4A)$ & $2(B-4A)$\\ \vspace{1mm}
    $1, 3, -, 0$ & $\langle \frac{1}{|B+4A|} \rangle_{\mathbb{Z}} \times \{0\}$ & 4 & 1 & $3(B+4A)$ & $2(B+4A)$\\
    \bottomrule
  \end{tabular}
\end{figure}

For each \( (m, x) \in \{(1, 0), (1, 2), (1, 3), (2, 1), (2, 9), (3, 5), (3, 10)\} \), we consider pairs $(A,B)$ with \( |A| = m \) and \( B \equiv x \pmod{5m} \); after parametrizing $B=5|A|t+x$ (for $t \in \mathbb{N}$), we compute $D(T)$ as a function of $t$ in the usual way.  This results in several progressions in $\mathcal{S}_1(U^3)$, which we verify are contained in the progressions from the statement of Theorem~\ref{thm:S_1(3)-second-acc-point}.  Figure~\ref{tab:prog-sectors-100-014} presents the results of these calculations for the half-lines with \( A > 0 \); the half-lines with $A<0$ turn out to yield the same progressions, so we have omitted these calculations.


\begin{figure}
  \centering
  \caption{$D$-value progressions for $U^3$ with fixed $1 \leq A \leq 3$ and fixed residue class $B \equiv x \pmod{5A}$ ($x \in \{1,2,\ldots, 5A\}$). We parametrize $B=5At + x$ ($t \in \mathbb{N}$). The second column gives the corresponding progression $5(D(T)-1/10)$. These values (as $t$ ranges) form subprogressions of the two progressions appearing in the statement of Theorem~\ref{thm:S_1(3)-second-acc-point}; the last column provides the changes of variable witnessing these containments.}
\label{tab:prog-sectors-100-014}
  \begin{tabular}{llll}
    \toprule
    $(A, B)$ & Scaled progression & Theorem progression & Change of variable \\
    \midrule
    $A = 1, B \equiv 0 \pmod{5}$ & $2/(5t + 1)$ & $4/(5s + 2)$ & $s = 2t$ \\
    $A = 1, B \equiv 2 \pmod{5}$ & $1/(5t + 3)$ & $3/(5s + 4)$ & $s = 3t + 1$ \\
    $A = 1, B \equiv 3 \pmod{5}$ & $3/(5t + 4)$ & $3/(5s + 4)$ & $s = t$ \\
    $A = 2, B \equiv 1 \pmod{10}$ & $1/(10t + 3)$ & $3/(5s + 4)$ & $s = 6t + 1$ \\ 
    $A = 2, B \equiv 9 \pmod{10}$ & $2/(10t + 11)$ & $4/(5s + 2)$ & $s = 4t + 4$\\
    $A = 3, B \equiv 5 \pmod{15}$ & $1/(15t + 8)$ & $4/(5s + 2)$ & $s = 12t + 6$\\
    $A = 3, B \equiv 10 \pmod{15}$ & $1/(15t + 13)$ & $4/(5s + 2)$ & $s = 12t + 10$\\
    \bottomrule
  \end{tabular}
\end{figure}

\subsection{Computing $\mathcal{S}_1(U^4)$}
\begin{figure}
  \centering
  \caption{Intersections of $U^4$ with the subspaces $\{x_i = \epsilon x_j\}$.}
  \label{tab:intersection-planes-u4-s13}
  \begin{tabular}{llll}
    \toprule
    Subspace & \textbf{$U_{i, j \epsilon}$} & $D(U_{i, j, \epsilon})$ & $\tau$ achieving $1/10$\\
    \midrule
    $\{x_1 = x_2\}$ & $\langle(2, 2, 3)\rangle_{\mathbb{R}}$ & 1/10 & $\{1/5, 4/5\}$\\
    $\{x_1 = -x_2\}$ & $ \langle(-2, 2, 3)\rangle_{\mathbb{R}}$ & 1/10 & $\{1/5, 4/5\}$\\
    $\{x_1 = x_3\}$ & $\langle (3, 2, 3) \rangle_{\mathbb{R}}$ & 1/10 & $\{1/5, 4/5\}$\\
    $\{x_1 = -x_3\}$ & $\langle (-3, 2, 3) \rangle_{\mathbb{R}}$ & 1/10 & $\{1/5, 4/5\}$\\
    $\{x_2 = x_3\}$ & $\langle (1, 0, 0) \rangle_{\mathbb{R}}$ & 1/2 & $-$\\
    $\{x_2 = -x_3\}$ & $\cup_{i = 0}^4 \langle (1, 0, 0) \rangle_{\mathbb{R}} + (0, 2i/5, 3i/5)$ & 1/10 & $1, 4: [2/5, 3/5]$\\
    \bottomrule
  \end{tabular}
\end{figure}

We parametrize $1$-dimensional subtori of $U^4$ as \[T = \langle A(0, 2, 3) + B(1, 0, 0) \rangle_\mathbb{R}\] for coprime $A$, $B$ with $B \geq 0$. We compute the intersections $U_{i, j, \epsilon} = U^4 \cap \{x_i = \epsilon x_j\}$ in Figure \ref{tab:intersection-planes-u4-s13}. Since $U_{2, 3, -}$ achieves the $D$-value $1/10$ on intervals of positive length, we restrict our attention to a small number of exceptional half-lines, as in the previous subsection.  This analysis leaves us with the cases $|A|=1$ and $B \equiv 0,1,4 \pmod{5}$; $|A|=2$ and $B \equiv 3,7 \pmod{10}$; and $|A|=3$ and $B \equiv 5,10 \pmod{15}$.  For each case, we approximate the critical value $\tau=1/5$ for the remaining tuples $(i,j,\epsilon)$; see Figure~\ref{tab:approx-1/5-100-023}.


\begin{figure}
  \centering
  \caption{Approximations to $\tau = 1/5$ for $U^4$. The values of $R^+,R^-$ lie in $\{0,1,2,3,4\}$. We consider the case where every $\delta = +$ as that corresponds to the region of interest.}
  \label{tab:approx-1/5-100-023}
  \begin{tabular}{llllll}
    \toprule
    $i, j, \epsilon, \ell$ & $\psi(T_{i, j, \epsilon, \ell})$ & $\lambda^+$ & $\lambda^-$ & $R^+$ ($+$) & $R^-$ ($+$)\\
    \midrule \vspace{1mm}
    $1, 2, +, 0$ & $\langle \frac{1}{|B-2A|} \rangle_{\mathbb{Z}} \times \{0\}$ & 3 & 2 & $2A-B$ & $B-2A$\\ \vspace{1mm}
    $1, 2, -, 0$ & $\langle \frac{1}{|B+2A|} \rangle_{\mathbb{Z}} \times \{0\}$ & 3 & 2 & $-2A-B$ & $2A+B$\\ \vspace{1mm}
    $1, 3, +, 0$ & $\langle \frac{1}{|B-3A|} \rangle_{\mathbb{Z}} \times \{0\}$ & 3 & 2 & $3A - B$ & $B-3A$\\ \vspace{1mm}
    $1, 3, -, 0$ & $\langle \frac{1}{|B+3A|} \rangle_{\mathbb{Z}} \times \{0\}$ & 3 & 2 & $-3A-B$ & $3A+B$\\
    \bottomrule
  \end{tabular}
\end{figure}

Each of our exceptional half-lines is eventually contained in the sector containing the positive $B$-axis, which corresponds to the case where every $\delta=+$.  Each exceptional line generates a progression, again contained in one of the progressions from the statement of Theorem \ref{thm:S_1(3)-second-acc-point}; see Figure~\ref{tab:prog-sectors-100-023} for $A>0$ (as in the previous subsection, $A<0$ gives the same results and these calculations are omitted).

\begin{figure}
  \centering
  \caption{$D$-value progressions for $U^4$ with fixed $1 \leq A \leq 3$ and fixed residue class $B \equiv x \pmod{5A}$ ($x \in \{1,2,\ldots, 5A\}$). We parametrize $B=5At + x$ ($t \in \mathbb{N}$). The second column gives the corresponding progression $5(D(T)-1/10)$. These values (as $t$ ranges) form subprogressions of the two progressions appearing in the statement of Theorem~\ref{thm:S_1(3)-second-acc-point}.}
\label{tab:prog-sectors-100-023}
  \begin{tabular}{lllll}
    \toprule
    $(A, B)$ & Scaled progression & Theorem progression & Change of variable\\
    \midrule
    $A = 1, B \equiv 0 \pmod{5}$ & $4/(5t + 2)$ & $4/(5s + 2)$ & $s = t$\\ 
    $A = 1, B \equiv 1 \pmod{5}$ & $3/(5t + 4)$ & $3/(5s + 4)$ & $s = t$\\ 
    $A = 1, B \equiv 4 \pmod{5}$ & $2/(5t + 6)$ & $4/(5s + 2)$ & $s = 2t+2$\\ 
    $A = 2, B \equiv 3 \pmod{10}$ & $3/(10t + 9)$ & $3/(5s + 4)$ & $s = 2t + 1$\\ 
    $A = 2, B \equiv 7 \pmod{10}$ & $2/(10t + 11)$ & $4/(5s + 2)$ & $s = 4t + 4$\\ 
    $A = 3, B \equiv 5 \pmod{15}$ & $2/(15t + 11)$ & $4/(5s + 2)$ & $s = 6t + 4$\\ 
    $A = 3, B \equiv 10 \pmod{15}$ & $2/(15t + 16)$ & $4/(5s + 2)$ & $s = 6t + 6$\\ 
    \bottomrule
  \end{tabular}
\end{figure}

\subsection{Computing $\mathcal{S}_1(U^5)$}
For $U^5$ we will have to work with sector decompositions instead of only exceptional half-lines.  We parametrize $1$-dimensional subtori of $U^5$ as \[T = \langle A(0, 1, 3) + B(1, 0, 1) \rangle_\mathbb{R}\] for coprime $A$, $B$ with $A \geq 0$. We compute the intersections $U_{i, j, \epsilon} = U^5 \cap \{x_i = \epsilon x_j\}$ and their corresponding $D$-values in Figure~\ref{tab:intersection-planes-u5-s13}. In Figure~\ref{tab:approx-2/5-101-013}, we show the calculations for approximating the critical points $\tau$. Finally, for each pair of residues for $(A,B) \pmod{5}$, we partition $\mathbb{R}_{\geq 0} \times \mathbb{R}$ into sectors, calculate $D(T)$ on each sector, and give the corresponding progressions; see Figure~\ref{tab:prog-sectors-101-013}. 

Figure~\ref{tab:prog-sectors-101-013} tells us that the expression $5(D(T) - 1/10)$ assumes values in the progressions \[\frac{1}{5t + 3}, \quad \frac{2}{5t + 1}, \quad \frac{3}{5t + 4}, \quad \frac{4}{5t + 2}.\] 
Regarding the first progression, we observe that $\{1/(5t + 3): t\in \mathbb{N}\} \subseteq \{4/(5s + 2): s \in \mathbb{N}\}$ (via $s =  4t + 2$) and $\{1/(5t + 3): t\in \mathbb{N}\} \subseteq \{3/(5s + 4): s \in \mathbb{N}\}$ (via $s =  3t + 1$). In fact, $\{1/(5t + 3): t\in \mathbb{N}\}$ is precisely these two progressions containing it. Likewise $\{2/(5t+1): t \in \mathbb{N}\} \subseteq \{4/(5s+2): s \in \mathbb{N}\}$ (via $s = 2t$). Thus, we obtain only the two progressions from the statement of Theorem \ref{thm:S_1(3)-second-acc-point}. 
\begin{figure}[]
  \centering
  \caption{Intersections of $U^5$ with subspaces $\{x_i = \epsilon x_j\}$.}
  \label{tab:intersection-planes-u5-s13}
  \begin{tabular}{llll}
    \toprule
    Subspace & $U_{i, j, \epsilon}$ & $D(U_{i, j, \epsilon})$ & $\tau$ achieving $1/10$\\
    \midrule
    $\{x_1 = x_2\}$ & $\langle(1, 1, 0)\rangle_{\mathbb{R}}$ & 1/2 & $-$\\
    $\{x_1 = -x_2\}$ & $ \langle(-1, 1, -4)\rangle_{\mathbb{R}}$ & 1/10 & $\{2/5, 3/5\}$\\
    $\{x_1 = x_3\}$ & $\langle (2, 1, 2) \rangle_{\mathbb{R}}$ & 1/6 & $-$ \\
    $\{x_1 = -x_3\}$ & $\langle (-2, -3, 2) \rangle_{\mathbb{R}}$ & 1/10 & $\{1/5, 4/5\}$\\
    $\{x_2 = x_3\}$ & $\langle (3, 2, 2) \rangle_{\mathbb{R}}$ & 1/10 & $\{1/5, 4/5\}$\\
    $\{x_2 = -x_3\}$ & $\langle (-1, -2, 2) \rangle_{\mathbb{R}}$ & 1/6 & $-$\\
    \bottomrule
  \end{tabular}
\end{figure} 

\begin{figure}[]
  \centering
  \caption{Approximations to $\tau=1/5,2/5$ for $U^5$. The values of $R^+,R^-$ lie in $\{0,1,2,3,4\}$.}
  \label{tab:approx-2/5-101-013}
  \begin{tabular}{lllllllll}
    \toprule
    $i, j, \epsilon$ & $\tau$ & $\psi(T_{i, j, \epsilon})$ & $\lambda^+$ & $\lambda^-$ & $R^+$ ($+$) & $R^-$ ($+$)& $R^+$ ($-$) & $R^-$ ($-$)\\
    \midrule \vspace{1mm}
    $1, 2, -$ & $2/5$ & $\langle \frac{1}{|A+2B|} \rangle_{\mathbb{Z}}$ & 4 & 1 & $3A+B$ & $2A+4B$ & $2A+4B$ & $3A+B$\\ \vspace{1mm}
    $1, 3, -$ & $1/5$ & $\langle \frac{1}{|3A+B|} \rangle_{\mathbb{Z}}$ & 3 & 2 & $2A+4B$ & $3A+B$ & $3A+B$ & $2A+4B$\\ \vspace{1mm}
    $2, 3, +$ & $1/5$ & $\langle \frac{1}{|B-2A|} \rangle_{\mathbb{Z}}$ & 3 & 2 & $2A+4B$ & $3A+B$ & $3A+B$ & $2A+4B$\\
    \bottomrule
  \end{tabular}
\end{figure}

\begin{figure}[htbp!]
    \centering
    \caption{$D$-values in sectors for $U^5$. The third column records the offsets $5(D(T)-1/10)$.}
    \label{tab:prog-sectors-101-013}
\begin{tabular}{llll}
    \toprule
    $(A, B)$ & Slopes of dividing rays & Offsets & Scaled progressions\\
    \midrule
    $(0, 1)$ & $-1$ & $\frac{2}{3A + B}, \frac{1}{-A - 2B}$ & $\frac{2}{5t + 1}, \frac{1}{5t + 3}$ \\
    $(0, 2)$ & $1, -1$ & $\frac{3}{A + 2B}, \frac{4}{3A + B}, \frac{2}{-A - 2B}$ & $\frac{3}{5t + 4}, \frac{4}{5t + 2}, \frac{2}{5t + 1}$\\
    $(0, 3)$ & $0, -2$ & $\frac{2}{A + 2B}, \frac{4}{2A - B}, \frac{3}{-A - 2B}$ & $\frac{2}{5t + 1}, \frac{4}{5t + 2}, \frac{3}{5t + 4}$\\
    $(0, 4)$ & $0$ & $\frac{1}{A + 2B}, \frac{2}{2A - B}$ & $\frac{1}{5t + 3}, \frac{2}{5t + 1}$\\
    $(1, 0)$ & $0, -2$ & $\frac{2}{A + 2B}, \frac{4}{2A - B}, \frac{3}{-A - 2B}$ & $\frac{2}{5t + 1}, \frac{4}{5t + 2}, \frac{3}{5t + 4}$\\
    $(1, 1)$ & $0$ & $\frac{1}{A + 2B}, \frac{2}{2A - B}$ & $\frac{1}{5t + 3}, \frac{2}{5t + 1}$\\
    $(1, 2)$ & (none) & $0$ & $-$\\
    $(1, 3)$ & $-1$ & $\frac{2}{3A + B}, \frac{1}{-A-2B}$ & $ \frac{2}{5t + 1}, \frac{1}{5t + 3}$\\
    $(1, 4)$ & $0, -1$ & $\frac{3}{A + 2B}, \frac{4}{2A - B}, \frac{2}{-A-2B}$ & $\frac{3}{5t + 4}, \frac{4}{5t + 2}, \frac{2}{5t + 1}$\\
    $(2, 0)$ & $-1$ & $\frac{2}{3A + B}, \frac{1}{-A-2B}$ & $\frac{2}{5t + 1}, \frac{1}{5t + 3}$\\
    $(2, 1)$ & $1, -1$ & $\frac{3}{A + 2B}, \frac{4}{3A + B}, \frac{2}{-A-2B}$ & $\frac{3}{5t + 4}, \frac{4}{5t + 2}, \frac{2}{5t + 1}$\\
    $(2, 2)$ & $0, -2$ & $\frac{2}{A + 2B}, \frac{4}{2A - B}, \frac{3}{-A-2B}$ & $\frac{2}{5t + 1}, \frac{4}{5t + 2}, \frac{3}{5t + 4}$\\
    $(2, 3)$ & $0$ & $\frac{1}{A + 2B}, \frac{2}{2A - B}$ & $\frac{1}{5t + 3}, \frac{2}{5t + 1}$\\
    $(2, 4)$ & (none) & $0$& $-$\\
    $(3, 0)$ & $0$ & $\frac{1}{A + 2B}, \frac{2}{2A - B}$ & $\frac{1}{5t + 3}, \frac{2}{5t + 1}$\\
    $(3, 1)$ & (none) & $0$ & $-$\\
    $(3, 2)$ & $-1$ & $\frac{2}{3A + B}, \frac{1}{-A-2B}$ & $\frac{2}{5t + 1}, \frac{1}{5t + 3}$\\
    $(3, 3)$ & $1, -1$ & $\frac{3}{A + 2B}, \frac{4}{3A + B}, \frac{2}{-A-2B}$ & $\frac{3}{5t + 4}, \frac{4}{5t + 2}, \frac{2}{5t + 1}$\\
    $(3, 4)$ & $0, -2$ & $\frac{2}{A + 2B}, \frac{4}{2A - B}, \frac{3}{-A-2B}$ & $\frac{2}{5t + 1}, \frac{4}{5t + 2}, \frac{3}{5t + 4}$\\
    $(4, 0)$ & $1, -1$ & $\frac{3}{A + 2B}, \frac{4}{3A + B}, \frac{2}{-A-2B}$ & $\frac{3}{5t + 4}, \frac{4}{5t + 2}, \frac{2}{5t + 1}$\\
    $(4, 1)$ & $0, -2$ & $\frac{2}{A + 2B}, \frac{4}{2A - B}, \frac{3}{-A-2B}$ & $\frac{2}{5t + 1}, \frac{4}{5t + 2}, \frac{3}{5t + 4}$\\
    $(4, 2)$ & $0$ & $\frac{1}{A + 2B}, \frac{2}{2A - B}$ & $\frac{1}{5t + 3}, \frac{2}{5t + 1}$\\
    $(4, 3)$ & (none) & $0$ & $-$\\
    $(4, 4)$ & $-1$ & $\frac{2}{3A + B}, \frac{1}{-A-2B}$ & $\frac{2}{5t + 1}, \frac{1}{5t + 3}$\\
    \bottomrule
  \end{tabular}
\end{figure}

\subsection{Computing $\mathcal{S}_1(U^6)$}
We proceed as in the previous subsection. We parametrize $1$-dimensional subtori of $U^6$ as \[T = \langle A(0, 1, 2) + B(1, 1, 0) \rangle_\mathbb{R}\] for coprime $A$, $B$ with $A \geq 0$. As before, we intersect $U^6$ with the subspaces $\{x_i = \epsilon x_j\}$; see Figure~\ref{tab:intersection-planes-u6-s13}. 
Then we approximate critical values of $\tau$; see Figure~\ref{tab:approx-2/5-110-012}. 
Finally, we partition $\mathbb{R}_{\geq 0} \times \mathbb{R}$ into sectors for each pair of equivalence classes for $(A, B)$, calculate the corresponding values of $D(T)$, and represent these values as progressions; see Figure~\ref{tab:prog-sectors-110-012}. As before we get exactly the two progressions from the statement of Theorem \ref{thm:S_1(3)-second-acc-point}.

\begin{figure}
  \centering
  \caption{Intersections of $U^6$ with the subspaces $\{x_i = \epsilon x_j\}$.}
  \label{tab:intersection-planes-u6-s13}
  \begin{tabular}{llll}
    \toprule
    Subspace & $U_{i, j, \epsilon}$ & $D(U_{i, j, \epsilon})$ & $\tau$ achieving $1/10$\\
    \midrule
    $\{x_1 = x_2\}$ & $\langle(1, 1, 4)\rangle_{\mathbb{R}}$ & 1/10 & $\{2/5, 3/5\}$\\
    $\{x_1 = -x_2\}$ & $ \langle(1, -1, -2)\rangle_{\mathbb{R}}$ & 1/6 & $-$\\
    $\{x_1 = x_3\}$ & $\langle (1, 0, 1) \rangle_{\mathbb{R}}$ & 1/2 & $-$ \\
    $\{x_1 = -x_3\}$ & $\langle (-3, 2, 3) \rangle_{\mathbb{R}}$ & 1/10 & $\{1/5, 4/5\}$\\
    $\{x_2 = x_3\}$ & $\langle (-2, 1, 1) \rangle_{\mathbb{R}}$ & 1/6 & $-$\\
    $\{x_2 = -x_3\}$ & $\langle (-4, 1, -1) \rangle_{\mathbb{R}}$ & 1/10 & $\{2/5, 3/5\}$\\
    \bottomrule
  \end{tabular}
\end{figure}

\begin{figure}
  \centering
  \caption{Approximations to $\tau=1/5,2/5$ for $U^6$. The values of $R^+,R^-$ lie in $\{0,1,2,3,4\}$.}
  \label{tab:approx-2/5-110-012}
  \begin{tabular}{lllllllll}
    \toprule
    $i, j, \epsilon$ & $\tau$ & $\psi(T_{i, j, \epsilon})$ & $\lambda^+$ & $\lambda^-$ & $R^+$ ($+$) & $R^-$ ($+$)& $R^+$ ($-$) & $R^-$ ($-$)\\
    \midrule \vspace{1mm}
    $1, 2, +$ & $2/5$ & $\langle \frac{1}{|B-A|} \rangle_{\mathbb{Z}}$ & 4 & 1 & $2A+3B$ & $3A+2B$ & $3A+2B$ & $2A+3B$\\ \vspace{1mm}
    $1, 3, -$ & $1/5$ & $\langle \frac{1}{|3A+2B|} \rangle_{\mathbb{Z}}$ & 3 & 2 & $2A+3B$ & $3A+2B$ & $3A+2B$ & $2A+3B$\\ \vspace{1mm}
    $2, 3, -$ & $2/5$ & $\langle \frac{1}{|4A+B|} \rangle_{\mathbb{Z}}$ & 4 & 1 & $2A+3B$ & $3A+2B$ & $3A+2B$ & $2A+3B$\\
    \bottomrule
  \end{tabular}
\end{figure}

\begin{figure}
    \centering
    \caption{$D$-values in sectors for $U^6$.}
    \label{tab:prog-sectors-110-012}
\begin{tabular}{llll}
    \toprule
    $(A, B)$ & Slopes of dividing rays & Offsets & Scaled progressions\\
    \midrule
    $(0, 1)$ & $-2$ & $\frac{2}{4A + B}$, $\frac{3}{A - B}$ &  $\frac{2}{5t + 1}$, $\frac{3}{5t + 4}$\\
    $(0, 2)$ & $0$ & $\frac{3}{3A + 2B}$, $\frac{1}{A - B}$ &  $\frac{3}{5t + 4}$, $\frac{1}{5t + 3}$\\
    $(0, 3)$ & $-3$ & $\frac{1}{4A + B}$, $\frac{3}{-3A - 2B}$ & $\frac{1}{5t + 3}$, $\frac{3}{5t + 4}$\\
    $(0, 4)$ & $-1$ & $\frac{3}{4A + B}$, $\frac{2}{A - B}$ & $\frac{3}{5t + 4}$, $\frac{2}{5t + 1}$\\
    $(1, 0)$ & $-1$ & $\frac{3}{4A + B}$, $\frac{2}{A - B}$ &  $\frac{3}{5t + 4}$, $\frac{2}{5t + 1}$\\
    $(1, 1)$ & (none) & 0 & $-$ \\
    $(1, 2)$ & $-2$ & $\frac{2}{4A + B}$, $\frac{3}{A - B}$ & $\frac{2}{5t + 1}$, $\frac{3}{5t + 4}$ \\
    $(1, 3)$ & $0$ & $\frac{3}{3A + 2B}$, $\frac{1}{A - B}$ & $\frac{3}{5t + 4}$, $\frac{1}{5t + 3}$ \\
    $(1, 4)$ & $-3$ & $\frac{1}{4A + B}$, $\frac{3}{-3A - 2B}$ & $\frac{1}{5t + 3}$, $\frac{3}{5t + 4}$ \\
    $(2, 0)$ & $-3$ & $\frac{1}{4A + B}$, $\frac{3}{-3A - 2B}$ & $\frac{1}{5t + 3}$, $\frac{3}{5t + 4}$ \\
    $(2, 1)$ & $-1$ & $\frac{3}{4A + B}$, $\frac{2}{A - B}$ & $\frac{3}{5t + 4}$, $\frac{2}{5t + 1}$ \\
    $(2, 2)$ & (none) & 0 & $-$ \\
    $(2, 3)$ & $-2$ & $\frac{2}{4A + B}$, $\frac{3}{A - B}$ & $\frac{2}{5t + 1}$, $\frac{3}{5t + 4}$\\
    $(2, 4)$ & $0$ & $\frac{3}{3A + 2B}$, $\frac{1}{A - B}$ & $\frac{3}{5t + 4}$, $\frac{1}{5t + 3}$ \\
    $(3, 0)$ & $0$ & $\frac{3}{3A + 2B}$, $\frac{1}{A - B}$ & $\frac{3}{5t + 4}$, $\frac{1}{5t + 3}$ \\
    $(3, 1)$ & $-3$ & $\frac{1}{4A + B}$, $\frac{3}{-3A - 2B}$ & $\frac{1}{5t + 3}$, $\frac{3}{5t + 4}$ \\
    $(3, 2)$ & $-1$ & $\frac{3}{4A + B}$, $\frac{2}{A - B}$ & $\frac{3}{5t + 4}$, $\frac{2}{5t + 1}$ \\
    $(3, 3)$ & (none) & 0 & $-$ \\
    $(3, 4)$ & $-2$ & $\frac{2}{4A + B}$, $\frac{3}{A - B}$ & $\frac{2}{5t + 1}$, $\frac{3}{5t + 4}$\\
    $(4, 0)$ & $-2$ & $\frac{2}{4A + B}$, $\frac{3}{A - B}$ & $\frac{2}{5t + 1}$, $\frac{3}{5t + 4}$\\
    $(4, 1)$ & $0$ & $\frac{3}{3A + 2B}$, $\frac{1}{A - B}$ & $\frac{3}{5t + 4}$, $\frac{1}{5t + 3}$ \\
    $(4, 2)$ & $-3$ & $\frac{1}{4A + B}$, $\frac{3}{-3A - 2B}$ & $\frac{1}{5t + 3}$, $\frac{3}{5t + 4}$ \\
    $(4, 3)$ & $-1$ & $\frac{3}{4A + B}$, $\frac{2}{A - B}$ & $\frac{3}{5t + 4}$, $\frac{2}{5t + 1}$ \\
    $(4, 4)$ & (none) & 0 & $-$ \\
    \bottomrule
  \end{tabular}
\end{figure}

\section{A new progression in the spectrum $\mathcal{S}_1(6)$}\label{sec:S_1(6)-prog}

\begin{figure}
  \centering
  \caption{Intersections of $U^7$ with the subspaces $\{x_i = \epsilon x_j\}$.}
  \label{tab:intersection-planes-u-s16}
  \begin{tabular}{llll}
    \toprule
    Subspace & $U_{i, j, \epsilon}$ & $D(U_{i, j, \epsilon})$ & $\tau$ achieving $1/3$\\
    \midrule
    $\{x_1 = x_2\}$ & $\langle(1, 1, 2, 3, 4, 5)\rangle_{\mathbb{R}}$ & 1/3 & $\{1/6, 5/6\}$ \\
    $\{x_1 = -x_2\}$ & $ \langle(1, -1, 0, 1, 2, 1)\rangle_{\mathbb{R}}$ & 1/2 & $-$\\
    $\{x_1 = x_3\}$ & $\langle (1, 0, 1, 2, 3, 3) \rangle_{\mathbb{R}}$ & 1/2 & $-$\\
    $\{x_1 = -x_3\}$ & $\langle (1, -2, -1, 0, 1, -1) \rangle_{\mathbb{R}}$ & 1/2 & $-$\\
    $\{x_1 = x_4\}$ & $\langle (1, -1, 0, 1, 2, 1) \rangle_{\mathbb{R}}$ & 1/2 & $-$\\
    $\{x_1 = -x_4\}$ & $\langle (1, -3, -2, -1, 0, -3) \rangle_{\mathbb{R}}$ & 1/2 & $-$\\
    $\{x_1 = x_5\}$ & $\langle (1, -2, -1, 0, 1, -1) \rangle_{\mathbb{R}}$ & 1/2 & $-$\\
    $\{x_1 = -x_5\}$ & $\langle (1, -4, -3, -2, -1, -5) \rangle_{\mathbb{R}}$ & 1/3 & $\{1/6, 5/6\}$\\
    $\{x_1 = x_6\}$ & $\cup_{\ell = 0}^1 \langle (1, -1, 0, 1, 2, 1) \rangle_{\mathbb{R}} + (0, \frac{\ell}{2}, \frac{\ell}{2}, \frac{\ell}{2}, \frac{\ell}{2}, 0)$ & 1/3 & $1: \{1/6, 1/3, 2/3, 5/6\}$\\
    $\{x_1 = -x_6\}$ & $\cup_{\ell = 0}^1 \langle (1, -2, -1, 0, 1, -1) \rangle_{\mathbb{R}} + (0, \frac{\ell}{2}, \frac{\ell}{2}, \frac{\ell}{2}, \frac{\ell}{2}, 0)$ & 1/3 & $1: \{1/6, 1/3, 2/3, 5/6\}$\\
    $\{x_2 = x_3\}$ & $\langle (0, 1, 1, 1, 1, 2) \rangle_{\mathbb{R}}$ & 1/2 & $-$\\
    $\{x_2 = -x_3\}$ & $\langle (2, -1, 1, 3, 5, 4) \rangle_{\mathbb{R}}$ & 1/3 & $\{1/6, 5/6\}$\\
    $\{x_2 = x_4\}$ & $\cup_{\ell = 0}^1 \langle (0, 1, 1, 1, 1, 2) \rangle_{\mathbb{R}} + (\frac{\ell}{2}, 0, \frac{\ell}{2}, 0, \frac{\ell}{2}, \frac{\ell}{2})$ & 1/3 & $1: \{1/6, 1/3, 2/3, 5/6\}$\\
    $\{x_2 = -x_4\}$ & $\cup_{\ell = 0}^1 \langle (1, -1, 0, 1, 2, 1) \rangle_{\mathbb{R}} + (0, \frac{\ell}{2}, \frac{\ell}{2}, \frac{\ell}{2}, \frac{\ell}{2}, 0)$ & 1/3 & $1: \{1/6, 1/3, 2/3, 5/6\}$\\
    $\{x_2 = x_5\}$ & $\cup_{\ell = 0}^2 \langle (0, 1, 1, 1, 1, 2) \rangle_{\mathbb{R}} + (\frac{\ell}{3}, 0, \frac{\ell}{3}, \frac{2\ell}{3}, 0, 0)$ & 1/3 & $1, 2: \{1/6, 5/6\}$\\
    $\{x_2 = -x_5\}$ & $\langle (2, -3, -1, 1, 3, 0) \rangle_{\mathbb{R}}$ & 1/2 & $-$\\
    $\{x_2 = x_6\}$ & $\langle (-1, 3, 2, 1, 0, 3) \rangle_{\mathbb{R}}$ & 1/2 & $-$\\
    $\{x_2 = -x_6\}$ & $\cup_{\ell = 0}^2 \langle (1, -1, 0, 1, 2, 1) \rangle_{\mathbb{R}} + (0, \frac{\ell}{3}, \frac{\ell}{3}, \frac{\ell}{3}, \frac{\ell}{3}, \frac{2\ell}{3})$ & 1/3 &  $1: \{1/6, 1/2\};$\\ 
    & & & $2: \{1/2, 5/6\}$\\
    $\{x_3 = x_4\}$ & $\langle (0, 1, 1, 1, 1, 2) \rangle_{\mathbb{R}}$ & 1/2 & $-$\\
    $\{x_3 = -x_4\}$ & $\langle (2, -3, -1, 1, 3, 0) \rangle_{\mathbb{R}}$ & 1/2 & $-$\\
    $\{x_3 = x_5\}$ & $\cup_{\ell = 0}^1 \langle (0, 1, 1, 1, 1, 2) \rangle_{\mathbb{R}} + (\frac{\ell}{2}, 0, \frac{\ell}{2}, 0, \frac{\ell}{2}, \frac{\ell}{2})$ & 1/3 & $1: \{1/6, 1/3, 2/3, 5/6\}$\\
    $\{x_3 = -x_5\}$ & $\cup_{\ell = 0}^1 \langle (1, -2, -1, 0, 1, -1) \rangle_{\mathbb{R}} + (0, \frac{\ell}{2}, \frac{\ell}{2}, \frac{\ell}{2}, \frac{\ell}{2}, 0)$ & 1/3 & $1: \{1/6, 1/3, 2/3, 5/6\}$\\
    $\{x_3 = x_6\}$ & $\langle (1, -2, -1, 0, 1, -1) \rangle_{\mathbb{R}}$ & 1/2 & $-$\\
    $\{x_3 = -x_6\}$ & $\langle (3, -4, -1, 2, 5, 1) \rangle_{\mathbb{R}}$ & 1/3 & $\{1/6, 5/6\}$\\
    $\{x_4 = x_5\}$ & $\langle (0, 1, 1, 1, 1, 2) \rangle_{\mathbb{R}}$ & 1/2 & $-$\\
    $\{x_4 = -x_5\}$ & $\langle (2, -5, -3, -1, 1, -4) \rangle_{\mathbb{R}}$ & 1/3 & $\{1/6, 5/6\}$\\
    $\{x_4 = x_6\}$ & $\langle (1, -1, 0, 1, 2, 1) \rangle_{\mathbb{R}}$ & 1/2 & $-$\\
    $\{x_4 = -x_6\}$ & $\langle (3, -5, -2, 1, 4, -1) \rangle_{\mathbb{R}}$ & 1/3 & $\{1/6, 5/6\}$\\
    $\{x_5 = x_6\}$ & $\langle (1, 0, 1, 2, 3, 3) \rangle_{\mathbb{R}}$ & 1/2 & $-$\\
    $\{x_5 = -x_6\}$ & $\cup_{\ell = 0}^2 \langle (1, -2, -1, 0, 1, -1) \rangle_{\mathbb{R}} + (0, \frac{\ell}{3}, \frac{\ell}{3}, \frac{\ell}{3}, \frac{\ell}{3}, \frac{2\ell}{3})$ & 1/3 & $1: \{1/2, 5/6\};$\\
    & & & $2: \{1/6, 1/2\}$\\
    \bottomrule
  \end{tabular}
\end{figure}

\begin{figure}
  \centering
  \caption{Approximations to the various values of $\tau$ for $U^7$.  Throughout, we consider the sector in which the half-line $\{6\} \times \mathbb{R}_{>0}$ is (eventually) contained; this sector has $\delta=-$ for the first row and $\delta=+$ for all other rows. In rows three, four, eight, and nine, $a \equiv B \pmod{2}$; in rows five and six, $a \equiv A \pmod{2}$; in rows ten and eleven, $a \equiv \delta B \pmod{3}$; in rows twelve and thirteen $a \equiv -\delta B \pmod{3}$; and in rows seventeen and eighteen, $a \equiv \delta A \pmod{3}$.}
  \label{tab:approx-s1(6)}
  \begin{tabular}{llllllll}
    \toprule
    $i, j, \epsilon,\ell$ & $\tau$ & $\psi(T_{i, j, \epsilon,\ell})$ & $\lambda^+$ & $\lambda^-$ & $R^+$ & $R^-$& $D(T_{i, j, \epsilon, \ell}) - \frac{1}{3}$\\
    \midrule \vspace{1mm}
    $1, 2, +, 0$ & $1/6$ & $\langle \frac{1}{|A-B|} \rangle_{\mathbb{Z}} \times \{0\}$ & 5 & 1 & $A+5B \leadsto 1$ & $5A+B \leadsto 5$ & $\frac{5}{6(B-A)}$\\ \vspace{1mm}
    $1, 5, -, 0$ & $1/6$ & $\langle \frac{1}{|4A+B|} \rangle_{\mathbb{Z}} \times \{0\}$ & 5 & 1 & $2A+5B \leadsto 1$ & $4A+B \leadsto 5$ & $\frac{5}{6(4A+B)}$\\ \vspace{1mm}
    $1, 6, +, 1$ & $1/6$ & $\left(\frac{a}{2|A+B|} + \langle \frac{1}{|A+B|} \rangle_{\mathbb{Z}}\right) \times \{\frac{1}{2}\}$ & 2 & 1 & $5A+2B \leadsto 4$ & $A+4B \leadsto 2$ & $\frac{2}{6(A+B)}$\\ \vspace{1mm}
    $1, 6, +, 1$ & $1/3$ & $\left(\frac{a}{2|A+B|} + \langle \frac{1}{|A+B|} \rangle_{\mathbb{Z}}\right) \times \{\frac{1}{2}\}$ & 1 & 2 & $4A+B \leadsto 5$ & $2A+5B \leadsto 1$ & $\frac{2}{6(A+B)}$\\ \vspace{1mm}
    $1, 6, -, 1$ & $1/6$ & $\left(\frac{a}{2|2A+B|} + \langle \frac{1}{|2A+B|} \rangle_{\mathbb{Z}}\right) \times \{\frac{1}{2}\}$ & 2 & 1 & $A+5B \leadsto 1$ & $5A+B \leadsto 5$ & $\frac{2}{6(2A+B)}$\\ \vspace{1mm}
    $1, 6, -, 1$ & $1/3$ & $\left(\frac{a}{2|2A+B|} + \langle \frac{1}{|2A+B|} \rangle_{\mathbb{Z}}\right) \times \{\frac{1}{2}\}$ & 1 & 2 & $5A+4B \leadsto 2$ & $A+2B \leadsto 4$ & $\frac{2}{6(2A+B)}$\\ \vspace{1mm}
    $2, 3, -, 0$ & $1/6$ & $\langle \frac{1}{|A+2B|} \rangle_{\mathbb{Z}} \times \{0\}$ & 5 & 1 & $5A+4B \leadsto 2$ & $A+2B \leadsto 4$ & $\frac{4}{6(A+2B)}$\\ \vspace{1mm}
    $2, 4, +, 1$ & $1/6$ & $\left(\frac{a}{2|A|} + \langle \frac{1}{|A|} \rangle_{\mathbb{Z}}\right) \times \{\frac{1}{2}\}$ & 2 & 1 & $5A+3B \leadsto 3$ & $A+3B \leadsto 3$ & $\frac{3}{6A}$\\\vspace{1mm}
    $2, 4, +, 1$ & $1/3$ & $\left(\frac{a}{2|A|} + \langle \frac{1}{|A|} \rangle_{\mathbb{Z}}\right) \times \{\frac{1}{2}\}$ & 1 & 2 & $4A+3B \leadsto 3$ & $2A+3B \leadsto 3$ & $\frac{3}{6A}$\\\vspace{1mm}
    $2, 5, +, 1$ & $1/6$ & $\left(\frac{a}{3|A|} + \langle \frac{1}{|A|} \rangle_{\mathbb{Z}}\right) \times \{\frac{1}{3}\}$ & 1 & 1 & $5A+2B \leadsto 4$ & $A+4B \leadsto 2$ & $\frac{2}{6A}$\\\vspace{1mm}
    $2, 5, +, 1$ & $5/6$ & $\left(\frac{a}{3|A|} + \langle \frac{1}{|A|} \rangle_{\mathbb{Z}}\right) \times \{\frac{1}{3}\}$ & 1 & 1 & $A+2B \leadsto 4$ & $5A+4B \leadsto 2$ & $\frac{2}{6A}$\\\vspace{1mm}
    $2, 6, -, 1$ & $1/6$ & $\left(\frac{a}{3|A+B|} + \langle \frac{1}{|A+B|} \rangle_{\mathbb{Z}}\right) \times \{\frac{1}{3}\}$ & 1 & 1 & $5A+3B \leadsto 3$ & $A+3B \leadsto 3$ & $\frac{3}{6(A+B)}$\\\vspace{1mm}
    $2, 6, -, 1$ & $1/2$ & $\left(\frac{a}{3|A+B|} + \langle \frac{1}{|A+B|} \rangle_{\mathbb{Z}}\right) \times \{\frac{1}{3}\}$ & 1 & 1 & $3A+B \leadsto 5$ & $3A+5B \leadsto 1$ & $\frac{1}{6(3A+B)}$\\\vspace{1mm}
    $3, 6, -, 0$ & $1/6$ & $\langle \frac{1}{|4A+3B|} \rangle_{\mathbb{Z}} \times \{0\}$ & 5 & 1 & $2A+3B \leadsto 3$ & $4A+3B \leadsto 3$ & $\frac{3}{6(4A+3B)}$\\\vspace{1mm}
    $4, 5, -, 0$ & $1/6$ & $\langle \frac{1}{|5A+2B|} \rangle_{\mathbb{Z}} \times \{0\}$ & 5 & 1 & $A+4B \leadsto 2$ & $5A+2B \leadsto 4$ & $\frac{4}{6(5A+2B)}$\\\vspace{1mm}
    $4, 6, -, 0$ & $1/6$ & $\langle \frac{1}{|5A+3B|} \rangle_{\mathbb{Z}} \times \{0\}$ & 5 & 1 & $A+3B \leadsto 3$ & $5A+3B \leadsto 3$ & $\frac{3}{6(5A+3B)}$\\\vspace{1mm}
    $5, 6, -, 1$ & $1/2$ & $\left(\frac{a}{3|2A+B|} + \langle \frac{1}{|2A+B|} \rangle_{\mathbb{Z}}\right) \times \{\frac{1}{3}\}$ & 1 & 1 & $2A+3B \leadsto 3$ & $4A+3B \leadsto 3$ & $\frac{3}{6(2A+B)}$\\
    $5, 6, -, 1$ & $5/6$ & $\left(\frac{a}{3|2A+B|} + \langle \frac{1}{|2A+B|} \rangle_{\mathbb{Z}}\right) \times \{\frac{1}{3}\}$ & 1 & 1 & $4A+B \leadsto 5$ & $2A+5B \leadsto 1$ & $\frac{1}{6(2A+B)}$\\
    \bottomrule
  \end{tabular}
\end{figure}

The goal of this section of to prove Theorem~\ref{thm:S_1(6)-prog}, which states that $1/3+1/6\Prog(6,11) \subseteq \mathcal{S}_1(6)$.  This example is noteworthy because it provides a new progression, beyond the progression $1/3+1/6\Prog(6,7)$ predicted by~\eqref{eq:old-conj}.  The progression $1/3+1/6\Prog(6,11)$ appears in the relative spectrum of
$$U^7:=\langle (1, 0, 1, 2, 3, 3), (0, 1, 1, 1, 1, 2) \rangle_{\mathbb{R}}.$$
We remark that this choice of generators bears some resemblance to the generators of $U^2 \subseteq (\mathbb{R}/\mathbb{Z})^4$, which (after permuting the coordinates) can be written as $(1,0,1,1),(0,1,1,2)$ (but we were not able to find an ``analogous'' torus in $(\mathbb{R}/\mathbb{Z})^8$).


We will show that $\mathcal{S}_1(U^7)$ contains the set $1/3 + 1/6\Prog(6, 11)$. We parametrize $1$-dimensional subtori of $U^7$ as \[T=\langle A(1, 0, 1, 2, 3, 3)+B(0, 1, 1, 1, 1, 2)\rangle_{\mathbb{R}},\] where $A,B$ are coprime and $A \geq 0$, and we check that this choice of generators satisfies $\eqref{eq:lattice-condition}$.  As before, we compute the intersections $U_{i, j \epsilon} = U^7 \cap \{x_i = \epsilon x_j\}$; see Figure \ref{tab:intersection-planes-u-s16}. 

This computation differs from previous ones in several ways. First, there are several pairs of triples $(i, j, \epsilon)$ and $(i', j', \epsilon')$ such that \(U_{i, j, \epsilon}\) and \(U_{i', j', \epsilon'}\) are identical; obviously $D(T_{i,j,\epsilon})=D(T_{i',j',\epsilon'})$ in such cases, and we will omit these redundant computations.
Second, there are some quadruples $(i,j,\epsilon,\ell)$ with several values of $\tau$ which must be analyzed individually, even after we account for the usual $x \mapsto -x$ symmetries.  This is visible in Figure \ref{tab:approx-s1(6)}, where the triples \((1, 6, +, 1)\), \((1, 6, -, 1)\), and \((2, 4, +, 1)\) each have two rows corresponding to different values of $\tau$.  For the triples $(2, 6, -)$ and $(5, 6, -)$ (which correspond to subgroups $U_{i, j, \epsilon}$ with three components), the values of $\tau$ no longer come in pairs $(\tau,-\tau)$ on each connected component $U_{i,j,\epsilon,\ell}$ (although of course a value $\tau$ on $U_{i,j,\epsilon,\ell}$ still corresponds to a value $-\tau$ on $U_{i,j,\epsilon,-\ell}$).  For triples $(i,j,\epsilon)$ where $U_{i,j,\epsilon}$ has $3$ components, the symmetry discussed in Section~\ref{sec:other-symmetries} guarantees that $D(U_{i,j,\epsilon,1})=D(U_{i,j,\epsilon,2})$, so in such cases we will analyze only $\ell=1$.

Recall that we are concerned only with finding a single progression in $\mathcal{S}_1(U^7)$, not with computing the entire relative spectrum.  So we will restrict our attention to the half-line yielding this progression, namely, the half-line consisting of pairs $(A,B)$ with $A=6$ and $B \equiv 5 \pmod{6}$ (with $B>6$); note that such $A,B$ are always coprime.  We analyze only the sector containing this half-line.
Comparing the expressions in the last column of Figure~\ref{tab:approx-s1(6)}, we find that
\[D(T) = \frac{1}{3} + \frac{1}{6(2A + B)},\] and substituting \(A = 6\) and \(B = 6s + 5\) (for $s \in \mathbb{N}$) gives \[D(T)=\frac{1}{3} + \frac{1}{6(6s + 17)}.\] The choice $(A,B)=(5,1)$ leads to $D(T) = 1/3 + 1/(6\cdot11)$, and the choice $(A,B)=(6,5)$ leads to $D(T) = 1/3 + 1/(6\cdot17)$.  Altogether, these values comprise the progression $1/3+1/6\Prog(6,11)$; this proves Theorem~\ref{thm:S_1(6)-prog}. 

We remark that $\mathcal{S}_1(U^7)$ also contains the progression $1/3+1/6\Prog(6,7)$; this gives a new way to achieve the progression in \eqref{eq:old-conj}.  This progression comes from the sector $\{0 \leq A \leq B\}$ with the modular constraints $(A,B) \equiv (1,0) \pmod{6}$.  A further analysis along the lines of the previous two sections (not shown in full detail here) reveals that $\mathcal{S}_1(U^7)$ does not contain any other progressions.


\section{Finite and infinite relative spectra}\label{sec:finite}
In this final section, we make a few brief remarks about when relative spectra are finite.  Let $n \geq 3$, and let $U \subseteq (\mathbb{R}/\mathbb{Z})^n$ be a $2$-dimensional proper subtorus.  Write $$Z(U):=\{x \in U: D(x)=D(U)\}$$ 
for the locus of points where $U$ achieves its $D$-value. 
Notice that $Z(U)$ is the intersection of $U$ with the hollow cube $\{x \in (\mathbb{R}/\mathbb{Z})^n: D(x)=D(U)\}$.  We can express $Z(U)$ as the union of the intersections of $U$ with the various facets of this hollow cube.  Each such intersection has dimension at most $1$ (otherwise $U$ would be contained in a translate of a coordinate hyperplane), and we conclude that $Z(U)$ consists of a finite union of points and (closed) line segments.  Of course these points and the endpoints of these segments all have rational coordinates.

Our criterion for the finiteness of $\mathcal{S}_1(U)$ concerns whether or not $Z(U)$ contains non-parallel line segments; notice that the ``parallel-ness'' of line segments is the same in $(\mathbb{R}/\mathbb{Z})^n$ and in $U$ (considered as a $2$-dimensional torus).

\begin{proposition}\label{prop:finiteness-criterion}
Let $n \geq 3$ be a natural number, and let $U \subseteq (\mathbb{R}/\mathbb{Z})^n$ be a $2$-dimensional proper subtorus.  The relative spectrum $\mathcal{S}_1(U)$ is finite if and only if the set $Z(U)$ contains non-parallel line segments.
\end{proposition}

\begin{proof}
First, suppose that $Z(U)$ contains two non-parallel line segments.  To show that $\mathcal{S}_1(U)$ is finite, it suffices to show that all but finitely many $1$-dimensional subtori $T \subseteq U$ intersect the union of these line segments and thus satisfy $D(T)=D(U)$.  Fix an identification of $U$ with $(\mathbb{R}/\mathbb{Z})^2$, and consider the images $s_1,s_2$ of the two non-parallel line segments.  Let $\ell$ be the minimum of the lengths of $s_1,s_2$, and let $L_1,L_2$ denote the lines containing $s_1,s_2$ (respectively).  Let $0<\theta\leq \pi/2$ be the measure of the angle formed by $L_1,L_2$.

Now, let $T' \subseteq (\mathbb{R}/\mathbb{Z})^2$ be a $1$-dimensional subtorus.  It forms an angle of at least $\theta/2$ with some $L_i$.  Then the height of $s_i$ in the direction perpendicular to $T'$ is at least
$$\ell \sin(\theta/2).$$
The perpendicular distance between consecutive leaves of $T'$ is $1/\vol(T')$, where $\vol(T')$ denotes the volume (i.e., length) of $T'$.  Thus, if $\vol(T') \geq 1/(\ell \sin(\theta/2))$, then $T'$ must intersect $s_i$, and we conclude that the corresponding $1$-dimensional subtorus $T \subseteq U$ has $D(T)=D(U)$.  Since there are only finitely many $1$-dimensional subtori with volume below any given threshold, it follows that all but finitely many $1$-dimensional subtori $T \subseteq U$ have $D(T)=D(U)$, as desired.

Second, suppose that $Z(U)$ does not contain any non-parallel line segments.  To show that $\mathcal{S}_1(U)$ is infinite, it suffices to show that there are infinitely many $1$-dimensional subtori $T \subseteq U$ that are disjoint from $Z(U)$: Indeed, such $T$'s have $D(T)>D(U)$, and there are infinitely many distinct such values since $D(T) \to D(U)$ as $\vol(T) \to \infty$.  By assumption, $Z(U)$ is contained in a finite union of rational parallel lines and hence is contained in a $1$-dimensional subgroup $\Gamma$ of $U$.

We claim that $U$ contains at least two non-parallel $1$-dimensional subtori that are contained in the coordinate hyperplanes of $(\mathbb{R}/\mathbb{Z})^n$.  Indeed, for each $1 \leq i \leq n$, the intersection $U \cap \{x_i=0\}$ is a $1$-dimensional subgroup, and the identity components of these intersections cannot all be identical because then this common identity component would consist of only the point $0$. Thus, there is a $1$-dimensional subtorus $T^*$ of $U$ that is contained in a coordinate hyperplane of $(\mathbb{R}/\mathbb{Z})^n$ and is not parallel to $\Gamma$.  Fix an identification of $U$ with $(\mathbb{R}/\mathbb{Z})^2$ such that $T^*$ is identified with the $x_2$-axis; let $\Gamma'$ denote the image of $\Gamma$ under this identification.  Since $D(T^*)=1/2>D(U)$, we know that $Z(U)$ is disjoint from $T^*$, and it follows that the image of $Z(U)$ is contained in $\Gamma' \setminus (\{0\} \times (\mathbb{R}/\mathbb{Z}))$.  Hence it suffices to find infinitely many $1$-dimensional subtori $T \subseteq (\mathbb{R}/\mathbb{Z})^2$ whose intersections with $\Gamma'$ are contained in $\{0\} \times (\mathbb{R}/\mathbb{Z})$.

The choice of $T^*$ guarantees that $\Gamma'$ has some finite slope $b/a$, where $a,b$ are coprime integers with $a>0$.  Let $q \in \mathbb{N}$ be such that the leaves of $\Gamma'$ are spaced out $1/q$ apart vertically.  Now, for each $s \in \mathbb{N}$, consider the $1$-dimensional subtorus
$$T_s:=\langle (a,b+a/(qs)) \rangle_{\mathbb{R}}.$$
As we travel along $T_s$ starting at $(0,0)$, the first intersection with $\Gamma'$ occurs at the point
$$(s/a)(a,b+1/(qs))=(s,sb/a+1/q) \in \{0\} \times (\mathbb{R}/\mathbb{Z}),$$
so in fact $T_s \cap \Gamma' \subseteq \{0\} \times (\mathbb{R}/\mathbb{Z})$, as desired.
\end{proof}
We remark that \cite{thesis} established $\mathcal{S}_1(n-1) \subset \acc \mathcal{S}_1(n)$ by showing that $\mathcal{S}_1(U' \times (\mathbb{R}/\mathbb{Z}))$ is infinite for every proper $1$-dimensional subtorus $U' \subseteq (\mathbb{R}/\mathbb{Z})^{n-1}$; the construction in the second part of the proof of Proposition~\ref{prop:finiteness-criterion} is a generalization of this argument.

It is not a priori obvious that the geometric condition for finiteness in Proposition~\ref{prop:finiteness-criterion} is ever satisfied.  In fact, one can show by ad-hoc arguments that it is never satisfied for $n \leq 6$.  For $n=7$, however, there are examples like the following.  Consider the $2$-dimensional proper subtorus $$U^8:=\langle (1,2,3,2,0,0,0),(0,0,0,2,1,2,3) \rangle_{\mathbb{R}}.$$
A tedious but routine calculation gives that $D(U^8)=3/10$, and that $Z(U^8)$ consists of the points $a (1,2,3,2,0,0,0)+b (0,0,0,2,1,2,3)$ where $(a,b)$ either lies on one of the eight line segments
$$\{2/5,3/5\} \times ([1/5,4/15] \cup [11/15,4/5]), \quad ([1/5,4/15] \cup [11/15,4/5]) \times \{2/5,3/5\}.$$
or is one of the four points $(1/5,1/5),(2/5,2/5),(3/5,3/5),(4/5,4/5)$.
It follows from Proposition~\ref{prop:finiteness-criterion} that $\mathcal{S}_1(U^8)$ is finite, i.e., $\mathcal{S}_1(U^8)$ does not contain any progressions.  Using a computer to calculate $D(T)$ for low-volume proper $1$-dimensional subtori $T \subseteq U^8$ reveals that $\mathcal{S}_1(U^8)$ contains (at least) the values $7/22, 11/34, 17/54, 23/74$ in addition to $3/10$.  Hence, for $\mathcal{S}_1(U^8)$, the finite symmetric difference in Theorem~\ref{thm:relative-spectrum-improved} is necessary.

\section*{Acknowledgements}
We thank Vikram Giri and Mayank Pandey for helpful conversations. Both authors were supported in part by NSF Graduate Research Fellowships (grant DGE--2039656).

\end{document}